\documentclass[a4paper,11pt]{amsart}

\makeatletter
\@namedef{subjclassname@2010}{%
  \textup{2010} Mathematics Subject Classification}
\makeatother
\usepackage{a4wide}

\usepackage{eulervm}

\usepackage{amscd}
\usepackage{amsmath}
\usepackage{amssymb}
\usepackage{amsthm}
\usepackage{float}
\usepackage[dvips]{graphicx}
\usepackage{enumerate}

\usepackage{array}
\usepackage{amscd}
\usepackage[all]{xy}
\usepackage{graphicx}
\usepackage{url}
\usepackage{enumitem}

\usepackage{amsfonts}
\usepackage{amssymb}
\usepackage{amsmath}
\usepackage{amsthm}

\usepackage{mathdots}

\usepackage[all]{xy}
\usepackage{graphicx}
\usepackage{mathrsfs}
\usepackage{color}
\usepackage{url}
\usepackage{array}
\usepackage{pifont}
\usepackage{tikz}
\usepackage{caption}
\usepackage{booktabs}
\usepackage{setspace}

\DeclareFontFamily{U}{rsfs}{%
\skewchar\font127}
\DeclareFontShape{U}{rsfs}{m}{n}{%
<-6>rsfs5<6-8.5>rsfs7<8.5->rsfs10}{}
\DeclareSymbolFont{rsfs}{U}{rsfs}{m}{n}
\DeclareSymbolFontAlphabet
{\mathrsfs}{rsfs}
\DeclareRobustCommand*\rsfs{%
\@fontswitch\relax\mathrsfs}

\theoremstyle{plain}
\newtheorem{thm}{Theorem}[section]
\newtheorem{prop}[thm]{Proposition}
\newtheorem{lem}[thm]{Lemma}
\newtheorem{defi}[thm]{Definition}
\newtheorem{rmk}[thm]{Remark}
\newtheorem{note}[thm]{Note}

\newtheorem{nota}[thm]{Notation}
\newtheorem{mukaispecseq}[thm]{Mukai Spectral Sequence}
\newtheorem{dualspecseq}[thm]{``Duality'' Spectral Sequence}
\newtheorem{specseq}[thm]{Spectral Sequence}

\newtheorem{prop-defi}[thm]{Proposition-Definition}
\newtheorem{thm-defi}[thm]{Theorem-Definition}
\newtheorem{lem-defi}[thm]{Lemma-Definition}

\newtheorem{conj}[thm]{Conjecture}
\newtheorem{exam}[thm]{Example}

\newdimen\argwidth
\def\db[#1\db]{
 \setbox0=\hbox{$#1$}\argwidth=\wd0
 \setbox0=\hbox{$\left[\box0\right]$}
  \advance\argwidth by -\wd0
 \left[\
 n.3\argwidth\box0 \kern.3\argwidth\right]}

\newcommand{\Del}{\overline{\Delta}}
\newcommand{\SX}{\scriptscriptstyle  X}
\newcommand{\SY}{\scriptscriptstyle  Y}
\newcommand{\SZ}{\scriptscriptstyle  Z}
\newcommand{\HX}{\widehat{X}}
\newcommand{\hx}{\widehat{x}}
\newcommand{\hy}{\widehat{y}}
\newcommand{\HY}{\widehat{Y}}
\newcommand{\SHX}{\scriptscriptstyle  \widehat{X} }
\newcommand{\SHY}{\scriptscriptstyle  \widehat{Y}}

\newcommand{\HL}{\widehat{L}}
\newcommand{\lx}{\ell_{\scriptscriptstyle{X}} }
\newcommand{\ly}{\ell_{\scriptscriptstyle{Y}}}
\newcommand{\lz}{\ell_{\scriptscriptstyle{Z}}}
\newcommand{\lhx}{\ell_{\scriptscriptstyle  \widehat{X}} }
\newcommand{\lhy}{\ell_{\scriptscriptstyle  \widehat{Y}}}

\newcommand{\SH}{\operatorname{\scriptscriptstyle{H}}}
\newcommand{\calExt}{\operatorname{\mathcal{E}\textit{xt}}}
\newcommand{\calEnd}{\operatorname{\mathcal{E}\textit{nd}}}
\newcommand{\Adiag}{\operatorname{Adiag}}
\newcommand{\alg}{\operatorname{\scriptstyle{alg}}}
\newcommand{\calHom}{\operatorname{\mathcal{H}\textit{om}}}

\newcommand{\BL}{\operatorname{\mathbf{L}}}
\newcommand{\Ga}{\Gamma}
\newcommand{\HGa}{\widehat{\Gamma}}
\newcommand{\HPsi}{\widehat{\Psi}}
\newcommand{\HPhi}{\widehat{\Phi}}

\newcommand{\TPsi}{\widetilde{\Psi}}

\newcommand{\coker}{\operatorname{coker}}

\newcommand{\aA}{\mathcal{A}}
\newcommand{\bB}{\mathcal{B}}
\newcommand{\cC}{\mathcal{C}}
\newcommand{\CC}{\mathbb{C}}
\newcommand{\dD}{\mathcal{D}}
\newcommand{\DD}{\mathbb{D}}
\newcommand{\eE}{\mathcal{E}}
\newcommand{\fF}{\mathcal{F}}
\newcommand{\gG}{\mathcal{G}}
\newcommand{\hH}{\mathcal{H}}
\newcommand{\HH}{\mathbb{H}}

\newcommand{\iI}{\mathcal{I}}

\newcommand{\mM}{\mathcal{M}}

\newcommand{\oO}{\mathcal{O}}
\newcommand{\pP}{\mathcal{P}}

\newcommand{\QQ}{\mathbb{Q}}

\newcommand{\RR}{\mathbb{R}}
\newcommand{\RRR}{\mathbf{R}}
\newcommand{\sS}{\mathcal{S}}
\newcommand{\tT}{\mathcal{T}}

\newcommand{\ZZ}{\mathbb{Z}}

\newcommand{\Supp}{\mathop{\rm Supp}\nolimits}
\newcommand{\Hom}{\mathop{\rm Hom}\nolimits}

\newcommand{\NS}{\mathop{\rm NS}\nolimits}

\newcommand{\Pic}{\mathop{\rm Pic}\nolimits}

\newcommand{\id}{\textrm{id}}

\newcommand{\ch}{\mathop{\rm ch}\nolimits}
\newcommand{\rk}{\mathop{\rm rk}\nolimits}

\newcommand{\Ext}{\mathop{\rm Ext}\nolimits}

\newcommand{\Coh}{\mathop{\rm Coh}\nolimits}

\newcommand{\cneq}{\mathrel{\raise.095ex\hbox{:}\mkern-4.2mu=}}
\newcommand{\eqcn}{\mathrel{=\mkern-4.5mu\raise.095ex\hbox{:}}}

\newcommand{\TOR}{\mathop{{\tT}or}\nolimits}

\newcommand{\End}{\mathop{\rm End}\nolimits}

\newcommand{\Imm}{\operatorname{Im}}

\newcommand{\Ree}{\operatorname{Re}}

\newcommand{\Sing}{\mathop{\rm Sing}\nolimits}

\newcommand{\HN}{\operatorname{HN}}

\begin{document}

\title[Fourier-Mukai Transforms and Stability Conditions]{Stability Conditions under the Fourier-Mukai Transforms on Abelian Threefolds}

\date{\today}

\author{Dulip Piyaratne}


\address{Department of Mathematics \\ The University of Arizona \\ 617 N. Santa Rita Ave. \\ Tucson  \\ AZ 85721 \\ USA.}
\email{piyaratne@math.arizona.edu}
\address{\rm {\itshape URL}:  \url{https://sites.google.com/site/dulippiya/}}

\subjclass[2010]{Primary 14F05; Secondary 14J30, 14J32, 14J60, 14K99, 	18E10, 18E30,  18E40}

\keywords{Fourier-Mukai transforms, Abelian varieties,  Threefolds, Derived category, Bridgeland stability conditions, Bogomolov-Gieseker inequality, Polarization}

\begin{abstract}

We realize explicit symmetries of Bridgeland stability conditions on any abelian threefold given by  Fourier-Mukai transforms.
 In particular, we extend the previous joint work   with  Maciocia 
to study the slope and tilt stabilities of sheaves and complexes under the Fourier-Mukai transforms, and 
then
to show that certain Fourier-Mukai transforms  give equivalences of the stability condition hearts of bounded t-structures
 which are double tilts of coherent sheaves.  
Consequently,  we show that
the conjectural construction proposed by Bayer,  Macr\`i and Toda gives rise to Bridgeland stability conditions on any abelian threefold by proving that tilt stable objects satisfy the  Bogomolov-Gieseker type inequality.
Our proof of the Bogomolov-Gieseker type inequality conjecture for any abelian threefold is a generalization of the 
 previous joint work   with  Maciocia 
for a principally polarized abelian threefold with Picard rank one case, and also
this  gives an alternative proof of the same result in full generality due to 
Bayer,  Macr\`i and Stellari.
Moreover, we realize the induced cohomological Fourier-Mukai  transform explicitly in anti-diagonal form, and 
consequently, we describe a polarization on the derived equivalent abelian variety by using Fourier-Mukai theory. 


\end{abstract}

\maketitle
\tableofcontents
\section{Introduction}
\label{sec:intro}
\subsection{Bridgeland stability conditions on threefolds}
\label{sec:intro-stability}
Motivated by Douglas's work on $\Pi$-stability for D-branes on Calabi-Yau threefolds (see \cite{Dou}), Bridgeland introduced the notion of stability conditions on triangulated categories (see \cite{BriStab}).
Bridgeland's approach can be interpreted essentially as an abstraction of the usual slope stability for sheaves. 
From the original motivation,  construction of Bridgeland stability conditions on the bounded derived category of a given projective threefold is an important problem. However, unlike for a projective surface, there is no known construction which gives stability conditions for all projective threefolds. See \cite{HuyStabNotes, MS} for further details.   

The category of coherent sheaves does not arise as a heart of a Bridgeland stability condition for higher dimensional smooth projective varieties (see \cite[Lemma 2.7]{TodLimit}). So more work is needed to construct the hearts for stability conditions on projective varieties of dimension above one. In general, when $\Omega$ is a complexified ample class on a projective variety $X$ (that is $\Omega = B + i \sqrt{3}\alpha H$ for some $B, H \in \NS_{\RR}(X)$ with  ample class $H$, and $\alpha \in \RR_{>0}$), it is expected that 
\begin{equation}
\label{eqn:centralcharge}
Z_{\Omega}(-) = -\int_X e^{-\Omega}\ch(-)
\end{equation}
defines 
a central charge function of some stability condition on $X$ (see \cite[Conjecture 2.1.2]{BMT}). 
In  \cite{BMT}, the authors conjecturally construct a heart for this central charge function by double tilting 
coherent sheaves on $X$. 
The first tilt of $\Coh(X)$ associated to the Harder-Narasimhan
filtration with respect to the slope   stability, is denoted by 
$$
\bB_{\Omega} = \langle \fF_{\Omega}[1], \tT_{\Omega}  \rangle.
$$
They proved that abelian category $\bB_{\Omega}$  of two term complexes is Noetherian, and  
furthermore, they introduced the notion of tilt slope   stability for
objects in $\bB_{\Omega}$.
 The conjectural stability condition heart  
 $$
 \aA_{\Omega} = \langle\fF_{\Omega}'[1], \tT_{\Omega}' \rangle
 $$
  is the
 tilt of $\bB_{\Omega}$ associated to the
Harder-Narasimhan filtration with respect to the tilt slope
  stability.
It was shown in \cite{BMT} that 
the pair $(Z_{\Omega}, \aA_{\Omega})$ defines a Bridgeland stability condition on $X$ if and only if
any $E \in\bB_{\Omega}$ tilt slope   stable object with zero tilt slope satisfies  
$ \Ree Z_{\Omega} (E[1]) < 0$. 
Moreover, they  proposed the following strong  inequality for tilt stable objects with zero tilt slopes, and this is now commonly known as the   \textit{Conjectural Bogomolov-Gieseker Type Inequality}:
\begin{equation*}
\label{eqn:BGineq}
 \ch_3^B(E) - \frac{1}{6} \alpha^2 H^2 \ch_1^B(E) \le 0.
\end{equation*}
Here $\ch^B(E) = e^{-B}\ch(E)$ is the twisted Chern character.

This conjecture has been shown to hold for 
 all Fano threefolds with Picard rank one (see \cite{BMT, MacP3, SchQuadric, LiFano3}), abelian threefolds (see \cite{MP1, MP2, PiyThesis,   BMS}), \'etale quotients of abelian threefolds (see \cite{BMS}), some toric threefolds (see \cite[Theorem 5.1]{BMSZ}) and threefolds which are  products of projective spaces and abelian varieties (see \cite{Koseki}).  
Recently, Schmidt found a counterexample to the original Bogomolov-Gieseker type inequality conjecture when $X$ is the blowup at a point of $\mathbb{P}^3$ (see \cite{SchCounterExample}). 
Therefore, this inequality needs some modifications in general setting and this was discussed in \cite{PiyFano3, BMSZ}.

\subsection{Bridgeland Stability under Fourier-Mukai transforms}
\label{sec:intro-satb-under-FMT}
The notion of Fourier-Mukai transform (FM transform for short)  was introduced by
Mukai in early 1980s (see \cite{MukFMT}). In particular, he showed that the Poincar\'e bundle induces a
non-trivial equivalence between the derived categories of an abelian variety and its
dual variety. 
Furthermore, he studied certain type of vector bundles on abelian varieties called 
semihomogeneous bundles, and moduli of them (see \cite{MukSemihomo}).
In particular, the moduli space parametrizing simple semihomogeneous bundles on an abelian variety $Y$ with a fixed Chern character is also an abelian variety, denoted by $X$. Moreover, the associated universal bundle $\eE$ on $X \times Y$ induces a derived equivalence $\Phi_{\eE}^{\SX \to \SY}$ from $X$ to $Y$, which is now commonly known as the Fourier-Mukai transform. 

Action of the Fourier-Mukai transform $\Phi_{\eE}^{\SX \to \SY}$  induces stability conditions on $D^b(Y)$ from the ones on $D^b(X)$.
This can be defined via the induced  map on $\Hom(K(Y), \CC)$ from $\Hom(K(X), \CC)$ by the transform.
More precisely,
if $ (Z, \aA)$ is a stability condition on $D^b(X)$ then 
$$
\Phi_{\eE}^{\SX \to \SY} \cdot (Z, \ \aA) : = \left( \Phi_{\eE}^{\SX \to \SY} \cdot Z , \  \Phi_{\eE}^{\SX \to \SY}\left(\aA\right) \right) 
$$
defines a stability condition on $D^b(Y)$, where 
$\Phi_{\eE}^{\SX \to \SY} \cdot Z (-) = Z \left( \left( \Phi_{\eE}^{\SX \to \SY}\right)^{-1}\left(-\right) \right)$. 
For abelian varieties we view this as  
\begin{equation}
\label{eqn:FMT-action-central-charge}
\Phi_{\eE}^{\SX \to \SY} \cdot Z_{\Omega}  = \zeta \, Z_{\Omega'} 
\end{equation}
for some $\zeta \in \CC \setminus \{0\}$, where $\Omega, \Omega'$ are complexified ample classes on $X, Y$ respectively.  When $\zeta$ is real, one can expect that the Fourier-Mukai transform $\Phi_{\eE}^{\SX \to \SY}$   gives an equivalence of some hearts of particular stability conditions on $X$ and $Y$, whose $\Omega$ and $\Omega'$ are determined by $\Imm \zeta  = 0$. 

In particular, we prove the following for abelian threefolds:

\begin{thm}
\label{prop:intro-main-stab-symmetries}
The Fourier-Mukai transform $  \Phi_{\eE}^{\SX \to \SY}: D^b(X) \to D^b(Y)$ between the abelian threefolds gives the following symmetries of Bridgeland stability conditions:
$$
\Phi_{\eE}^{\SX \to \SY} [1]  \cdot \left(  Z_{\Omega} , \ \aA_{\Omega}  \right) =  \left(  \zeta Z_{\Omega'}, \ \aA_{\Omega'} \right)
$$
for some   $\zeta  \in \RR_{>0}$, and complexified ample classes $\Omega, \Omega'$ on $X, Y$ respectively.
Here $\aA_{\Omega} , \aA_{\Omega'}$ are the double tilted stability condition hearts as 
in  the construction of \cite{BMT}, and $ Z_{\Omega},  Z_{\Omega'}$ are the central charge functions as defined in \eqref{eqn:centralcharge}.
\end{thm}

The  analogous result of the above theorem for abelian surfaces holds due to Huybrechts and Yoshioka  and see \cite{HuyK3Equivalence, YoshiokaFMT} for further details.

 \subsection{Main ingredients}
 \label{sec:intro-main-ideas}
\subsubsection{Fourier-Mukai theory and polarizations}
The  Fourier-Mukai transform $\Phi_{\eE}^{\SX \to \SY}: D^b(X) \to D^b(Y)$ between the abelian varieties  induces a linear isomorphism $\Phi_{\eE}^{\SH}$ from
$H^{2*}_{\alg}(X,\QQ)$ to $H^{2*}_{\alg}(Y,\QQ)$, called the cohomological Fourier-Mukai transform.
In this article, we realize this linear isomorphism in anti-diagonal form with respect to some twisted Chern characters (see  Theorem \ref{prop:antidiagonal-rep-cohom-FMT}). Furthermore,  we  prove the following.

\begin{thm}[{= \ref{prop:derived-induce-polarization}}]
\label{prop:intro-derived-polarization}
If the ample line bundle $L$ defines a polarization on $X$,  then
the line bundle $\det (\Xi(L))^{-1}$ is ample and so it defines a polarization on $Y$. Here 
 $\Xi $ is the Fourier-Mukai functor from $D^b(X)$ to $D^b(Y)$ defined by 
$$
\Xi = \eE_{\{a\}\times Y}^* \circ \Phi_\eE^{\SX  \to \SY} \circ \eE_{X \times \{b\}}^*,
$$
where $a, b$ are any two points on $X, Y$ respectively; and $\eE_{\{a\}\times Y}^*$ denotes the functor $\eE_{\{a\}\times Y}^* \otimes (-)$ and similar for $ \eE_{X \times \{b\}}^*$. 
\end{thm}

This theorem generalizes  similar results for abelian surfaces (see \cite[Section 1.3]{YoshiokaFMT}) and for all abelian varieties with respect to the classical Fourier-Mukai transform with kernel the Poincar\'e bundle (see \cite{BL-polarization}).
\subsubsection{Stability under Fourier-Mukai transforms}
The main goal of this paper is to prove  Theorem \ref{prop:intro-main-stab-symmetries}, and for that we need to
establish the corresponding equivalence of the double tilt stability condition hearts on the abelian threefolds.
This is a generalization of the main results in \cite{MP1, MP2, PiyThesis}.
More specifically, we extend many techniques in  \cite{MP1, MP2, PiyThesis} on a principally polarized abelian threefold with Picard rank one  to a general abelian threefold.

In Section \ref{sec:FMT-abelian-varieties}, we study the behavior of slope stability of sheaves under the 
Fourier-Mukai transform $\Phi_{\eE}^{\SX \to \SY}$ on any abelian varieties.
In Section \ref{sec:equivalence-stab-hearts-surface} we establish the analogous result of Theorem \ref{prop:intro-main-stab-symmetries} for abelian surfaces, and our main aim is to get some familiarization with Fourier-Mukai techniques to prove our main theorem.
Here we closely follow the proof of Yoshioka in \cite{YoshiokaFMT}.

Understanding the homological Fourier-Mukai transform for abelian threefolds is central to
this paper.
In  Sections \ref{sec:FMT-sheaves-abelian-threefolds} and \ref{sec:further-FMT-sheaves-abelian-threefolds}, we study the slope stability of sheaves
under the Fourier-Mukai transforms. In particular, at the end of Section \ref{sec:further-FMT-sheaves-abelian-threefolds}, we prove that 
\begin{equation*}
\label{dddd}
 \left.\begin{aligned}
         & \Phi_{\eE}^{\SX \to \SY} \left(\tT_{\Omega} \right) \subset \langle \bB_{\Omega'}, \bB_{\Omega'}[-1], \bB_{\Omega'}[-2] \rangle \\
         & \Phi_{\eE}^{\SX \to \SY} \left(\fF_{\Omega} \right) \subset \langle \bB_{\Omega'}[-1], \bB_{\Omega'}[-2], \bB_{\Omega'}[-3] \rangle
       \end{aligned}
  \ \right\}.
\end{equation*}
From the definition of the first tilt, we have that the images 
under the  Fourier-Mukai  transform $ \Phi_{\eE}^{\SX \to \SY}$   of 
the objects in the abelian category $\bB_{\Omega}$  
have non-zero
cohomologies with respect to $\bB_{\Omega'}$ only in
positions $0$, $1$ and $2$.
We prove a similar result for the Fourier-Mukai transform $\Phi_{\eE^\vee}^{\SY \to \SX}[1]: D^b(Y) \to D^b(X)$.
That is 
\begin{equation*}
\label{ddddds}
 \left.\begin{aligned}
         & \Phi_{\eE}^{\SX \to \SY} \left( \bB_{\Omega} \right)  
         \subset \langle \bB_{\Omega'} ,  \bB_{\Omega'}[-1],  \bB_{\Omega'}[-2]  \rangle \rangle \\
         & \Phi_{\eE^\vee}^{\SY \to \SX}[1] \left(\bB_{\Omega'} \right)  
         \subset \langle \bB_{\Omega} ,  \bB_{\Omega}[-1],  \bB_{\Omega}[-2]  \rangle 
       \end{aligned}
  \ \right\}.
\end{equation*}
Since we have the isomorphisms
$\Phi_{\eE^\vee}^{\SY \to \SX}[1] \circ \Phi_{\eE}^{\SX \to \SY} \cong [-2]$ and $\Phi_{\eE}^{\SX \to \SY} \circ \Phi_{\eE^\vee}^{\SY \to \SX}[1] \cong [-2]$,
the abelian categories  $\bB_{\Omega} $ and $\bB_{\Omega'}$ behave
 somewhat similarly to the category of coherent sheaves on an abelian surface under the Fourier-Mukai transforms.
Finally, in Section \ref{sec:FMT-tilt-stability}, we study the behavior of tilt stability under the Fourier-Mukai transforms.  
In particular,  we prove that 
\begin{equation*}
\label{sdssdsdsd}
 \left.\begin{aligned}
         & \Phi_{\eE}^{\SX \to \SY} \left(\tT_{\Omega}' \right) \subset \langle \fF_{\Omega'}', \tT_{\Omega'}' [-1] \rangle \\
         & \Phi_{\eE}^{\SX \to \SY}  \left(\fF_{\Omega}' \right) \subset \langle \fF_{\Omega'}'[-1], \tT_{\Omega'}' [-2] \rangle 
       \end{aligned}
  \ \right\},
\end{equation*}
and similar results for $ \Phi_{\eE^\vee}^{\SY \to \SX}[1]$.
From the definition of the second tilt, we have the following:
\begin{thm}
\label{prop:intro-equivalence-threefolds}
The derived equivalences $\Phi_{\eE}^{\SX \to \SY} $ and $\Phi_{\eE^\vee}^{\SY \to \SX}$ give the equivalences of the double tilted hearts 
$$
\Phi_{\eE}^{\SX \to \SY}[1] \left(\aA_{\Omega} \right) \cong \aA_{\Omega'}, \
\text{ and } \ \Phi_{\eE^\vee}^{\SY \to \SX}[2] \left(\aA_{\Omega'} \right) \cong \aA_{\Omega}.
$$
\end{thm}

\subsubsection{Bogomolov-Gieseker type inequality for abelian threefolds}
For a given smooth projective threefold $X$, let $\mM_{\Omega}$ be
the class of
tilt  stable   objects  $E$ with zero tilt slope
and $\Ext^1_{\SX}(\oO_x, E) = 0$ for all $x \in X$.
In Lemma \ref{prop:minimal-objects-threefold-hearts}, we see 
that the objects in $\mM_{\Omega}[1]$ are minimal objects (also called
simple objects in the literature)  in $\aA_{\Omega}$.
Moreover,  due to Lemma \ref{prop:reduction-BG-ineq-class}, we only need to check the
Bogomolov-Gieseker type inequalities for tilt stable objects  in $\mM_{\Omega}$.

Minimal objects of the abelian subcategories $\aA_{\Omega}$ are sent to minimal objects of $\aA_{\Omega'}$  under the Fourier-Mukai  transform $\Phi_{\eE}^{\SX \to \SY}[1]$.
This enables us to obtain an inequality involving the top
part of the Chern character of minimal objects in these abelian
categories.
This is exactly the   Bogomolov-Gieseker type
inequality for tilt stable objects  in $\mM_{\Omega}$.
Therefore, we have the following:

\begin{thm}[=\ref{prop:BG-ineq-abelian-threefolds}]
\label{prop:intro-BGineq-abelian-threefolds}
Any tilt stable object with zero tilt slope satisfies the
strong Bogomolov-Gieseker type inequality for any abelian threefold.
\end{thm}

Theorems \ref{prop:intro-equivalence-threefolds} and \ref{prop:intro-BGineq-abelian-threefolds} together with the double tilting construction in \cite{BMT} proves Theorem \ref{prop:intro-main-stab-symmetries}.

\subsection{Higher dimensional abelian varieties}
\label{sec:intro-higher-dim-abelian-varieties}
In Section \ref{sec:cojectural-any-abelian}, for any abelian variety we conjecturally construct a heart for the central charge function \eqref{eqn:centralcharge}, by using the  notion of
 very weak stability condition (see Conjecture \ref{prop:conjecture-stab-cond}).  This essentially generalizes  the  single tilting construction due to Bridgeland and Arcara-Bertram for surfaces (\cite{BriK3, AB}), and the conjectural double tilting construction due to Bayer-Macr\`i-Toda for threefolds (\cite{BMT}).  


 By considering the complexified ample classes $\Omega$ and $\Omega'$  determined by $\Imm \zeta  = 0$ in \eqref{eqn:FMT-action-central-charge},    we formulate the following for the Fourier-Mukai transform $\Phi_{\eE}^{\SX \to \SY}$.

\begin{conj}[=\ref{prop:conjecture-equivalence-stab-hearts}]
\label{prop:intro-conjectural-equivalence}
The Fourier-Mukai transform $\Phi_{\eE}^{\SX \to \SY}: D^b(X) \to D^b(Y)$ gives the equivalence of stability condition hearts conjecturally constructed in Conjecture \ref{prop:conjecture-stab-cond}:
$$
\Phi_{\eE}^{\SX \to \SY} [k]  \left(  \aA^{\SX}_{\Omega}  \right) = \aA^{\SY}_{\Omega'}.
$$
Here  $\Omega = -D_{\SX} + \lambda e^{i k\pi/g }\, \lx $ and $\Omega' = D_{\SY}  - (1/\lambda) e^{-i k \pi/g} \, \ly$
are complexified ample classes on $X$ and $Y$ respectively, for any   $k \in \{1, 2, \ldots, (g-1)\}$ and any $ \lambda \in \RR_{>0}$. 
\end{conj}


\subsection{Relation to the existing works}
\label{sec:intro-current-work}
 
 \subsubsection{Relation to \cite{MP1, MP2, PiyThesis}}
As mentioned before, this paper generalizes previous work \cite{MP1, MP2, PiyThesis} on a principally polarized abelian threefold with Picard rank one to any abelian threefold.
Moreover, many proofs in this paper are adopted from that of the similar results in those works. 
Also for the completeness and for the convenience of the reader, we  give almost all the proofs relevant to general abelian threefolds.
In particular, we extend the proof 
of the Bogomolov-Gieseker type inequality conjecture in \cite{MP1,MP2, PiyThesis}  for any abelian threefold by using the Fourier-Mukai theory. 

Let us highlight the connections of the notations in this paper with the notations in \cite{MP1, MP2, PiyThesis}.
Suppose $X$ is a principally polarized abelian threefold with Picard rank one. Let $\lx \in \NS(X)$ be the corresponding principal polarization, and so   $\lx^3/6 =1$. The twisted Chern character of any $E \in D^b(X)$ is of the form  $\ch^B(E) = (a_0, a_1 \lx, a_2 \lx^2/2 , a_3\lx^3/6)$ for some $a_i \in \QQ$ when $B$ is a rational class, and   in \cite{MP1,MP2, PiyThesis} the authors simply denote such Chern characters in  vector form
\begin{equation}
\label{eqn:chern-ppa3}
(a_0, a_1, a_2, a_3) \in \QQ^{4}.
\end{equation}
 They consider the twisted slope function on $\Coh(X)$ defined by 
$a_1/a_0$, 
and study the slope stability of sheaves under the Fourier-Mukai transforms on $X$. 
Moreover, they consider  the tilt slope defined in terms of $a_0$, $a_1$ and $a_2$,
and study the tilt stability of complexes in the first tilted hearts under the Fourier-Mukai transforms. 
In this paper we are interested in the twisted slope functions and also tilt slope functions defined with respect to the numerology in the vector 
$$
v^{B, \lx} (E)= (\lx^3 \ch_0^B(E) , \lx^2 \ch_1^B(E), 2\lx \ch_2^B(E) , 6 \ch_3^B(E)).
$$
Here $\lx$ is any ample class in $\NS_{\QQ}(X)$. 
Now one can see that for the principally polarized abelian threefold with Picard rank one case, 
$$
v^{B, \lx} (E) = 6 (a_0, a_1, a_2, a_3),
$$
that is a fixed scalar multiple of the vector in \eqref{eqn:chern-ppa3}.

 \subsubsection{Relation to other works}

The main results in this paper were summarized in the author's article \cite{PiyKinosaki} for the
Proceedings of Kinosaki Symposium on Algebraic Geometry  2015.

In \cite{BMS}, the authors  establish the Bogomolov-Gieseker type inequality conjecture  for any abelian threefold  by extensive use of the multiplication map $x \mapsto mx$ on abelian threefolds.

In \cite{YoshiokaFMT}, Yoshioka studied the behavior of slope stability under the 
Fourier-Mukai transform  on  abelian surfaces.
Moreover, he established the claim in Conjecture \ref{prop:intro-conjectural-equivalence} for abelian surfaces using Fourier-Mukai theory, however, this is  firstly known due to Huybrechts (\cite{HuyK3Equivalence}).

In a forthcoming article we use the main result of this paper (Theorem \ref{prop:intro-main-stab-symmetries}) to prove the full support property and to  study the  stability manifold of any abelian threefold.

\subsection{Notation}

\begin{itemize}[leftmargin=*]

\item When  $\aA$ is the heart of a bounded t-structure  on 
a triangulated category $\dD$,  by 
$H_{\aA}^i(-)$ we denote the corresponding  $i$-th cohomology functor. 

\item 
For a set of objects $\sS \subset \dD$ in a triangulated category $\dD$,  by 
$\langle \sS \rangle \subset \dD$ we denote
its extension closure, that is the smallest extension closed subcategory 
of $\dD$ which contains $\sS$.

\item Unless otherwise stated,  throughout this paper, all the varieties are smooth projective and defined over 
$\mathbb{C}$.  For a variety $X$, by $\Coh(X)$ we denote  the category of 
coherent sheaves on $X$, and by $D^b(X)$ we denote the bounded derived category of $\Coh(X)$. That is 
$D^b(X) = D^b( \Coh(X))$.

\item For  $D^b(X)$ we simply write  $\hH^i(-)$ for $H_{\Coh(X)}^i(-)$.

\item For a variety $X$, 
by $\omega_X$ we denote its canonical line bundle, and let $K_X = c_1(\omega_X) $.

\item For  $ M = \QQ, \RR, \text{ or } \CC$  we write $\NS_M(X) = \NS(X) \otimes_{\ZZ} M$.


\item For $0 \le i \le \dim X$, $\Coh_{\le i}(X)  = \{E \in \Coh(X): \dim \Supp(E)  \le i  \}$, $\Coh_{\ge i}(X) = \{E \in \Coh(X): \text{for } 0 \ne F \subset E,   \ \dim \Supp(F)  \ge i  \}$ and $\Coh_{i}(X) = \Coh_{\le i}(X) \cap \Coh_{\ge i}(X)$. 

\item For $E \in D^b(X)$, $E^\vee = \RRR \calHom(E, \oO_X)$. When $E$ is a sheaf we write its dual sheaf $\hH^0(E^\vee)$ by $E^*$.

\item The structure sheaf of a closed subscheme $Z \subset X$ as an object in $\Coh(X)$ is denoted by $\oO_Z$, and when $Z = \{x\}$ for a closed point $x\in X$, it is simply denoted by $\oO_x$.

\item $\ch_{\le k} = (\ch_0, \ch_1, \ldots, \ch_k, 0, \ldots, 0)$, and 
$\ch_{\ge k} = (0, \ldots, 0, \ch_k, \ch_{k+1}, \ldots, \ch_n)$.

\item For  $B \in \NS_{\RR}(X)$, the twisted Chern character $\ch^B(-) = e^{-B} \cdot \ch(-)$. 
For ample $H \in \NS_{\RR}(X)$, we define $v^{B,H}(E) = (H^3 \ch_0^B(E), H^2 \ch_1^B(E), 2H \ch_2^B(E), 6 \ch_3^B(E))$.

\item  The twisted slope on $\Coh(X)$ is defined by $\mu_{H,B}(E) = \dfrac{H^2 \ch_1^B(E)}{H^3 \ch_0(E)} = \dfrac{v_1^{B,H}(E)}{v_0^{B,H}(E)}$.

\item  Tilt slope on $\bB_{H,B}$ is defined by 
$$
\nu_{H,B, \alpha} (E) = \dfrac{H \ch_2^{B}(E) - (\alpha^2/2) H^3\ch_0(E)}{H^2 \ch^B_1(E)}
= \dfrac{v_1^{B,H}(E) - \alpha^2 v_0^{B,H}(E)}{2 v_1^{B,H}(E)}.
$$ 

\item $\HN^{\mu}_{H, B}(I) = \langle E \in \Coh(X) : E \text{ is } \mu_{H , B}\text{-semistable with }
\mu_{H , B}(E) \in I \rangle$. Similarly, we define  $\HN^{\nu}_{H, B}(I) \subset \bB_{H,B}$.

\item We denote the upper half plane $\{z \in \mathbb{C} : \Imm z >0\}$ by $\mathbb{H}$. 

\item We will denote a  $g \times g$ anti-diagonal matrix with entries $a_k$, $k=1, \ldots,  g$ by
$$
\Adiag(a_1, \ldots, a_g)_{ij} : = \begin{cases}
a_k & \text{if } i=k, j=g+1-k \\
0 & \text{otherwise}.
\end{cases}
$$

\end{itemize}
\subsection{Acknowledgements}
The author would like to specially thank  Antony Maciocia for his guidance given to his doctoral studies.
The author is grateful to Yukinobu Toda for very helpful discussions relating to this work, and also to Arend Bayer and Tom Bridgeland for very useful comments and suggestions given 
during the  author's PhD defense. 
This work is supported by the World Premier International Research Center Initiative (WPI Initiative), MEXT, Japan.
\section{Preliminaries}
\label{sec:preliminaries}
\subsection{Some homological algebra}
\label{sec:some-homological-algebra}

A {\it triangulated category} $\dD$ is an additive category equipped with a shift functor, and a class of triangles, called distinguished triangles satisfying certain axioms. We denote the shift
 functor by $[1]: \dD \to \dD$, and  write a distinguished triangle as $A \to B \to C \to A[1]$. 
The bounded derived categories of coherent sheaves on smooth projective varieties are the most important examples of triangulated categories in this paper.

\begin{defi}
\rm
A {\it t-structure} on $\dD$ is a pair of strictly full subcategories $(\dD^{\le 0}, \dD^{\ge 0} )$ such that,
if we let $\dD^{\le n} =\dD^{\le 0}[-n]$ and  $\dD^{\ge n} =\dD^{\ge 0}[-n]$, for $n \in \ZZ$, then we have
\begin{enumerate}[label=(\roman*)]
\item  $\dD^{\le 0} \subset \dD^{\le 1}$, $\dD^{\ge 0} \supset  \dD^{\ge 1}$,
\item $\Hom_{\dD}(E, F ) = 0$  for $E \in \dD^{\le 0}$ and $F \in \dD^{\ge 1}$,
\item for any $G \in \dD$ there exists a distinguished triangle $E \to G \to F \to E[1]$ such that $E \in \dD^{\le 0}$ and $F \in \dD^{\ge 1}$.
\end{enumerate}
The {\it heart} $\cC $ of this t-structure is $\cC = \dD^{\le 0} \cap  \dD^{\ge 0}$.
The t-structure is called {\it bounded} if
$$
\bigcup_{n \in \ZZ} \dD^{\le n} = \dD = \bigcup_{n \in \ZZ} \dD^{\ge n}.
$$
\end{defi}
It is known that the heart $\cC$ is an abelian category, and also
a bounded t-structure is determined by its heart (see \cite[Lemma 3.1]{BriK3}).
So we denote the $i$-th cohomology of $E \in \dD$ with respect to the t-structure $(\dD^{\le 0}, \dD^{\ge 0} )$ by
$H^i_{\cC}(E)$.

If $A \to B \to C \to A[1]$ is a distinguished triangle in $\dD$, then we have the exact sequence
$$
\cdots \to  H^{i-1}_{\cC}(C) \to  H^i_{\cC}(A) \to  H^i_{\cC}(B) \to  H^i_{\cC}(C) \to  H^{i+1}_{\cC}(A) \to  \cdots
$$
 of cohomologies from $\cC$.

Let $D^b(\aA)$ be the bounded derived category of an abelian category $\aA$. Then the pair of subcategories
\begin{equation*}
 \left.\begin{aligned}
         &D^b(\aA)^{\le 0} = \{E \in D^b(\aA) : H^i_{\aA}(E)= 0 \text{ for } i>0   \} \\
         &D^b(\aA)^{\ge 0} = \{E \in D^b(\aA) : H^i_{\aA}(E)= 0 \text{ for } i<0   \}
       \end{aligned}
  \ \right\}
\end{equation*}
define a bounded t-structure on $D^b(\aA)$ and the corresponding heart is $\aA$. This is called the {\it standard t-structure} on $D^b(\aA)$.

Let us discuss about the torsion theory of an abelian category.
It provides a useful method, called tilting, to construct interesting t-structures from the known ones.
This was first introduced by Happel,  Reiten and Smal{\o}  in \cite{HRS}.

\begin{defi}
\label{dsdsdsdsd}
\rm
A {\it torsion pair} on an abelian category $\aA$ is a pair of subcategories $(\tT, \fF)$ of $\aA$ such that
\begin{enumerate}[label=(\roman*)]
\item $\Hom_{\aA}(T,F) = 0$ for every $T \in \tT$, $F \in \fF$, and
\item every $E \in \aA$ fits into a short exact sequence
$0 \to T \to E \to F \to 0$
in $\aA$ for some  $T \in \tT$, $F \in \fF$.
\end{enumerate}
\end{defi}

\begin{lem}[{\cite[Proposition 2.1]{HRS}}]
\label{sdsdsdsd}
Let $\aA$ be the heart of a bounded t-structure on a triangulated category $\dD$ and let $(\tT,\fF)$ be a torsion pair on $\aA$. Then
the full subcategory defined by
$$
\bB = \{E \in \dD : H^{i}_{\aA}(E) = 0 \text{ for } i \ne -1,0, \ H^{-1}_{\aA}(E) \in \fF, \ H^{0}_{\aA}(E) \in \tT \}
$$
is the heart of bounded t-structure given by the pair of subcategories
\begin{equation*}
 \left.\begin{aligned}
         &\dD^{\le 0} = \{X \in \dD : H^{i}_{\aA}(E) = 0 \text{ for } i > 0, \  H^{0}_{\aA}(E) \in \tT \} \\
         &\dD^{\ge 0} = \{X \in \dD : H^{i}_{\aA}(E) = 0 \text{ for } i <-1, \  H^{-1}_{\aA}(E) \in \fF \}
       \end{aligned}
  \ \right\}.
\end{equation*}
\end{lem}

The abelian subcategory $\bB \subset \dD$ is usually called the {\it tilt} of $\aA$ with respect to the torsion pair $(\tT,\fF)$
and  we also write $\bB = \langle \fF[1], \tT \rangle$. The t-structures defined by the hearts $\aA$ and $\bB$ give two different views for the objects in the triangulated category $\dD$.

The {\it Grothendieck group} $K(\aA)$ of an abelian category $\aA$ is the quotient of the free abelian group generated by the classes $[A]$ of objects $A \in \aA$ modulo the relations given
by $[A]+[C] = [B]$ for every short exact sequences
$0 \to A \to B \to C \to 0$ in $\aA$.
Similarly,  the Grothendieck group $K(\dD)$ of a triangulated category $\dD$ is the free abelian group generated by the classes $[A]$ of $A\in \dD$ with the relations $[A]+[C] = [B]$ for
every distinguished triangles
$A \to B \to C \to A[1]$  in $\dD$.
If $\aA$ is the heart of a bounded t-structure on $\dD$ then 
$
K(\dD)=K(\aA)$.
Moreover, when $\aA = \Coh(X)$ for a variety $X$ we write 
$$
K(X) = K(\Coh(X)) = K(D^b(X)).
$$

\subsection{Bridgeland stability on varieties}
Let us introduce the notion of stability conditions as in \cite{BriStab}.
Let $\aA$ be an abelian category. 

A group homomorphism $Z : K(\mathcal{A}) \to \mathbb{C}$ is called a \textit{stability function (also known as central charge function)},
if for all $0 \ne E \in \mathcal{A}$, $
Z(E) \in  \HH \cup \RR_{<0}$.

The \textit{phase} of $0 \ne E \in \mathcal{A}$ is defined by $
\phi(E)  = \frac{1}{\pi}  \arg Z(E) \in (0,1]$.

An object $0 \ne E \in \mathcal{A}$ is called \textit{(semi)stable},
 if     for any $ 0 \ne A \varsubsetneq  E $ in $\aA$,  
 $\phi(A) < (\le) \, \phi(E/A)$.
 
A \textit{Harder-Narasimhan filtration} of $0 \ne E \in \mathcal{A}$ is a finite chain of subobjects
\begin{equation}
\label{eqn:HN-filtration-definition}
0=E_0 \subset E_1 \subset \cdots \subset E_{n-1} \subset E_n = E,
\end{equation}
where factors $F_k = E_k/E_{k-1}$, $k=1, \ldots, n$, are semistable in $\mathcal{A}$ with
$$
\phi(F_1) > \phi(F_2) >  \cdots > \phi(F_{n-1}) > \phi(F_n).
$$
The stability function $Z$ satisfies the \textit{Harder-Narasimhan property} for $\aA$, if such a  filtration exists  for any non-trivial object in $\aA$.

When the Harder-Narasimhan property holds for $\aA$ with respect to the stability function $Z$,
one can show that the filtration \eqref{eqn:HN-filtration-definition} is unique for a given $E \in \aA$.

\begin{defi}[{\cite[Proposition 5.3]{BriStab}}]
\label{defi:Bridgeland-stability}
A stability condition on a triangulated category  $\dD$
is  given by  a pair $(Z, \aA)$, where $\aA$ is the heart of a  bounded t-structure on $\dD$ and a $Z: K(\aA) \to \CC$ is  stability function, such that 
the Harder-Narasimhan property holds for $\aA$ with respect to the stability function $Z$.
\end{defi}

Let $X$ be a  smooth projective variety and let $D^b(X)$  be
the bounded derived category of coherent sheaves on $X$.
We are interested in stability conditions $\sigma = (Z, \aA)$ on 
$D^b(X)$, where  the stability function $Z: K(X) \to \CC$
factors through the Chern character map $\ch: K(X) \to H^{2*}_{\alg}(X,\QQ)$. 
Such stability conditions are usually called  \textit{numerical stability conditions}.

A stability condition $\sigma$ on $D^b(X)$ is called \textit{geometric} if
all the skyscraper sheaves $\oO_x$ of $x \in X$ are $\sigma$-stable of the same phase.
The following result gives some properties of  geometric stability
conditions on  varieties.

\begin{prop}
\label{prop:property-higher-dim-stability-hearts}
Let $X$ be a smooth projective variety of dimension $n$.
Let $\sigma= (Z, \aA)$ be a geometric stability condition on $D^b(X)$ with
all the skyscraper sheaves $\oO_x$ of $x \in X$ are $\sigma$-stable with phase one.
If $E \in  \aA$ then $\hH^i(E) = 0$ for $i \notin \{-n+1, -n+2, \ldots ,0\}$.
\end{prop}
\begin{proof}
The following proof is adapted from \cite[Lemma 10.1]{BriK3}.
Let $\pP$ be the corresponding slicing of $\sigma$.
Since $\aA = \pP((0,1])$ and $\Coh_0(X) \subset \pP(1)$, from the Harder-Narasimhan property, we only need to consider $E \in \aA $ such that $\Hom_{\SX} (\Coh_0(X), E) =0$.
For any skyscraper sheaf $\oO_x$ of $x \in X$ we have $\oO_x[i] \in \pP(1+i)$ and  $E[i] \in \pP((i,1+i])$.
Therefore, for all $i<0$,
$\Hom_{\SX}(E, \oO_x[i]) = 0$, and $\Hom_{\SX}(\oO_x, E[1+i])$ $\cong$ $\Hom_{\SX}(E,\oO_x [n-1-i])^* =0$.
So by \cite[Proposition 5.4]{BM}, $E$ is quasi-isomorphic to a complex of locally free sheaves of length $n$.
This completes the proof as required. 
\end{proof}

When $X$ is a smooth projective curve, the central charge  function $Z$ defined by 
$Z(-) = -\deg(-) + i \rk (-)$ together with the  heart 
$\Coh(X)$ of the standard t-structure defines a geometric stability condition on $D^b(X)$.
However, for a smooth projective variety $X$ with $\dim  X  \ge 2$, there is no numerical
stability condition on $D^b(X)$ with $\Coh(X)$ as the heart of a stability condition 
(see \cite[Lemma 2.7]{TodLimit} for a proof). 
In fact, for a smooth projective surface $X$, when $\sigma = (Z, \aA)$ is a geometric Bridgeland stability condition, the heart $\aA$ is a tilt of $\Coh(X)$ with respect to a torsion pair coming from the usual slope stability on $\Coh(X)$ (see \cite{BriK3, AB}).  

\subsection{Double tilting stability construction on threefolds}
\label{sec:double-tilting-construction}
Let us briefly recall the conjectural construction of stability conditions on
a given smooth projective threefold $X$ as
introduced in \cite{BMT}.

Let $H, B \in \NS_{\RR}(X)$ such that $H$ an ample class. The twisted Chern
character with respect to $B$
 is defined by 
 \begin{equation*}
  \ch^B(-) = e^{-B} \ch (-).
 \end{equation*}
The twisted slope $\mu_{H , B}$ on  $\Coh(X)$ is defined by, for $E \in \Coh(X)$
$$
\mu_{H, B} (E) = \begin{cases}
+ \infty & \text{if } E \text{ is a torsion sheaf} \\
\frac{H^{2} \ch_1^B(E)}{H^3 \ch^B_0(E)} & \text{otherwise}.
\end{cases}
$$
So we have $\mu_{H,B + \beta H} = \mu_{H, B} - \beta$.

We say  $E \in \Coh(X)$ is $\mu_{H , B}$-(semi)stable, if for any $0 \ne
F \varsubsetneq E$, $\mu_{H , B}(F)< (\le) \mu_{H ,
  B}(E/F)$. 
  
\begin{defi}
\label{def:discriminant}
For $E \in D^b(X)$ we define
\begin{align*}
 & \Delta(E) = (\ch_1(E))^2- 2 \ch_0(E) \ch_2(E) \in H^4_{\alg}(X, \ZZ) ,\\
& \overline{\Delta}_{H, B}(E)  = (H^2 \ch_1^B(E))^2 - 2 H^3 \ch_0(E)  H \ch_2^B(E).
\end{align*}
\end{defi}

\begin{lem}[{Bogomolov-Gieseker Inequality, \cite{HLBook}}]
\label{prop:usual-BG-ineq}
Let $E$ be $\mu_{H,B}$ semistable torsion free sheaf. Then it satisfies 
$$
H \cdot \Delta(E) \ge 0, \ \text{ and }  \ \overline{\Delta}_{H, B}(E) \ge 0.
$$
\end{lem}

  The Harder-Narasimhan property holds for $\mu_{H , B}$ stability on
$\Coh(X)$. 
This enables us to define the following slopes:
\begin{equation*}
\label{eqn:highest-lowest-HN-slopes}
 \left.\begin{aligned}
         &\mu_{H , B}^{+}(E)  = \max_{0 \ne G \subseteq E} \ \mu_{H , B}(G) \\
         & \mu_{H , B}^{-}(E)  = \min_{G \subsetneq E} \ \mu_{H , B}(E/G)
       \end{aligned}
  \ \right\}.
\end{equation*}
Moreover, 
 for a given interval $I \subset \mathbb{R} \cup\{+\infty\}$, we  define 
 the subcategory $\HN^{\mu}_{H, B}(I) \subset
\Coh(X)$ by
\begin{equation}
\label{eqn:HN-mu-interval-subcat}
\HN^{\mu}_{H, B}(I) = \langle E \in \Coh(X) : E \text{ is } \mu_{H , B}\text{-semistable with }
\mu_{H , B}(E) \in I \rangle.
\end{equation}
The subcategories  $\tT_{H , B}$ and $\fF_{H , B}$ of $\Coh(X)$ are defined by
\begin{align*}
\tT_{H , B} = \HN^{\mu}_{H, B}((0, +\infty]), \ \ \
\fF_{H , B} = \HN^{\mu}_{H, B}((-\infty, 0]).
\end{align*}
Now $( \tT_{H , B} , \fF_{H, B})$ forms a torsion pair on $\Coh(X)$
and let the abelian category
\begin{equation*}
\bB_{H , B} = \langle \fF_{H , B}[1], \tT_{H, B} \rangle \subset D^b(X)
\end{equation*}
be the corresponding tilt of $\Coh(X)$.

Let $\alpha \in \RR_{>0}$. 
Following \cite{BMT}, the tilt-slope $\nu_{H, B, \alpha} $ on $\bB_{H , B}$ is defined by, for $E \in \bB_{H,B}$
$$
\nu_{H, B, \alpha}(E) =
\begin{cases}
+\infty & \text{if } H^2 \ch^B_1(E) = 0 \\
\frac{H \ch_2^{B}(E) - (\alpha^2/2) H^3\ch_0(E)}{H^2 \ch^B_1(E)} & \text{otherwise}.
\end{cases}
$$
In \cite{BMT}, the notion of $\nu_{H , B, \alpha}$-stability for objects in
$\bB_{H , B}$ is introduced in a similar way to $\mu_{H,
  B}$-stability on $\Coh(X)$. Also it is proved that the abelian
category $\bB_{H , B}$ satisfies the Harder-Narasimhan property with respect to
$\nu_{H , B, \alpha}$-stability. 
Then similar to \eqref{eqn:HN-mu-interval-subcat}  we define the subcategory $\HN^{\nu}_{H, B, \alpha}(I) \subset \bB_{H, B}$ for an
interval $I \subset \mathbb{R} \cup\{+\infty\}$.
The subcategories $\tT_{H , B, \alpha}'$ and $\fF_{H , B, \alpha}'$ of $\bB_{H, B}$ are defined by
\begin{align*}
\tT_{H , B, \alpha}' = \HN^{\nu}_{H, B, \alpha}((0, +\infty]), \ \ \
\fF_{H, B}' = \HN^{\nu}_{H, B, \alpha}((-\infty, 0]).
\end{align*}
Then $( \tT_{H , B, \alpha}' , \fF_{H , B, \alpha}')$ forms a torsion pair on
$\bB_{H , B}$, and let the abelian category 
\begin{equation}
\label{eqn:double-tilt-heart}
\aA_{H , B, \alpha} = \langle
\fF_{H , B, \alpha}'[1],\tT_{H , B, \alpha}' \rangle \subset D^b(X)
\end{equation}
 be the corresponding tilt.
 
 \begin{defi}
 \label{defi:central-charge}
 The central charge $Z_{H,B, \alpha} : K(X) \to \CC$ is defined by
 $$
 Z_{H,B, \alpha}(-) = \int_X e^{-B - i \sqrt{3}\alpha  H} \ch(-).
 $$
 \end{defi}
 
In \cite{BMT}, authors made the following conjecture to construct stability conditions.
 \begin{conj}[{\cite[Conjecture 3.2.6]{BMT}}]
 \label{prop:BMT-stab-cond-conjecture}
  The pair $(Z_{H,B, \alpha}, \aA_{H,B, \alpha})$
  is a Bridgeland stability condition on $D^b(X)$.
  \end{conj}
Let us assume $H, B \in \NS_{\QQ}(X)$ and $\alpha^2 \in \QQ$ then 
 similar to the proof of \cite[Proposition 5.2.2]{BMT}  one can show that the abelian category $\aA_{H,B, \alpha}$ is Noetherian. Therefore Conjecture \ref{prop:BMT-stab-cond-conjecture} is equivalent to saying that any $\nu_{H,B, \alpha}$-stable object
 $E \in \bB_{H,B}$ with 
 $\nu_{H,B,\alpha}(E) =0$ satisfies 
\begin{equation*}
\label{eqn:weak-BG-ineq}
\Ree Z_{H,B, \alpha}(E[1]) < 0.
\end{equation*}
See \cite[Corollary 5.2.4]{BMT} for further details.

Moreover in \cite{BMT} they proposed the following strong inequality: 
  \begin{conj}[{\cite[Conjecture 1.3.1]{BMT}}]  
  \label{prop:BG-ineq-conjecture}
Any $\nu_{H,B, \alpha}$ stable objects 
 $E \in \bB_{H,B}$ with 
 $\nu_{H,B,\alpha}(E) =0$ satisfies 
 the so-called \textbf{￼Bogomolov-Gieseker Type Inequality}:
$$
\ch_{3}^B(E) - \frac{1}{6}\alpha^2 \ch_1^{B}(E) \le 0.
$$
\end{conj}
Since this stronger conjectural inequality implies the above weak inequality,  Conjecture \ref{prop:BG-ineq-conjecture} implies Conjecture\ref{prop:BMT-stab-cond-conjecture}.
\subsection{Some properties of tilt stable objects and minimal objects}
\label{sec:tiltproperties}

Let $X$ be a smooth projective threefold.
We follow the same notations for tilt stability introduced in Section \ref{sec:double-tilting-construction} for $X$.

\begin{prop}[{\cite[Lemma 3.2.1]{BMT}}]
\label{prop:first-tilt-behaves-like-sheaves-surfaces}
For any $0 \ne E \in \bB_{H,B}$, one of the following conditions holds:
\begin{enumerate}[label=(\roman*)]
\item $H^2 \ch_1^B(E) > 0$,
\item $H^2 \ch_1^B(E) =0$ and $\Imm Z_{H,B,\alpha}(E) > 0$,
\item  $H^2 \ch_1^B(E) = \Imm Z_{H,B,\alpha}(E) =0$, $- \Ree Z_{H,B,\alpha}(E) > 0$ and $E \cong T$ for some $0 \ne T \in \Coh_0(X)$.
\end{enumerate}
\end{prop}


\begin{prop}[{\cite[Proposition 3.2]{PiyFano3}}]
\label{prop:reflexivityatminus1place}
Let $E \in \HN^{\nu}_{H, B, \alpha}((-\infty,+\infty))$. Then
 $\hH^{-1}(E)$ is a reflexive sheaf.
\end{prop}

%

Let us recall the following slope bounds from \cite{PT} for cohomology sheaves of complexes in the abelian category  $\bB_{H,B}$. 
\begin{prop}
\label{prop:slope-bounds} 
Let $E \in \bB_{H, B}$.
Then we have the following:
 \begin{enumerate}[label=(\arabic*)]
\item if $E \in \HN^{\nu}_{H, B, \alpha}((-\infty, 0)) $,  then 
$\hH^{-1}(E) \in \HN^{\mu}_{H, B} ((-\infty, -  \alpha))$;

\item if $E \in \HN^{\nu}_{H, B}((0, +\infty)) $,  then 
$\hH^{0}(E) \in \HN^{\mu}_{H, B} ((  \alpha, +\infty])$;
and

\item if $E$ is tilt semistable with $\nu_{H,B, \alpha}(E) =0$, then
\begin{enumerate}[label=(\roman*)]
\item $\hH^{-1}(E) \in \HN^{\mu}_{H, B} ((-\infty, - \alpha])$ with equality
$\mu_{H,B}(E_{-1}(E)) = -\alpha $ holds
  if and only if $H^2 \ch_2^{B - \alpha H}(\hH^{-1}(E)) = 0$, that is 
  when $\Del_{H,B}(\hH^{-1}(E))=0$,
   and
\item when $\hH^{0}(E)$ is torsion free
 $\hH^{0}(E) \in \HN^{\mu}_{H, B} ([ \alpha , +\infty))$ with equality
$\mu_{H,B}(\hH^{0}(E)) = \alpha $ holds
  if and only if $H^2 \ch_2^{B+   \alpha H}(\hH^{0}(E)) =0$, that is 
  when $\Del_{H,B}(\hH^{0}(E))=0$.
\end{enumerate}
\item Let $E$ be $\nu_{H,B, \alpha}$-stable with $\nu_{H,B, \alpha}(E) =0$. Then 
$$
H^2\ch_1^{B + \alpha H}(E) \ge 0, \ \text{ and } \ H^2\ch_1^{B - \alpha H}(E) \ge 0.
$$
\end{enumerate}
\end{prop}
\begin{proof}
(1), (2) and (3) follows from $t=0$ case of \cite[Proposition 3.13]{PT}. 
(4) follows from (3) or from  $t=0$ case of \cite[Proposition 3.6]{PiyFano3}.
\end{proof}

First we recall the definition of a minimal object in an arbitrary abelian category.

\begin{defi}
\label{defi:minimal-objects}
\rm
Let $\mathcal{C}$ be an abelian category. Then a non-trivial object $A \in \mathcal{C}$ is
said to be a {\it minimal object} if
$0 \to E \to A \to F \to 0$ is a short exact sequence in $\mathcal{C}$ then
$E \cong 0$ or $F \cong 0$. That is, $A \in \cC$ is minimal when $A$ has no proper subobjects in $\cC$.
\end{defi}

\begin{defi}
\label{defi:double-tilt-minimal-objects}
\rm
Let $\mM_{H,B, \alpha}$ be the class of all objects $E \in \bB_{H , B,\alpha}$ such that
\begin{enumerate}[label=(\roman*)]
\item $E$ is $\nu_{H,B,\alpha}$-stable,
\item $\nu_{H,B, \alpha}(E) = 0$, and
\item $\Ext_{\SX}^1(\oO_x , E) = 0$ for any skyscraper sheaf $\oO_x$ of $x \in X$.
\end{enumerate}
\end{defi}

\begin{lem}[{\cite[Lemma 2.3]{MP1}}]
\label{prop:minimal-objects-threefold-hearts}
The following objects are minimal in  $\aA_{H,B,\alpha}$:
\begin{enumerate}
\item the skyscraper sheaves $\oO_x$  of any $x \in X$, and
\item objects which are isomorphic to $E[1]$, where $E \in \mM_{H,B,\alpha}$.
\end{enumerate}
\end{lem}

\begin{prop}[{\cite[Proposition 7.4.1]{BMT}}]
\label{prop:trivial-discriminant-tilt-stable-objects}
Let $E$ be a $\mu_{H, B}$-stable locally free sheaf on $X$ with $\overline{\Delta}_{H,B}(E) =0$. 
Then either $E$ or $E[1]$ in $\bB_{H,B}$ is  $\nu_{H,B, \alpha}$-stable.
\end{prop}
\begin{exam}
\label{example:minimal-objects-line-bundles}
\rm
Let $L$ be a line bundle on the smooth projective threefold $X$. 
Let $D = c_1(L)$.
By direct computation we have $\overline{\Delta}_{H, D \pm \alpha H}(L) =  0$ for any $\alpha >0$.
So by Proposition~\ref{prop:trivial-discriminant-tilt-stable-objects}, 
$L \in \bB_{H, D - \alpha H}$  and $L[1] \in \bB_{H, D + \alpha H}$
are tilt stable objects. 
Moreover, one can check that   $\nu_{H, D- \alpha H, \alpha}(L) =0$ and 
$\nu_{H, D + \alpha H, \alpha}(L[1]) =0$.
So by Lemma~\ref{prop:minimal-objects-threefold-hearts}, 
\begin{align*}
L[1] \in \aA_{H, D- \alpha H, \alpha}, \ \text{ and } \ L[2] \in \aA_{H, D + \alpha H, \alpha}
\end{align*}
  are  minimal objects.
\end{exam}
\begin{exam}
\label{example:minimal-objects-semihomogeneous-bundles}
\rm
Let $X$ be an abelian threefold. 
Let $\lx \in \NS_{\QQ}$ be an ample class. From Lemma \ref{prop:semihomo-numerical}--(2), for any $D \in \NS_{\QQ}(X)$ there exists  stable semihomogeneous bundles 
$E $ on $X$ with 
$$
D = c_1(E)/ \rk(E). 
$$
Moreover, $\ch(E) = \rk(E) e^{D}$.
By direct computation one can check that
$\overline{\Delta}_{\lx, D +\pm \alpha \lx}(E) =  0$ for any $\alpha >0$.
So by Proposition~\ref{prop:trivial-discriminant-tilt-stable-objects}, 
$E \in \bB_{\lx, D - \alpha \lx}$  and $E[1] \in \bB_{\lx, D + \alpha \lx}$
are tilt stable objects. 
Moreover, one can check that 
 $\nu_{\lx, D- \alpha \lx, \alpha}(E) =0$ and 
$\nu_{\lx, D + \alpha \lx, \alpha}(E[1]) =0$.
So by Lemma~\ref{prop:minimal-objects-threefold-hearts}, 
\begin{align*}
E[1] \in \aA_{\lx, D- \alpha \lx, \alpha}, \ \text{ and } \ E[2] \in \aA_{\lx, D + \alpha \lx, \alpha}
\end{align*}
  are  minimal objects.
\end{exam}


\begin{note}
\label{prop:BG-ineq-for-tilt-stable-trivial-discriminant}
\rm
The tilt stable objects associated to minimal objects in Examples
\ref{example:minimal-objects-line-bundles} and \ref{example:minimal-objects-semihomogeneous-bundles} clearly satisfy the corresponding Bogomolov-Gieseker type inequalities in Conjecture \ref{prop:BG-ineq-conjecture}.
\end{note}
Let us  reduce the requirement of Bogomolov-Gieseker type inequalities to the tilt stable objects in $\mM_{H,B, \alpha}$ (see Definition \ref{defi:double-tilt-minimal-objects}).
First we need the following proposition.
\begin{prop}[{\cite[Proposition 3.5]{LM}}]
\label{prop:some-tilt-stable-extensions}
Let
$0 \to E \to E' \to Q \to 0$
be  a non splitting short exact sequence  in $\bB_{H,B}$ with
 $Q \in \Coh_{0}(X)$, $\Hom_{\SX}(\oO_x, E') = 0$ for any $x \in X$, and
 $H^2 \ch_1^B(E) \ne 0$. If $E$ is $\nu_{H,B, \alpha}$-stable then $E'$
 is $\nu_{H,B, \alpha}$-stable.
\end{prop}
\begin{lem}[{\cite[Proposition 2.9]{MP1}}]
\label{prop:reduction-BG-ineq-class}
Let $E \in  \bB_{H,B}$ be $\nu_{H,B,\alpha}$ stable with $\nu_{H,B,\alpha}(E)=0$.
Then there exists 
$E' \in \mM_{H,B,\alpha}$
(that is $E'[1]$ is a minimal object in $\aA_{H,B,\alpha}$)  such that
$$
0 \to E \to E' \to Q \to 0
$$
is a short exact sequence  in $\bB_{H,B}$ for some $Q \in \Coh_{0}(X)$.

Since we have $\ch_{3}^B(Q) - \frac{1}{6}\alpha^2 \ch_1^{B}(Q) =  \ch_{3}(Q) \ge 0$,
$E$ satisfies the Bogomolov-Gieseker type inequality in Conjecture \ref{prop:BG-ineq-conjecture}
if $E' \in \mM_{H,B, \alpha}$   satisfies the corresponding
inequality.
\end{lem}
\subsection{Fourier-Mukai theory }
\label{sec:FM-theory}
Let us quickly recall some of the important notions in Fourier-Mukai theory.
Further details can be found in \cite{HuyFMTBook}.

Let $X,Y$ be smooth projective varieties and let $p_i$, $i=1,2$ be the  projection
maps from $X \times Y$ to $X$ and $Y$, respectively.
The {\it Fourier-Mukai functor} $\Phi_{\eE}^{\SX \to \SY}:
D^b(X) \to D^b(Y)$ with kernel $\eE \in D^b(X \times Y)$
is defined by
$$
\Phi_{\eE}^{\SX \to \SY}(-) = \RRR p_{2*} (\eE \stackrel{\textbf{L}}{\otimes} p_1^*(-)).
$$

Let
$ \eE_L =  \eE^\vee \stackrel{\BL}{\otimes} p_2^*\omega_Y \, [\dim Y] $, and 
$\eE_R = \eE^\vee \stackrel{\BL}{\otimes} p_1^* \omega_X \, [\dim X]$.
We have the following adjunctions (see \cite[Proposition 5.9]{HuyFMTBook}):
$$
\Phi_{\eE_L}^{\SY \to \SX} \dashv \Phi_{\eE}^{\SX \to \SY}  \dashv \Phi_{\eE_R}^{\SY \to \SX}.
$$

When $\Phi_{\eE}^{\SX \to \SY}$ is an equivalence of the derived categories,
usually it is called a {\it Fourier-Mukai transform}.
On the other hand by Orlov's Representability Theorem (see \cite[Theorem 5.14]{HuyFMTBook}),
any equivalence between $D^b(X)$ and $D^b(Y)$ is isomorphic
to a Fourier-Mukai  transform $\Phi_{\eE}^{\SX \to \SY}$ for some $\eE \in D^b(X \times Y)$.

Any Fourier-Mukai  functor $\Phi_{\eE}^{\SX \to \SY}: D^b(X) \to D^b(Y)$ induces a linear map
$\Phi^{\SH}_{\eE} : H^{2*}_{\alg}(X, \QQ)  \to   H^{2*}_{\alg}(Y, \QQ)$, usually 
called the cohomological Fourier-Mukai  functor,
 and it is a linear isomorphism when $\Phi_{\eE}^{\SX \to \SY}$ is a  Fourier-Mukai  transform.
The induced transform fits into the following commutative diagram, due to the Grothendieck-Riemann-Roch theorem.
$$
\xymatrixcolsep{4.5pc}
\xymatrixrowsep{2.25pc}
\xymatrix{
 D^b(X)  \ar[d]_{[-]} \ar[r]^{\Phi_{\eE}^{\SX \to \SY}}  &   D^b(Y) \ar[d]^{[-]} \\
 K(X)  \ar[d]_{v_X(-)}  \ar[r]^{\Phi^K_{\eE}}  &   K(Y) \ar[d]^{v_Y(-)} \\
 H^{2*}_{\alg}(X, \QQ) \ar[r]^{\Phi^{\SH}_{\eE}}  &   H^{2*}_{\alg}(Y, \QQ)
}
$$
Here $v_Z(-) = \ch(-) \sqrt{\text{td}_Z}$ is the Mukai vector map, where
$\ch: K(Z) \to H^{2*}_{\alg}(Z, \QQ)$ is the  Chern character map and $\text{td}_Z$ is the Todd class of $Z$.

Let  $v \in H^{2*}_{\alg}(X, \QQ)$  be a Mukai vector. Then $v= \sum_{i=0}^{\dim X} v_i$ for $v_i \in H^{2i}_{\alg}(X, \QQ)$ and the Mukai dual of
$v$ is  defined by $v^\vee = \sum_{i=0}^{\dim X} (-1)^i v_i$.
A symmetric bilinear form $\langle - , - \rangle_{\SX}$ called \textit{Mukai pairing} is defined by the formula
\begin{equation*}
\langle v, w \rangle_{\SX} = - \int_X v^\vee \cdot w \cdot e^{{c_1(X)}/{2} }.
\end{equation*}

Note that for an abelian variety $X$, $\text{td}_X =1$ and $c_1(X) =0$. Hence the Mukai vector $v(E)$ of $E \in D^b(X)$ is the same as its Chern character $\ch(E)$.

Due to Mukai and C\u{a}ld\u{a}raru-Willerton, 
for any $u \in H^{2*}_{\alg}(Y, \QQ)$ and $v \in H^{2*}_{\alg}(X, \QQ)$ we have  
\begin{equation}
\label{eqn:Mukai-pairing-isometry}
\left\langle \Phi^{\SH}_{\eE_L}(u) \, , \, v \right\rangle_{\scriptscriptstyle X}
= \left\langle  u \, , \, \Phi^{\SH}_{\eE }(v) \right\rangle_{\scriptscriptstyle Y}
\end{equation}
(see \cite[Proposition 5.44]{HuyFMTBook}, \cite{CW}).
\subsection{Abelian varieties}
\label{sec:abelian-varieties}
Over any field, an {\it abelian variety} $X$  is a complete group variety, that is
 $X$ is an algebraic variety equipped with the maps
 $X \times X \to X, \ (x,y) \mapsto x+y$  (the group law), and
$X \to X, \ x \mapsto -x$ (the inverse map),
together with the identity element $e \in X$.
For $a \in X$, the morphism $t_a : X \to X$ is defined by
$t_a  : x \mapsto x + a$.
Over the field of complex numbers, an abelian variety is a complex torus with the structure of a projective algebraic variety.

Let $\Pic^0(X)$ be  the subgroup of the abelian group $\Pic(X)$ consisting of elements represented by the line bundles which are algebraically equivalent to zero, and the corresponding quotient $\Pic(X)/ \Pic^0(X)$ is the N\'eron-Severi group 
$\NS(X) $. 
The group $\Pic^0(X)$ is naturally isomorphic to an abelian variety called the \textit{dual abelian variety} of $X$,  denoted by $\HX$.

The \textit{Poincar\'e line bundle} $\pP$ on the product $X \times \HX$ is the
uniquely determined line bundle satisfying
(i) $\pP_{X \times \{\widehat{x}\}} \in \Pic(X)$ is represented by $\widehat{x} \in \HX$, and
(ii) $\pP_{\{e\} \times \HX } \cong \oO_{\HX}$.
In \cite{MukFMT}, Mukai proved that the Fourier-Mukai functor $\Phi^{\SX \to \SHX}_{\pP}: D^b(X) \to D^b(\HX)$ is an equivalence of the derived categories, that is a Fourier-Mukai transform.

A vector bundle $E$ on an abelian variety $X$ is called \textit{homogeneous} if we have
$t_x^*E \cong E$ for all $x \in X$. A vector bundle $E$ on $X$ is homogeneous if and only if $E$
can be filtered by line bundles from $\Pic^0(X)$ (see \cite{MukSemihomo}).
We call a vector bundle $E$ is \textit{semihomogeneous} if for every $x \in X$ there exists a flat line bundle $\pP_{X \times \{\widehat{x}\}}$ on $X$
such that $t_x^*E \cong E \otimes  \pP_{X \times \{\widehat{x}\}}$.
A vector bundle $E$ is called \textit{simple} if we have $\End_{\SX}(E) \cong \CC$. 

\begin{lem}[{\cite[Theorem 5.8]{MukSemihomo}}]
\label{prop:Mukai-semihomognoeus-properties}
Let $E$ be a simple vector bundle on an abelian variety $X$. Then the following conditions are equivalent:
\begin{enumerate}[label=(\arabic*)]
\item $\dim H^1(X, \calEnd(E))=g$,
\item $E$ is semihomogeneous,
\item $\calEnd(E)$ is a homogeneous vector bundle.
\end{enumerate}
\end{lem}

\begin{lem}[{\cite{MukSemihomo, Orl}}]
\label{prop:semihomo-numerical}
We have the  following about simple semihomogeneous bundles:
\begin{enumerate}[label=(\arabic*)]
\item 
A rank $r$ simple semihomogeneous bundle $E$ has the Chern character 
\begin{equation*}
\label{semihomochern}
\ch(E) = r \ e^{c_1(E)/r}.
\end{equation*}
\item 
For any $D_{\SX} \in \NS_{\QQ}(X)$, there  exists simple semihomogeneous bundles $E$ on $X$ with 
$\ch(E) = r \, e^{D_{\SX}}$ for some $r \in \ZZ_{>0}$. 
\item 
Let  $E$ be a semihomogeneous bundle on $X$. Then $E$ is Gieseker semistable with respect to any ample bundle $L$, and if $E$ is simple then it is slope stable with respect to $c_1(L)$.
\end{enumerate}
\end{lem}
See \cite{Orl} for further details.

The image of an ample line bundle  $L$ on $X$ under the Fourier-Mukai transform $\Phi_{\pP}^{\SX \to \SHX}$ is
$$
\Phi_{\pP}^{\SX \to \SHX} (L) \cong \HL
$$
for some rank $\chi(L) = c_1(L)^g/g!$ semihomogeneous bundle $\HL$. Here 
$g = \dim X$.
Moreover, $-c_1(\HL)$ is an ample divisor class on $\HX$. See  \cite{BL-polarization}  for further details. Therefore, we have the following:

\begin{lem}[{\cite{BL-polarization}}]
\label{classicalcohomoFMT}
Let $\lx \in \NS_{\QQ}(X)$ be an ample class on $X$, and let $g = \dim X$.
 Under the induced cohomological transform $\Phi_{\pP}^{\SH}: H^{2*}_{\alg}(X, \QQ) \to H^{2*}_{\alg}(\HX,\QQ)$ of  $\Phi_{\pP}^{\SX \to \SHX}$ we have
\begin{equation*}
\Phi_{\pP}^{\SH}(e^{\lx}) = ({\lx^g}/{g!}) \, e^{-\lhx}
\end{equation*}
 for some ample class $\lhx \in \NS_{\QQ}(\HX)$, satisfying 
\begin{equation*}
({\lx^g}/{g!}) ({\lhx^g}/{g!}) =1.
\end{equation*}
Moreover, for each $0 \le i \le g$,
\begin{equation*}
\Phi_\pP^{\SH}\left( \frac{ \ell_{\SX}^i}{i!} \right) = \frac{(-1)^{g-i}  \ell_{\SX}^g}{g! (g-i)!} \, \ell_{\SHX}^{g-i}. 
\end{equation*}
\end{lem}


\subsection{Some sheaf theory}
\label{s2.8}
In this paper, we shall encounter reflexive sheaves at several occasions, and so we recall some of the key properties of them.

Let $X$ be a smooth projective variety of dimension $n$.

Any coherent sheaf $E$ on $X$ admits a \textit{locally free resolution} of length $n$.
In other words, $E$ fits into an exact sequence:
$$
0 \to F_n \to \cdots \to F_1 \to F_0 \to E \to 0
$$
for some locally free sheaves $F_i$ on $X$.


For a coherent sheaf $E$ on $X$,  its dual is  $E^* = \calHom(E, \oO_X)$.
There is a natural map from any $E \in \Coh(X)$ to its double dual $E^{**}$, $E \to E^{**}$.
If this map is an isomorphism then $E$ is called a \textit{reflexive} sheaf.
When $E$ is a torsion free sheaf, $E$ injects into its double dual.

\begin{lem}[{\cite[Lemma 1.1.2]{OSS}}]
\label{prop:dim-of-supp-Ext}
For any coherent sheaf $E$ on $X$ we have
$$
\dim \Supp \left(\calExt^{i}(E, \oO_X)\right)\le (n-i), \text{ for all i}.
$$
\end{lem}

\begin{defi}
\label{defi:singularity-set}
The \textit{singularity set} $\Sing(E)$ of a coherent sheaf $E \in \Coh(X)$ is defined as the locus where $E$ is not locally free, that is
$$
\Sing(E) = \{ x \in X : \Ext_{\SX}^1(E, \oO_x) \ne 0 \}.
$$
\end{defi}
This coincides with
$$
S_{n-1}(E) = \bigcup_{i = 1}^n \Supp \left(\calExt^{i}(E, \oO_X)\right).
$$
See \cite[Chapter 2]{OSS} for further details.

We collect some of the useful results about reflexive sheaves as follows.
\begin{lem}
\label{prop:reflexive-sheaf-results}
We have the following:
\begin{enumerate}[label=(\arabic*)]
\item if $E$ is a reflexive sheaf then $\dim \Sing(E) \le n-3$;
\item a coherent sheaf $E$ is reflexive if and only if it fits into a short exact sequence
$$
0 \to E \to F \to G \to 0
$$
in $\Coh(X)$ for a locally free sheaf $F$ and a torsion free sheaf $G$;
\item any $E \in \Coh(X)$ fits into an exact sequence
$$
0 \to T \to E \to E^{**} \to Q \to 0
$$
in $\Coh(X)$, where $T$ is the maximal torsion subsheaf of $E$ and  $Q$ is a torsion sheaf supported in a subscheme of at least codimension $2$;
\item for any $E \in \Coh(X)$, its dual $E^*$ is a reflexive sheaf;
\item any rank one reflexive sheaf is locally free, that is a line bundle.
\end{enumerate}
\end{lem}
\begin{proof}
See  Propositions 1.1, 1.3, 1.9 and Corollary 1.2 of \cite{HarshorneReflexive} for proofs of (2), (1), (5) and (4).
The claim in (3) is an easy exercise.
\end{proof}

When $\dim X =3$, one can easily prove the following result which is useful in this  paper to identify reflexive sheaves.
\begin{lem}
\label{prop:reflexive-sheaf-threefold}
A coherent sheaf $E$ on a smooth projective threefold $X$ is reflexive if and only if
\begin{enumerate}[label=(\roman*)]
\item $\Ext^1_{\SX}(\oO_x, E) = 0$ for all $x \in X$, and
\item $\Ext^2_{\SX}(\oO_x, E) \ne 0$ for finitely many $x \in X$.
\end{enumerate}
\end{lem}

The following result of Simpson is very important for us.

\begin{lem}[{\cite[Theorem 2]{Simpson}}]
\label{prop:Simpson-result-trivial-disciminant}
Let $X$ be a smooth projective variety of dimension $n \ge 3$. Let $L$ be an ample line bundle on $X$ and 
let $H$ be $c_1(L)$.
Let $E$ be a slope semistable reflexive sheaf on $X$ with respect to $H$ such that $H^{n-1} \ch_1(E)=H^{n-2} \ch_2(E) =0$.
Then all the Jordan-H\"{o}lder slope stable factors of $E$ are 
locally free sheaves which have vanishing Chern classes.
\end{lem}

\section{Cohomological Fourier-Mukai Transforms and Polarizations}
\label{sec:cohomological-FMT}

Let $Y$ be a $g$-dimensional abelian variety. Let us fix a class $D_{\SY} \in \NS_{\QQ}(Y)$. 
Let $X$ be the fine moduli space of rank $r$ simple semihomogeneous bundles $E$ on $Y$ with $c_1(E)/r =D_{\SY}$. Due to Mukai $X$ is a $g$-dimensional abelian variety. Let $\eE$ be the associated universal bundle on $X \times Y$; so by 
Lemma  \ref{prop:semihomo-numerical}--(1) we have 
$$
\ch(\eE_{\{x\} \times Y}) = r \, e^{D_{\SY}}.
$$
 Let $\Phi_\eE^{\SX  \to \SY} : D^b(X) \to D^b(Y)$ be the corresponding Fourier-Mukai transform from $D^b(X)$ to $D^b(Y)$ with kernel $\eE$.  
 Then its quasi inverse is given by 
$\Phi_{\eE^\vee}^{\SY \to \SX}[g]$. Again, by 
Lemma  \ref{prop:semihomo-numerical}--(1) we have
$$
\ch(\eE_{X \times \{y\}} ) = r \, e^{D_{\SX}}
$$
for some $D_{\SX} \in \NS_{\QQ}(X)$.

\begin{defi}
\label{defi:polarization}
\rm
A \textit{polarization} on $X$ is by definition the first Chern class $c_1(L)$ of an ample line bundle $L$ on $X$. 
However, it is usual to say the line bundle $L$ itself a polarization. 
\end{defi}

Let $a \in X$ and $b \in Y$.
Consider the Fourier-Mukai functor $\Gamma$ from $D^b(X)$ to $D^b(\HY)$ defined by 
$$
\Gamma = \Phi_\pP^{\SY  \to \SHY} \circ \eE_{\{a\}\times Y}^* \circ \Phi_\eE^{\SX  \to \SY} \circ \eE_{X \times \{b\}}^*\, [g],
$$
where $\eE_{\{a\}\times Y}^*$ denotes the functor $\eE_{\{a\}\times Y}^* \otimes(-)$ and similar for $ \eE_{X \times \{b\}}^*$. 
Let $\widehat \Gamma: D^b(\HY) \to D^b(X) $ be the  Fourier-Mukai functor defined by 
$$
\widehat \Gamma =    \eE_{X \times \{b\}} \circ \Phi_{\eE^\vee}^{\SY\to \SX  }  \circ \eE_{\{a\}\times Y} \circ\Phi_{\pP^\vee}^{ \SHY \to \SY }\, [g].
$$
Then $\widehat \Gamma $ and $\Gamma$ are adjoint functors to each other. By direct computation, 
$\Gamma (\oO_{x}) = \oO_{Z_x}$ for some $0$-subscheme $Z_x \subset \HY$, and 
$\Gamma (\oO_{\widehat y}) = \oO_{Z_{\widehat y}}$ for some $0$-subscheme $Z_{\widehat y} \subset  X$;
where the lengths of $Z_x$ and $Z_{\widehat y}$ are $r^3$ and $r$ respectively. 
Therefore, the Fourier-Mukai kernel  of $\Gamma$ is $\fF \in  \Coh_g(X \times \HY)$, with 
$\fF^{\vee} \cong \calExt^g(\fF, O_{X \times \HY})[-g]$.
So $\Gamma (\Coh_{i}(X)) \subset \Coh_i(\HY)$ and $\widehat \Gamma (\Coh_i (\HY)) \subset \Coh_i(X)$ for all $i$. 
Also by direct computation,  $\Gamma(\oO_X)$ and $\widehat \Gamma (\oO_{\HY}) $ are homogeneous bundles of rank $r$ and $r^3$ respectively. 

Let $\lx \in \NS_{\QQ}(X)$ be an ample class.
\begin{prop}
\label{prop:support-cohomological-result}
Under the induced cohomological map $ \Gamma^{\SH}: H^{2*}_{\alg}(X,\QQ) \to H^{2*}_{\alg}(\HY,\QQ) $,
$$
\Gamma^{\SH}(e^{\ell_{\SX}}) = r\  e^{\ell_{\SHY}}, 
$$
for some ample class $\ell_{\SHY} \in \NS_{\QQ}(\HY)$ satisfying $r^2 \, {\lx^g} =  {\lhy^g}$. 
Hence, under 
the induced cohomological map 
$ \widehat \Gamma^{\SH}: H^{2*}_{\alg}(\HY,\QQ) \to H^{2*}_{\alg}(X,\QQ) $,
$$
\widehat \Gamma^{\SH}(e^{\ell_{\SHY}}) = r^3 \  e^{\ell_{\SX}}.
$$
Moreover, for each $0 \le i \le g$,
$$
\Gamma^{\SH}( \ell_{\SX}^i) = r  \, \ell_{\SHY}^i, \  \ \
\widehat \Gamma^{\SH}(\lhy^i) = r^3 \,  \lx^i.
$$
\end{prop}
\begin{proof}
Since $\Gamma (\Coh_{i}(X)) \subset \Coh_i(\HY)$, for any $E$ we have
$$
\Ga^{\SH} (\ch_{\ge j}(E)) = \ch_{ \ge j}(\Gamma (E)).
$$
Here $\ch_{\ge j } = (0,\cdots, \ch_j, \ch_{j+1}, \cdots, \ch_g)$.
Therefore,
\begin{align*}
\Gamma^{\SH}(e^{ \ell_{\SX}})  &  = \Gamma^{\SH}\left(e^{0} + \ch_{\ge 1}(e^{\ell_{\SX}}) \right)  
 =\Gamma^{\SH}\left(e^{0} \right) + \Gamma^{\SH}\left( \ch_{\ge 1}(e^{\ell_{\SX}}) \right)\\
& = \ch(\Gamma(\oO_X)) + \Gamma^{\SH}\left( \ch_{\ge 1}(e^{\ell_{\SX}}) \right) 
 = (r, 0, \cdots,0) + (0, *, \cdots, *) \\ 
 & = (r, * , \cdots, *).
\end{align*}

For any $k \in \ZZ$, There exists a semihomogeneous bundle  $E_k$ with $k \lx = c_1(E_k)/\rk(E_k)$.
 Under the transform $\Gamma (E_k)$ is  also a semihomogeneous bundle 
such that $c_1(\Gamma (E_k))/ \rk(\Gamma (E_k)) = D_k$ for some $D_k \in \NS_{\QQ}(\HY)$.

So we deduce
$$
\Gamma^{\SH}(e^{ \ell_{\SX}}) = r \  e^{\ell_{\SHY}}, 
$$
for some  class $\ell_{\SHY} \in \NS_{\QQ}(\HY)$.

Moreover, for any $k$
\begin{align*}
\Gamma^{\SH}(e^{k \ell_{\SX}})  & = \Gamma^{\SH}\left(k e^{\lx} - (k-1) e^0 + (0,0, *, \cdots, *) \right) \\
& = k  r e^{\ell_{\SHY}} - (k-1)r e^0 + (0, 0, *, \cdots, *) = (r,  r k \lhy , *, \cdots, *).
\end{align*}
So it has to be equal to $re^{k \lhy}$. 

For any $0 \le i  \le  g$, we can write $\lx^i$ as a $\QQ$-linear combination of $\{e^0, e^{\ell_{\SX}}, \cdots, e^{g \ell_{\SX}}\}$.
Since $\Gamma^{\SH}(e^{k \ell_{\SX}})  = re^{k \lhy}$, we have 
$\Gamma^{\SH}( \ell_{\SX}^i) = r  \, \ell_{\SHY}^i$. 

Similarly, we can prove the results involving $\widehat \Gamma^{\SH}$.

Now let us prove that the class $\ell_{\SHY}$ is ample.

 For any $0 \le j \le g$, let $\HY^{(j)} \subset \HY$ be a closed $j$-dimensional subscheme of $\HY$.  
Then we have
\begin{align*}
\int_{\HY} \ell_{\SHY}^{g-j} \cdot [\HY^{(j)}] 
& =  \frac{1}{r^4}  \int_{\HY} \ell_{\SHY}^{g-j} \cdot \Gamma^{\SH} \widehat \Gamma^{\SH} [\HY^{(j)}] \\
& =  \frac{(-1)^{g-j}}{r^4} \left\langle  \ell_{\SHY}^{g-j} , \,  \Gamma^{\SH} \widehat \Gamma^{\SH} [\HY^{(j)}]  \right\rangle_{\SHY} \\
& =  \frac{(-1)^{g-j}}{r^4} \left\langle  \widehat \Gamma^{\SH}(\ell_{\SHY}^{g-j}) , \,   \widehat \Gamma^{\SH} [\HY^{(j)}]  \right\rangle_{\SX}, \ \ \ \text{by \eqref{eqn:Mukai-pairing-isometry}} \\
& =  \frac{(-1)^{g-j}}{r} \left\langle \ell_{\SX}^{g-j}  , \,   \widehat \Gamma^{\SH} [\HY^{(j)}]  \right\rangle_{\SX} \\
& = \frac{1}{r}  \int_{X} \ell_{\SX}^{g-j}  \cdot \widehat \Gamma^{\SH} [\HY^{(j)}] > 0,
\end{align*}
as $\widehat{\Gamma} \left(\oO_{\HY^{(j)}}\right) \in \Coh_{j}(X)$   and $\ell_{\SX}$ is an ample class.
Hence, from the Nakai-Moishezon criterion, $\ell_{\SHY}$ is an ample class on $\HY$.
 \end{proof}

 By Theorem \ref{classicalcohomoFMT}, under the induced cohomological map of $\Phi_{\pP^\vee}^{ \SHY \to \SY }$ we have
$$
\Phi_{\pP^\vee}^{\SH}(e^{\ell_{\SHY}}) =({\ell_{\SHY}^g}/{g!})  \, e^{- \ell_{\SY}},
$$
for some ample class $\ell_{\SY} \in \NS_{\QQ}(Y)$.

Let  $\Xi : D^b(X) \to D^b(Y)$ be the Fourier-Mukai functor defined by 
$$
\Xi = \eE_{\{a\}\times Y}^* \circ \Phi_\eE^{\SX  \to \SY} \circ \eE_{X \times \{b\}}^* =\Phi_{\pP^\vee}^{ \SHY \to \SY } \circ \Gamma.
$$
The image of $e^{\ell_{\SX}}$ under its induced cohomological transform $\Xi^{\SH}$  is 
$ ({r^3 \lx^g}/{g!}) \, e^{- \ell_{\SY}} $. Therefore, we deduce the following.

\begin{thm} 
\label{prop:general-cohomo-FMT}
 If $\lx \in \NS_{\QQ}(X)$ is an ample class then 
$$
 e^{- D_{\SY}} \, \Phi_{\eE}^{\SH} \, e^{-D_{\SX}} ( e^{\ell_{\SX}}) =  (r \, {\lx^g}/{g!}) \,  e^{-\ell_{\SY}},
$$
for some ample class $\ly \in \NS_{\QQ}(Y)$, satisfying $({\lx^g}/g!)({\ly^g}/g!)=  1/r^2$.
Moreover,  for each $0 \le i \le g$,
\begin{equation*}
 e^{- D_{\SY}} \, \Phi_{\eE}^{\SH} \, e^{-D_{\SX}} \left( \frac{ \ell_{\SX}^i}{i!} \right) = \frac{(-1)^{g-i}  r \, \ell_{\SX}^g}{g! (g-i)!} \, \ly^{g-i}. 
\end{equation*}
\end{thm}

This gives us the following:
\begin{thm}
\label{prop:derived-induce-polarization}
If the ample line bundle $L$ defines a polarization on $X$,  then
the line bundle $\det (\Xi(L))^{-1}$ is ample and so it defines a polarization on $Y$.
\end{thm}

Let us introduce the following notation:
\begin{nota}
\rm
Let $B, \lx \in \NS_{\QQ}(X)$. For $E \in D^b(X)$, the entries $v^{B,{\lx}}_i(E)$, $i=0, \ldots, g$, are defined by
$$
v^{B,{\lx}}_i(E) = i! \, \lx^{g-i} \cdot \ch^B_{i}(E).
$$
Here $\ch^B_{i}(E)$ is the $i$-th component of the $B$-twisted Chern character $\ch^B(E)= e^{-B} \ch(E)$. 

The vector $v^{B,{\lx}}(E)$ is defined by 
\begin{equation*}
v^{B,{\lx}}(E) = \left( v^{B,{\lx}}_0(E) , \ldots, v^{B,{\lx}}_g(E) \right).
\end{equation*}

We will denote an $g \times g$ anti-diagonal matrix with entries $a_k$, $k=1, \ldots,  g$ by
$$
\Adiag(a_1, \ldots, a_g)_{ij} : = \begin{cases}
a_k & \text{if } i=k, j=g+1-k \\
0 & \text{otherwise}.
\end{cases}
$$
\end{nota}

\begin{thm}
\label{prop:antidiagonal-rep-cohom-FMT}
If we consider $v^{-D_{\SX} ,\lx}, v^{D_{\SY},\ly}$ as column vectors, then 
$$
v^{D_{\SY},\ly}\left(\Phi_\eE^{\SX  \to \SY}(E) \right) = \frac{g!}{r \, \ell_{\SX}^g} \,  \Adiag\left(1,-1,\ldots, (-1)^{g-1}, (-1)^g\right)  \ v^{-D_{\SX}, \lx}(E).
$$
\end{thm}
\begin{proof}
The $i$-th entry of $v^{D_{\SY},\ly}\left(\Phi_\eE^{\SX  \to \SY}(E) \right) $ is
\begin{align*}
v^{D_{\SY},\ly}_{i} \left(\Phi_\eE^{\SX  \to \SY}(E) \right)
& = i! \, \ly^{g-i} \cdot \ch^{D_{\SY}}_{i}\left(\Phi_\eE^{\SX  \to \SY}(E) \right) \\
& = i!  \int_{Y} {\ell_{\SY}^{g-i}} \cdot \ch^{D_{\SY}}\left(\Phi_\eE^{\SX  \to \SY}(E) \right) \\
& = i!  \int_{Y} {\ell_{\SY}^{g-i}} \cdot e^{-D_{\SY}}\ch\left(\Phi_\eE^{\SX  \to \SY}(E) \right) \\
& =(-1)^{g-i}  i!  \left\langle {\ell_{\SY}^{g-i}} , \, e^{-D_{\SY}} \ch\left(\Phi_\eE^{\SX  \to \SY}(E) \right) \right\rangle_{\SY}\\
& =(-1)^{g-i}  i!  \left\langle {\ell_{\SY}^{g-i}} , \, e^{-D_{\SY}} \Phi_\eE^{\SH}(\ch(E))  \right\rangle_{\SY}\\
& =(-1)^{g-i}  i!  \left\langle {\ell_{\SY}^{g-i}} , \, e^{-D_{\SY}} \Phi_\eE^{\SH}e^{-D_{\SX}}(\ch^{-D_{\SX}}(E))  \right\rangle_{\SY}\\
& =(-1)^{g-i}  i!  \left\langle {\left(e^{-D_{\SY}} \Phi_\eE^{\SH}e^{-D_{\SX}}\right)^{-1}(\ell_{\SY}^{g-i}}) , \, \ch^{-D_{\SX}}(E) \right\rangle_{\SX}\\
& =\frac{ g! (g-i)! }{r \, \lx^g}  \left\langle \lx^i , \, \ch^{-D_{\SX}}(E) \right\rangle_{\SX}, \ \ \  \text{ from Theorem \ref{prop:general-cohomo-FMT} }\\
& =\frac{ (-1)^i g! (g-i)! }{r  \, \lx^g}  \int_X  \lx^i \cdot \ch^{-D_{\SX}}(E) \\
& =\frac{ (-1)^i g! (g-i)! }{r \, \lx^g}  \,  \lx^i \cdot \ch^{-D_{\SX}}_{g-i}(E) \\
& =\frac{ (-1)^i g! }{r  \, \lx^g}  \, v^{-D_{\SX},\lx}_{g-i} (E). 
\end{align*}
This completes the proof. 
\end{proof}
Let $\DD$ denote the derived dualizing functor $\RRR \calHom(-, \oO_X)$. 
The following   is a generalization of Mukai's result on classical  Fourier-Mukai  transform.
\begin{lem}[{\cite[Lemma 2.2]{PP}}] 
\label{prop:dual-FMT}
We have the isomorphism
$$
(\Phi^{\SX \to \SY}_{\eE^\vee} \circ \DD) [g]\cong  \DD \circ \Phi^{\SX \to \SY}_{\eE}.
$$
Here $\Phi^{\SX \to \SY}_{\eE^\vee}: D^b(X) \to D^b(Y)$ is the  Fourier-Mukai  transform 
from $X$ to $Y$ with the kernel $\eE^\vee$.
\end{lem}

This gives us the convergence of the following spectral sequence.
\begin{dualspecseq}
\label{Spec-Seq-Dual}
$$
\hH^p  \left( \Phi^{\SX \to \SY}_{\eE^\vee} \left( \calExt^{q+g} (E, \oO_X) \right) \right)
 \Longrightarrow \ \ ? \ \ \Longleftarrow 
 \calExt^{p+g}\left( \hH^{g-q} \left( \Phi^{\SX \to \SY}_{\eE}(E)\right), \oO_X \right)
$$
for $E \in \Coh(X)$.
\end{dualspecseq}

We have the following for the Fourier-Mukai transform $\Phi^{\SX \to \SY}_{\eE^\vee}: D^b(X) \to D^b(Y)$ :
\begin{prop}
\label{prop:dual-antidiagonal-rep-cohom-FMT}
If we consider $v^{D_{\SX} ,\lx}, v^{-D_{\SY},\ly}$ as column vectors, then 
$$
v^{-D_{\SY},\ly}\left(\Phi_\eE^{\SX  \to \SY}(E) \right) = \frac{g!}{r \, \ell_{\SX}^g} \,  \Adiag\left(1,-1,\ldots, (-1)^{g-1}, (-1)^g\right)  \ v^{D_{\SX}, \lx}(E).
$$
\end{prop}

\section{Stability Conditions Under  FM Transforms on Abelian Varieties }
\label{sec:stability-under-FMT}
\subsection{Action of FM transforms on Bridgeland Stability Conditions}
\label{sec:action-FMT-central-charge}
This section generalizes some of the similar results in \cite{MP2, PiyThesis}.

Recall that a Bridgeland stability condition $\sigma$ on a triangulated category
$\dD$ consists of a stability function $Z$ together with a slicing $\pP$
of $\dD$ satisfying certain axioms. Equivalently, one can define
$\sigma$ by giving a bounded t-structure on $\dD$ together with a
stability function $Z$ on the corresponding heart $\aA$ satisfying the
Harder-Narasimhan property. Then $\sigma$ is usually written as the pair $(Z, \pP)$
or $(Z, \aA)$. 

Let $\Upsilon:   \dD \to \dD'$ be an equivalence of triangulated categories, and let $W: K(\dD) \to \CC$ be a group homomorphism. Then
$$
\left( \Upsilon \cdot W \right) ([E]) = W \left( [ \Upsilon^{-1}(E) ] \right)
$$
defines an induced group morphism $ \Upsilon \cdot W$ in  $\Hom (K(\dD'),
\CC)$ by the equivalence $\Upsilon$. Moreover, this can be extended to a natural induced stability condition 
 on $\dD'$ by defining
$\Upsilon \cdot (Z, \aA) = (\Upsilon \cdot Z , \Upsilon(\aA))$.

Let $X, Y$ be two derived equivalent $g$-dimensional abelian varieties as in Section \ref{sec:cohomological-FMT},
 which is given by the Fourier-Mukai transform $\Phi_{\eE}^{\SX \to \SY} : D^b(X) \to D^b(Y)$. 
Let $\lx \in \NS_{\QQ}(X)$ be an ample class on $X$ and let $\ly \in \NS_{\QQ}(Y)$ be the induced ample class on $Y$ as in Theorem \ref{prop:general-cohomo-FMT}. 

Let $u$ be a complex number. 
Consider the 
function $Z_{-D_{\SX}+u\lx}: K(X) \to \CC$ defined by 
$$
Z_{-D_{\SX}+u\lx}(E)= -\int_{X} e^{-\left( -D_{\SX}+u\lx \right)}\ch(E) = 
\left\langle  e^{ -D_{\SX}+u\lx } , \ \ch(E)  \right\rangle_{\SX}.
$$
For $E \in D^b(Y)$ we have
\begin{align*}
\left( \Phi_{\eE}^{\SX \to \SY}  \cdot Z_{-D_{\SX}+u\lx}\right)(E) 
& = \left\langle  e^{ -D_{\SX}+u\lx } , \ \ch\left(\left(\Phi_{\eE}^{\SX \to \SY}\right)^{-1}(E)\right)  \right\rangle_{\SX} \\
& = \left\langle    e^{ -D_{\SX}+ u\lx } , \ \left(\Phi_{\eE}^{\SH}\right)^{-1} \left(\ch(E)\right)  \right\rangle_{\SX} \\
& = \left\langle  \Phi_{\eE}^{\SH} \left( e^{ -D_{\SX}+ u\lx }\right) , \ \ch(E)  \right\rangle_{\SY}\\
& = \left\langle e^{D_{\SY}} \left( e^{-D_{\SY}}  \Phi_{\eE}^{\SH} e^{ -D_{\SX}}\right) (e^{ u\lx })  , \ \ch(E)  \right\rangle_{\SY} \\
& =  (r \, \lx^g u^g /g!) \ \left \langle   e^{D_{\SY} -\ly/u}  , \ \ch(E)  \right\rangle_{\SY},
\end{align*}
 since by Theorem \ref{prop:general-cohomo-FMT},   
$e^{-D_{\SY}}  \Phi_{\eE}^{\SH} e^{ -D_{\SX}}(e^{ u\lx }) = (r \, \lx^g u^g /g!) \, e^{-\ly/u} $. So we have the following relation:

\begin{lem}
\label{prop:FMTact-central-charge}
We have $\Phi_{\eE}^{\SX \to \SY}  \cdot Z_{-D_{\SX}+u\lx} = \zeta\ Z_{D_{\SY} - \ly/u}$, for $\zeta = r \, \lx^g u^g /g!$.
\end{lem}

Assume there exist a stability condition for any complexified ample class  $-D_{\SX}+u\lx$ with a heart $\aA^{\SX}_{-D_{\SX}+u\lx}$ and a slicing $\pP^{\SX}_{-D_{\SX}+u\lx}$ associated to the central charge function $Z_{-D_{\SX}+u\lx}$.  
Furthermore, assume similar stability conditions exist on $Y$. 
From Lemma \ref{prop:FMTact-central-charge} for any $\phi \in \RR$,  \ $\zeta \, Z_{D_{\SY} - \ly/u}\left( \Phi_{\eE}^{\SX \to \SY}\left( \pP^{\SX}_{-D_{\SX}+u\lx}(\phi) \right)\right) \subset \RR_{>0} e^{i\pi \phi}$; that is 
$$
Z_{D_{\SY} - \ly/u}\left( \Phi_{\eE}^{\SX \to \SY}\left(  \pP^{\SX}_{-D_{\SX}+u\lx}(\phi)  \right)\right)  \subset \RR_{>0} e^{i\left(\pi \phi - \arg(\zeta)\right)}.
$$
So  we would expect
$$
 \Phi_{\eE}^{\SX \to \SY}\left( \pP^{\SX}_{-D_{\SX}+u\lx} (\phi) \right) = \pP^{\SY}_{D_{\SY}- \ly/u} \left( \phi - \frac{\arg (\zeta)}{\pi}\right),
$$
and so 
$$
 \Phi_{\eE}^{\SX \to \SY}\left( \pP^{\SX}_{-D_{\SX}+u\lx} ((0,\,1])\right) =  \pP^{\SY}_{D_{\SY}- \ly/u} \left(\left( -\frac{\arg (\zeta)}{\pi},  \, -\frac{\arg (\zeta)}{\pi} +1\right]\right).
$$
For $0\le \alpha <1$, 
$  \pP^{\SY}_{D_{\SY}- \ly/u} ((\alpha, \alpha+1]) =\left \langle   \pP^{\SY}_{D_{\SY}- \ly/u} ((0, \alpha])\,[1] , \,   \pP^{\SY}_{D_{\SY}- \ly/u}  ((\alpha, 1]) \right \rangle$ is a tilt of $  \aA^{\SY}_{D_{\SY}- \ly/u}  =   \pP^{\SY}_{D_{\SY}- \ly/u} ((0,1])$
 associated to a torsion theory coming from $Z_{D_{\SY}- \ly/u} $ stability. 
Therefore, one would expect $  \Phi_{\eE}^{\SX \to \SY} \left( \aA^{\SX}_{-D_{\SX} + u\lx } \right)$  is a
tilt of $\aA^{\SY}_{D_{\SY}- \ly/u}$ associated to a torsion theory coming from $Z_{D_{\SY}- \ly/u}$ stability, up to some shift. 

Moreover, for the Fourier-Mukai transform $\Phi_{\eE}^{\SX \to \SY}$ when $\zeta$ is real, that is, 
\begin{equation}
\label{eqn:condition-for-abelian-equivalence}
u^g \in \RR
\end{equation}
we would expect that the Fourier-Mukai transform $\Phi_{\eE}^{\SX \to \SY}: D^b(X) \to D^b(Y)$ gives the equivalence of associated stability condition hearts. We conjecturally formulate this for any dimensional abelian varieties in Section \ref{sec:cojectural-any-abelian}. 

\subsection{Very weak stability conditions}
\label{sec:very-weak-stability}
Let us recall the 
general arguments of very weak stability conditions.
We closely follow the notions as in \cite[Section 2]{PT}.

Let $\dD$ be a triangulated category, and 
$K(\dD)$ its Grothendieck group. 

\begin{defi}
\label{defi:weak-stability}
\rm
A \textit{very weak stability condition}
on $\dD$ is a pair $(Z, \aA)$, 
where $\aA$ is the heart of a bounded t-structure on $\dD$, 
and $Z: K(\dD) \to \mathbb{C}$ is a group homomorphism 
satisfying the following conditions: 
\begin{enumerate}
\item For any $E \in \aA$, we have
$Z(E) \in \mathbb{H} \cup \mathbb{R}_{\le 0}$.
Here $\mathbb{H}$ is the 
upper half plane $ \{ z \in \CC: \Imm Z >0\}$. 

\item The associated slope function $\mu:  \aA \to \mathbb{R} \cup \{+\infty\}$ is defined by 
$$
\mu (E) =
\begin{cases}
+ \infty & \ \text{if } \Imm Z(E)=0 \\
 -\frac{\Ree Z(E)}{\Imm Z(E)} & \ \text{otherwise},
\end{cases}
$$
and it satisfies the Harder-Narasimhan  property.  
\end{enumerate}
\end{defi}

We say that 
$E \in \aA$ is $\mu$-(semi)stable 
if for any non-zero subobject $F \subset E$
in $\aA$, 
we have the inequality:
$\mu(F) <(\le) \, \mu(E/F)$.

The  Harder-Narasimhan filtration of an object $E \in \aA$ is 
a chain of subobjects 
$0=E_0 \subset E_1 \subset \cdots \subset E_n=E$
 in $\aA$ such that each $F_i=E_i/E_{i-1}$ is 
$\mu$-semistable with 
$\mu(F_i)>\mu(F_{i+1})$. 
If such  Harder-Narasimhan filtrations exists for all objects in $\aA$,
we say that $\mu$ satisfies the  Harder-Narasimhan property.

For a given a very weak stability condition $(Z, \aA)$, 
we define its slicing on $\dD$ (see \cite[Definition~3.3]{BriStab})
\begin{align*}
\{\pP(\phi)\}_{\phi \in \mathbb{R}}, \
\pP(\phi) \subset \dD
\end{align*}
as in the case of Bridgeland 
stability conditions (see \cite[Proposition~5.3]{BriStab}). 
Namely, for $0<\phi \le 1$, 
the category $\pP(\phi)$ is defined to 
be 
\begin{align*}
\pP(\phi) =\{ 
E \in \aA : 
E \mbox{ is } \mu \mbox{-semistable with }
\mu(E)=-1/\tan (\pi \phi)\} \cup \{0\}.
\end{align*} 
Here we set $-1/\tan \pi =\infty$. 
The other subcategories are defined by setting
\begin{align*}
\pP(\phi+1)=\pP(\phi)[1].
\end{align*}
For an interval $I \subset \mathbb{R}$, 
we define $\pP(I)$ to be the smallest extension 
closed subcategory of $\dD$ which contains 
$\pP(\phi)$ for each $\phi \in I$. 
For $0 \le s \le 1$, the pair $(\pP((s, 1]), \pP((0, s]) )$ of subcategories of $\aA = \pP((0,1])$ is a torsion pair, and the corresponding tilt is $\pP((s, s+1])$. 

Note that  
the category $\pP(1)$ contains the 
following category 
\begin{align*}
\cC \cneq \{E \in \aA : Z(E)=0\}. 
\end{align*}
It is easy to check that $\cC$ 
is closed under subobjects and quotients in $\aA$. 
In particular, $\cC$ is an 
abelian subcategory of 
$\aA$. 
Moreover, 
the pair $(Z, \aA)$ gives a \textit{Bridgeland stability 
condition} on $\dD$ if $\cC=\{0\}$. 
\subsection{Conjectural stability conditions}
\label{sec:cojectural-any-abelian}
Let $X$ be a $g$-dimensional abelian variety with $g \ge 2$. 
Motivated by the constructions for smooth projective surfaces  (see \cite{BriK3, AB}) together with some observations in  Mathematical Physics, for  $X$, it is expected that 
the function defined by
\begin{equation*}
Z_{B + i  \omega}(-) = - \int_X e^{-B- i \omega} \ch(-)
\end{equation*}
is a central charge function of some geometric stability condition on $D^b(X)$ (see \cite[Conjecture 2.1.2]{BMT}). 
Here $B + i \omega \in \NS_{\CC}(X)$ is a complexified ample class on $X$, that is by definition $B, \omega \in \NS_{\RR}(X)$ with $\omega$ an ample class. 
By using the notion of very weak stability, let us conjecturally construct a heart for this central charge function.  

For $0\le k\le g$, we define the $k$-truncated Chern character by
$$
\ch_{\le k}(E) = (\ch_0(E), \ch_1(E), \ldots, \ch_k(E), 0, \ldots, 0),
$$
and the function $Z^{(k)}_{B + i \omega} : K(X) \to \CC$ by
\begin{equation*}
Z^{(k)}_{B + i \omega}(E) = - i^{n-k}\int_X e^{-B- i \omega} \ch_{\le k}(E).
\end{equation*}

The usual slope stability on sheaves gives the very weak stability condition $(Z^{(1)}_{B + i \omega}, \Coh(X))$. 
Moreover, we formulate the following:
\begin{conj} 
\label{prop:conjecture-stab-cond}
For each $1 \le k < g$, the pair $\sigma_k = (Z^{(k)}_{B + i \omega}, \aA^{(k)}_{B + i \omega})$ gives a very weak stability condition on $D^b(X)$, where the hearts $\aA^{(k)}_{B + i \omega}$, $1\le k \le g$  are defined by 
\begin{equation*}
 \left.\begin{aligned}
         & \aA^{(1)}_{B + i \omega} = \Coh(X) \\
         &\aA^{(k+1)}_{B + i \omega} = \pP_{\sigma_k}((1/2,\,3/2])
       \end{aligned}
  \ \right\}.
\end{equation*}
Moreover, the pair $\sigma_g = (Z_{B + i \omega}, \aA^{(g)}_{B + i \omega})$ is a Bridgeland stability condition on $D^b(X)$. 
\end{conj}

This is known to be true for abelian surfaces (\cite{BriK3, AB}) and abelian threefolds (\cite{MP1, MP2, PiyThesis, BMS}). 

\begin{rmk}
\label{prop:remark-conj-stab}
\rm
Although we assumed $X$ to be an abelian variety, the above Conjecture \ref{prop:conjecture-stab-cond} makes sense for any smooth projective variety.
In fact,  $(Z^{(1)}_{\omega,B}, \aA^{(1)}_{B + i \omega}  = \Coh(X))$ is a very weak stability condition for any variety and a Bridgeland stability condition for curves. By \cite{BriK3, AB}, $ (Z^{(2)}_{B + i \omega}, \aA^{(2)}_{B + i \omega} )$ is a Bridgeland stability condition for surfaces. 
In \cite{BMT}, the authors proved that the pair $(Z^{(2)}_{B + i \omega}, \aA^{(2)}_{B + i \omega} )$  is again a 
very weak stability condition for threefolds. Here the stability was called tilt slope stability.
The usual Bogomolov-Gieseker inequality for $Z^{(1)}_{B + i \omega}$ stable sheaves plays a crucial role in these proofs. Clearly the same arguments work for any higher dimensional varieties. Therefore, we  can always construct the category  $\aA^{(3)}_{B + i \omega}$ when $\dim X \ge 2$. In \cite{BMT}, the authors conjectured that this category is a heart of a Bridgeland stability condition with the central charge $Z_{B+i\omega}$. Moreover, they reduced it to prove Bogomolov-Gieseker type inequalities for $Z^{(2)}_{B + i \omega}$ stable objects $E \in \aA^{(2)}_{B + i \omega}$ with $\Ree Z^{(2)}_{B + i \omega}(E) =0$, and the strong form of this inequality is
\begin{equation*}
\label{BGineq}
\ch^B_3 (E) \le \frac{\omega^2}{18} \ch_1^B(E).
\end{equation*}
This is exactly Conjecture \ref{prop:BG-ineq-conjecture}. 
\end{rmk}

Let $X, Y$ be two derived equivalent $g$-dimensional abelian varieties as in Section \ref{sec:cohomological-FMT}, which is given by the Fourier-Mukai transform $\Phi_{\eE}^{\SX \to \SY} : D^b(X) \to D^b(Y)$. 
Let $\lx \in \NS_{\QQ}(X)$ be an ample class on $X$ and let $\ly \in \NS_{\QQ}(Y)$ be the induced ample class on $Y$ as in Theorem \ref{prop:general-cohomo-FMT}. 

By considering the complexified classes associated to the condition \eqref{eqn:condition-for-abelian-equivalence}, we conjecture the following for all  abelian varieties.

\begin{conj}
\label{prop:conjecture-equivalence-stab-hearts}
The Fourier-Mukai transform $\Phi_{\eE}^{\SX \to \SY}: D^b(X) \to D^b(Y)$ gives the equivalence of stability condition hearts conjecturally constructed in Conjecture \ref{prop:conjecture-stab-cond}:
$$
\Phi_{\eE}^{\SX \to \SY} [k]  \left(  \aA_{\Omega}  \right) = \aA_{\Omega'}.
$$
Here  $\Omega = -D_{\SX} + \lambda e^{i k\pi/g }\, \lx $ and $\Omega' = D_{\SY}  - (1/\lambda) e^{-i k \pi/g} \, \ly$
are complexified ample classes on $X$ and $Y$ respectively, for  any $k \in \{1, 2, \ldots, (g-1)\}$ and any $ \lambda \in \RR_{>0}$. 
\end{conj}
\begin{note}
\rm
This conjecture is known to be true for abelian surfaces and we discuss it in Section \ref{sec:equivalence-stab-hearts-surface}. 
Moreover, the main aim of the next sections  is to show this conjecture indeed holds on abelian threefolds; see  Theorem \ref{prop:equivalence-hearts-abelian-threefolds}. 
\end{note}
\section{Bogomolov-Gieseker Type Inequality on Abelian Threefolds}
Let $X, Y$ be derived equivalent abelian threefolds  and let $\lx, \ly$ be ample classes on them respectively as in Theorem \ref{prop:general-cohomo-FMT}.
\begin{nota}
\rm
Let $\Psi$ be the Fourier-Mukai transform $\Phi_{\eE}^{\SX \to \SY}$ from $X$ to $Y$ with kernel  $\eE$, and let $\HPsi =  \Phi_{\eE^\vee}^{\SY \to \SX}$.  
\end{nota}
\begin{prop}
\label{prop:imgainary-part-central-charge}
We have the following:
\begin{enumerate}[label=(\arabic*)]
\item For $E \in D^b(X)$, 
$$
\Imm Z_{\lx, -D_{\SX} + \frac{ \lambda}{2} \lx , \frac{ \lambda}{2} }(E) =\frac{ \lambda \sqrt{3}}{4} \left( v_2^{-D_{\SX}, \lx}(E)- \lambda v_1^{-D_{\SX}, \lx}(E) \right), 
$$
and  for $E \in D^b(Y)$, 
$$
\Imm Z_{\ly, D_{\SY} - \frac{1}{2 \lambda} \ly ,  \frac{1}{2 \lambda}}(E) =\frac{ \sqrt{3}}{4\lambda} \left( v_2^{D_{\SY}, \ly}(E) +\frac{1}{\lambda} v_1^{-D_{\SY}, \ly}(E) \right).
$$
\item For $E \in D^b(Y)$, 
$$
\Imm Z_{\lx, -D_{\SX} + \frac{ \lambda}{2} \lx , \frac{ \lambda}{2} }(\HPsi[1](E)) 
=- \frac{3!\lambda^3}{ r \ly^3}  \Imm Z_{\ly, D_{\SY} - \frac{1}{2 \lambda} \ly ,  \frac{1}{2 \lambda}}(E),
$$
and for $E \in D^b(X)$
$$
 \Imm Z_{\ly, D_{\SY} - \frac{1}{2 \lambda} \ly ,  \frac{1}{2 \lambda}}(\Psi(E)) 
 = - \frac{3!}{\lambda^3 r \lx^3}  \Imm Z_{\lx, -D_{\SX} + \frac{ \lambda}{2} \lx , \frac{ \lambda}{2} }(E).
$$
\end{enumerate}
\end{prop}
\begin{proof}
Let us prove (1).  By definition 
\begin{align*}
Z_{\lx, -D_{\SX} + \frac{ \lambda}{2} \lx , \frac{ \lambda}{2} }(E) 
& = - \int_X e^{D_{\SX} - \frac{ \lambda}{2} \lx - i \frac{ \lambda \sqrt{3}}{2} \lx } \ch(E) \\
& = - \int_X e^{-\lambda \lx \left( \frac{1}{2} + i \frac{\sqrt{3}}{2} \right) }   \ch^{-D_{\SX}}(E).
\end{align*}
Hence its imaginary part is 
\begin{align*}
\Imm Z_{\lx, -D_{\SX} + \frac{ \lambda}{2} \lx , \frac{ \lambda}{2} }(E) 
& = \int_X (0 , \lambda \lx \sqrt{3}/2, -\lambda^2 \lx^2 \sqrt{3}/4, 0) \cdot \ch^{-D_{\SX}}(E) \\
& = \frac{ \lambda \sqrt{3}}{4} \left( v_2^{-D_{\SX}, \lx}- \lambda v_1^{-D_{\SX}, \lx} \right)
\end{align*}
as required.  Similarly one can prove the other part. \\

Part (2) follows from Theorem \ref{prop:antidiagonal-rep-cohom-FMT} together with part (1). 
\end{proof}

Most of the next sections are devoted to prove the following:
\begin{thm}
\label{prop:equivalence-hearts-abelian-threefolds}
The   Fourier-Mukai  transforms $\Psi[1]$ and $\HPsi[2]$ give the equivalences
\begin{align*}
&\Psi[1]\left(\aA_{\lx, -D_{\SX} + \frac{ \lambda}{2} \lx , \frac{ \lambda}{2} } \right) \cong 
\aA_{\ly, D_{\SY} - \frac{1}{2 \lambda} \ly ,  \frac{1}{2 \lambda}}, \\
&\HPsi[2] \left( \aA_{\ly, D_{\SY} - \frac{1}{2 \lambda} \ly ,  \frac{1}{2 \lambda}} \right) \cong
\aA_{\lx, -D_{\SX} + \frac{ \lambda}{2} \lx , \frac{ \lambda}{2} }
\end{align*}
of the abelian categories as in \eqref{eqn:double-tilt-heart} of Section \ref{sec:double-tilting-construction}.
\end{thm}

\begin{rmk}
\label{prop:equivalence-stab-parameters}
\rm
One can see that 
the complexified ample classes 
\begin{equation*}
\left.\begin{aligned}
 &\Omega = \left( -D_{\SX} +  \lambda \lx/2 \right)  + i \sqrt{3} \lambda \lx/2 \\        
 &  \Omega' = \left( D_{\SY} - \ly/(2 \lambda) \right) + i  \sqrt{3}  \ly/(2 \lambda)
       \end{aligned}
  \ \right\}
\end{equation*}
on $X$, $Y$ associated to  the above theorem are exactly the
solutions given for the $g=3$ case in \eqref{eqn:condition-for-abelian-equivalence}. Moreover, the
shifts are compatible with the images of the skyscraper sheaves $\oO_x$, $\oO_y$ under the   Fourier-Mukai  transforms which are also minimal objects in the corresponding abelian categories, as
discussed in Example~\ref{example:minimal-objects-semihomogeneous-bundles}.
\end{rmk}


\begin{thm}
\label{prop:BG-ineq-abelian-threefolds}
The Bogomolov-Gieseker type inequality in Conjecture \ref{prop:BG-ineq-conjecture} holds for $X$.
\end{thm}
\begin{proof}
By deforming tilt stability parameters it is enough to consider  a dense
family of classes $B \in \NS_{\QQ}(X)$, and $\alpha \lx  \in \NS_{\QQ}(X)$  for 
$\nu_{\lx, B, \alpha}$  stable objects of zero tilt slope.

For any given $B \in \NS_{\QQ}(X), \alpha \in \QQ_{>0}$ and ample class 
$\lx \in \NS_{\QQ}(X)$,  one can 
find  $-D_{\SX} \in \NS_{\QQ}(X)$ and 
$\lambda \in \QQ_{>0}$ such that
\begin{align*}
B =  -D_{\SX} +  \lambda \lx/2,  \ \text{ and} \  \alpha   = \lambda /2 .
\end{align*} 
Now   one can find a non-trivial   Fourier-Mukai  transform $\Psi$ which gives the
equivalence of abelian categories as in Theorem~\ref{prop:equivalence-hearts-abelian-threefolds}.
From Lemma \ref{prop:reduction-BG-ineq-class}, it is enough to check that the  
Bogomolov-Gieseker type inequality is satisfied by each object in  $\mM_{\lx, B, \alpha }$.

Moreover,  the objects  in 
$$
\mM^o: = \{M: M \cong \eE^*_{X \times\{y\}}[1]
\text{ for some } y \in X \} \subset \mM_{\lx, B, \alpha }
$$ 
satisfy the  
Bogomolov-Gieseker type inequality (Example~\ref{example:minimal-objects-semihomogeneous-bundles} and Note~\ref{prop:BG-ineq-for-tilt-stable-trivial-discriminant}).
So we only need to check the Bogomolov-Gieseker type inequality for objects in
$\mM_{\lx, B, \alpha } \setminus \mM^o$.

Let  $E \in \mM_{\lx, B, \alpha } \setminus \mM^o$.
Then $E[1] \in \aA_{\lx, B, \alpha }$ is a minimal object and so by the
equivalence in Theorem~\ref{prop:equivalence-hearts-abelian-threefolds},
$\Psi[1](E[1]) \in \aA_{\ly, B', \alpha'}$ is also a minimal object. 
Here 
\begin{align*}
  B' =  D_{\SY} - \ly/(2 \lambda),    \ \text{ and} \  \alpha' = {1}/(2 \lambda).
\end{align*}

So
$\Psi[1](E[1]) \in \fF'_{\ly, B', \alpha'}[1]$ or 
 $\Psi[1](E[1]) \in \tT'_{\ly, B', \alpha'}$.
Since 
$\Imm   Z_{\lx, B, \alpha}(E) = 0$,
from Proposition~\ref{prop:imgainary-part-central-charge}, $\Imm  Z_{\ly, B', \alpha'}(\Psi[1](E[1]))=0$.
 
If $\Psi[1](E[1]) \in  \tT'_{\ly, B', \alpha'}$
then by Proposition~\ref{prop:first-tilt-behaves-like-sheaves-surfaces},  $\Psi[1](E[1]) \in \Coh_{0}(Y)$ and so $E$ has a filtration
of objects from $\mM^o$; which is not possible.
Hence, $\Psi[1](E)$
is a $\nu_{\ly, B', \alpha'}$ stable object with zero tilt slope.
Moreover, for any $y \in Y$ we have
$$
\Ext_{\SY}^1(\oO_y, \Psi[1](E)) 
\cong \Hom_{\SY}(\oO_y, \Psi[2](E)) 
\cong \Hom_{\SX}(\eE^*_{X \times\{y\}}[1],E) 
= 0,
$$
as $E \not \cong  \eE^*_{X \times\{y\}}[1]$. 
Hence,  $\Psi[1](E) \in \mM_{\ly, B', \alpha' }$. 
Therefore, from Proposition \ref{prop:slope-bounds}-(4)
$v_1^{B' - \alpha'\ly}(\Psi[1](E)) \ge 0$. So
we have 
\begin{equation*}
v_1^{B' - \alpha'\ly}(\Psi[1](E)) =v_1^{D_{\SY}  , \ly}(\Psi[1](E)) + \frac{1}{ \lambda} v_0^{D_{\SY} , \ly}(\Psi[1](E)) \ge 0.
\end{equation*}
From Theorem \ref{prop:antidiagonal-rep-cohom-FMT},
 we have 
 \begin{equation*}
v_2^{-D_{\SX}  , \lx}(E) - \frac{1}{ \lambda} v_3^{-D_{\SX} , \lx}(E) \ge 0.
\end{equation*}
Since $\Imm   Z_{\lx, B, \alpha}(E) = 0$,  from Proposition~\ref{prop:imgainary-part-central-charge}--(1)
$$
v_2^{-D_{\SX}, \lx}(E) = \lambda v_1^{-D_{\SX}, \lx}(E).
$$
Therefore 
\begin{equation*}
\label{BGineq}
v_3^{-D_{\SX}, \lx}(E) - \lambda^2 v_1^{-D_{\SX}, \lx}(E) \le 0.
\end{equation*}
This is the required  Bogomolov-Gieseker type inequality for $E$.
\end{proof}

We can now deduce the main theorem of this paper:
\begin{thm}
\label{prop:Bridgeland-stab-abelian-threefolds}
Conjecture \ref{prop:BMT-stab-cond-conjecture} holds for abelian threefolds. Therefore, we have the symmetries of Bridgeland stability conditions as in Theorem \ref{prop:intro-main-stab-symmetries}.
\end{thm}

\section{FM Transform of Sheaves on Abelian Varieties}
\label{sec:FMT-abelian-varieties}

Let $X, Y$ be two derived equivalent $g$-dimensional abelian varieties as in Section \ref{sec:cohomological-FMT}, which is given by the Fourier-Mukai transform $\Phi_{\eE}^{\SX \to \SY} : D^b(X) \to D^b(Y)$. 
Let $\lx \in \NS_{\QQ}(X)$ be an ample class on $X$ and let $\ly \in \NS_{\QQ}(Y)$ be the induced ample class on $Y$ as in Theorem \ref{prop:general-cohomo-FMT}.

We study some properties of the slope stability of the images under the  Fourier-Mukai transform $\Phi_{\eE}^{\SX \to \SY}$ in this section. 
The slope $\mu_{\lx, B}$ of $E \in \Coh(X)$ is defined by 
$$
\mu_{\lx, B}(E) = \frac{\lx^{g-1} \ch_1^{B}(E)}{\lx^{g} \ch_0^{B}(E)},
$$
where $g = \dim X$, and we consider the notion of slope stability as similar to threefolds in Section \ref{sec:double-tilting-construction}.

\begin{nota}
\rm
Let use write  
\begin{align*}
&\Psi : = \Phi_{\eE}^{\SX \to \SY}  : D^b(X) \to D^b(Y), \\
&\HPsi : =\Phi_{\eE^\vee}^{\SY \to \SX}  : D^b(Y) \to D^b(X).
\end{align*}
Let $\cC$ be a heart of a bounded t-structure on $D^b(Y)$. For a sequence of integers $i_1, i_2, \ldots, i_k$  we define
$$
V^{\Psi}_{\cC}(i_1, i_2, \ldots, i_k) = \{ E \in D^b(X) : H^{i}_{\cC}(E) =0 \text{ for } i \ne \{ i_1, i_2, \ldots, i_k \}\}.
$$
For $E \in D^b(X)$ we write 
$$
\Psi^{k}(E) = H^k_{\Coh(Y)} (\Psi(E)) = \hH^k(\Psi(E)).
$$
We consider similar notions for $\HPsi$. 
\end{nota}


\begin{note}
\label{prop:restriction-setup}
\rm
There exists a minimal $N \in \ZZ_{>0}$ such that 
$N \lx$ becomes integral.
Let us fix a divisor $H_{\SX}$ from the linear system $| N \lx |$. 
The divisor $H_{\SX, x} \in | N \lx |$ is the translation of $H_{\SX}$ by $-x$:
\begin{equation*}
H_{\SX, x} := t_x^{-1}\left( H_{\SX} \right).
\end{equation*}
For positive integer $m$, let $m H_{\SX, x}$ be the divisor in the linear system $|m N \lx |$.
So $mH_{\SX, x}$ is the zero locus of a section of the line bundle $\oO_X(mH_{\SX, x})$, and
 we have the short exact sequence
\begin{equation*}
\label{ses_res}
0 \to \oO_X(- mH_{\SX, x}) \to \oO_X \to \oO_{mH_{\SX}, x} \to 0
\end{equation*}
in $\Coh(X)$.

Let $E \in \Coh(X)$.
Apply the functor $E \stackrel{\textbf{L}}{\otimes} (-)$ to  the above short exact sequence and consider the long exact sequence
of $\Coh(X)$-cohomologies.
Since $\oO_X(- mH_{\SX, x}), \oO_X$ are locally free, we have the long exact sequence
$$
 0 \to \TOR_1(E, \oO_{mH_{\SX, x}}) \to  E(- mH_{\SX, x}) \to E \to E \otimes\oO_{mH_{\SX, x}} \to 0
$$
in $\Coh(X)$ and $\TOR_i(E, \oO_{mH_{\SX, x}}) = 0$ for $i \ge 2$.

Assume  $E \in \Coh_k(X)$ for some $k \in \{0,1,\cdots, g\}$.
 For generic $x \in X$, we have $\dim (\Supp(E) \cap H_{\SX,x}) \le (k-1)$ and so
 $\TOR_1(E, \oO_{mH_{\SX,x}}) \in \Coh_{\le k-1}(X)$.
However, $  E(- mH_{\SX,x}) \in \Coh_k(X)$, and so  $\TOR_1(E, \oO_{mH_{\SX, x}}) =0$.
Therefore we have the short exact sequence
\begin{equation}
\label{eqn:restriction-ses-higher-degree}
0   \to  E(- mH_{\SX, x}) \to E \to E|_{m H_{\SX,x}} \to 0
\end{equation}
in $\Coh(X)$, where we write 
$$
E|_{mH_{\SX,x}} := E \otimes\oO_{mH_{\SX,x}}.
$$
Since any $E \in \Coh(X)$ is an extension of sheaves from $\Coh_k(X)$, $1 \le k \le g$,
for generic $x \in X$ we have  $\TOR_i(E, \oO_{mH_{\SX,x}}) = 0$ for $i \ge 1$
and so the short exact sequence \eqref{eqn:restriction-ses-higher-degree}.
Moreover, when $0 \to E_1 \to E_2 \to E_3 \to 0$ is a short exact sequence in $\Coh(X)$,
for generic $x \in X$ we have $\TOR_i(E_j, \oO_{mH_{\SX,x}}) = 0$ for $i \ge 1$ and all $j$, and so
\begin{equation*}
0 \to E_1|_{mH_{\SX,x}} \to E_2|_{mH_{\SX,x}} \to E_3|_{mH_{\SX,x}} \to 0
\end{equation*}
is a short exact sequence in $\Coh(X)$.
\end{note}

\begin{nota}
\rm
Similarly, we consider the divisors $H_{\SY, y}, H_{\SHX, \hx}, H_{\SHY,\hy}$ on $Y, \HX, \HY$ with respect to the induced ample classes $\ly, \lhx, \lhy$ as in Theorem \ref{prop:general-cohomo-FMT}  under the Fourier-Mukai transforms.
\end{nota}

\begin{note}
\rm
From the definition of  Fourier-Mukai  transform, for any $E \in \Coh(X)$ we have 
$$
\Psi^k(E) = 0, \text{ for all } k \not \in \{0, 1, \ldots, g\}.
$$
\end{note}

\begin{prop}
\label{prop:V-0-for-higher-ample-twist}
Let $E \in \Coh(X)$. Then for large enough $m \in \ZZ_{>0}$, and any $x \in X$, we have the following:
\begin{enumerate}
\item $E(mH_{\SX, x}) \in V^{\Psi}_{\Coh(Y)}(0)$.
\item If $E \in \Coh_k(X)$ such that $\calExt^i(E, \oO_X) =0$ for $i \ne k$, then 
$E(-mH_{\SX, x}) \in V^{\Psi}_{\Coh(Y)}(g-k)$.
\end{enumerate}
\end{prop}
\begin{proof}
(i) \
Let $y \in Y$ be any point.
We have 
\begin{align*}
\Hom_{\SY}(\Psi^g(E(mH_{\SX, x})), \oO_y) 
& \cong \Hom_{\SY}( \Psi(E(mH_{\SX, x}))[g], \Psi(\eE^*_{X \times \{y\}}[g]) \\
& \cong \Hom_{\SX}(E(mH_{\SX, x}), \eE^*_{X \times \{y\}} ) \\
& \cong \Hom_{\SX}(E(mH_{\SX, x}) \otimes \eE_{X \times \{y\}}, \oO_{X}) \\
& \cong \Hom_{\SX}(\oO_{X}, E(mH_{\SX, x}) \otimes \eE_{X \times \{y\}}[g] )\\
& \cong H^{g}(X, E(mH_{\SX, x}) \otimes \eE_{X \times \{y\}}) =0,
\end{align*}
for large enough $m\in \ZZ_{>0}$. Hence, $\Psi^g(E(mH_{\SX, x})) =0$. So we have 
\begin{align*}
\Hom_{\SY}(\Psi^{g-1}(E(mH_{\SX, x})), \oO_y) 
& \cong \Hom_{\SY}( \Psi(E(mH_{\SX, x}))[g-1], \Psi(\eE^*_{X \times \{y\}}[g]) \\
& \cong \Hom_{\SX}(E(mH_{\SX, x}), \eE^*_{X \times \{y\}}[1] ) \\
& \cong \Hom_{\SX}(E(mH_{\SX, x}) \otimes \eE_{X \times \{y\}}, \oO_{X}[1]) \\
& \cong \Hom_{\SX}(\oO_{X}, E(mH_{\SX, x}) \otimes \eE_{X \times \{y\}}[g-1] )\\
& \cong H^{g-1}(X, E(mH_{\SX, x}) \otimes \eE_{X \times \{y\}}) =0,
\end{align*}
for large enough $m\in \ZZ_{>0}$.
Hence, $\Psi^{g-1}(E(mH_{\SX, x})) =0$. In this way, one can show that for large enough $m \in \ZZ_{>0}$,
$\Psi^{k}(E(mH_{\SX, x})) =0$ for all $k \ne 0$ as required. \\

\noindent (ii) \
From Lemma \ref{prop:dual-FMT} and part (i), we have  
\begin{align*}
 \Phi^{\SX \to \SY}_{\eE}(E(-mH_{\SX, x}))  
& \cong \left( \Phi^{\SX \to \SY}_{\eE^{\vee}} \left( \left(E(-mH_{\SX, x})\right)^\vee \right)[g] \right)^\vee \\
  & \cong \left( \Phi^{\SX \to \SY}_{\eE^{\vee}} \left(  \calExt^k(E, \oO_X)(mH_{\SX, x})  \right)[g-k] \right)^\vee\\
  & \cong \left( \hH^0 \left( \Phi^{\SX \to \SY}_{\eE^{\vee}} \left(  \calExt^k(E, \oO_X)(mH_{\SX, x})  \right) \right) \right)^\vee [-(g-k)].
\end{align*}
 Therefore, we have the required claim.
\end{proof}

\begin{prop}
\label{prop:FMT-cohomology-vanishing-torsion-sheaves}
Let $E \in \Coh_{\le k}(X)$. Then for $j \ge k+1$
$$
\Psi^{j}(E) = 0.
$$
\end{prop}
\begin{proof}
Let $y \in Y$ be any point.
If $k \le (g-1)$ then as similar to the proof of Proposition \ref{prop:V-0-for-higher-ample-twist}--(i),
we have 
\begin{align*}
\Hom_{\SY}(\Psi^g(E), \oO_y) 
& \cong H^{g}(X, E \otimes \eE_{X \times \{y\}}) =0,
\end{align*}
as $E \otimes \eE_{X \times \{y\}} \in \Coh_{\le k}(X)$;  hence, $\Psi^g(E(mH_{\SX, x})) =0$. 

If $k \le (g-2)$, then similarly we have 
\begin{align*}
\Hom_{\SY}(\Psi^{g-1}(E ), \oO_y) 
& \cong H^{g-1}(X, E  \otimes \eE_{X \times \{y\}}) =0,
\end{align*}
as $E \otimes \eE_{X \times \{y\}} \in \Coh_{\le k}(X) \subset  \Coh_{\le (g-2)}(X)$;
hence, $\Psi^{g-1}(E(mH_{\SX, x})) =0$. In this way, one can show that   for $j \ge k+1$, 
$
\Psi^{j}(E) = 0$.
\end{proof}

\begin{prop}
\label{prop:FMT-0-cohomology-reflexive}
Let $E \in \Coh(X)$. Then we have the following:
\begin{enumerate}
\item If $E \in V^{\Psi}_{\Coh(Y)}(0)$ then $\Psi^0(E)$ is a locally free sheaf. 
\item $\Psi^0(E)$ is a reflexive sheaf on $Y$.
\end{enumerate} 
\end{prop}
\begin{proof}(i) \ \
Suppose  $E \in V^{\Psi}_{\Coh(Y)}(0)$.  
For any $y \in Y$, we have 
\begin{align*}
\Ext^{1}_{\SY}(\Psi^0(E), \oO_y) & \cong \Hom_{\SY}(\Psi^0(E), \oO_y[1]) \\
& \cong  \Hom_{\SX}(\HPsi  \Psi^0(E)[g], \HPsi (\oO_y) [g+1]) \\
& \cong \Hom_{\SX}(E,  \eE^*_{X \times \{y\}}[g+1]).
\end{align*}
Hence $\Sing \Psi^0(E) = \emptyset$, that is $\Psi^0(E)$ is a locally free sheaf (see Definition \ref{defi:singularity-set}). \\

\noindent (ii) \ For generic $x \in X$ and  $m \in \ZZ_{>0} $, apply the  Fourier-Mukai  transform 
$\Psi$ to the $\oO_X(mH_{\SX, x})$ twisted short exact sequence \eqref{eqn:restriction-ses-higher-degree}:
$$
0 \to E \to E(mH_{\SX, x}) \to E(mH_{\SX, x})|_{m H_{\SX, x}} \to 0.
$$
By considering long exact sequence of $\Coh(Y)$ cohomologies, we get the following short exact sequence 
\begin{equation*}
0 \to \Psi^0(E) \to \Psi^0(E(mH_{\SX, x})) \to Q \to 0
\end{equation*}
 for some subsheaf $Q$ of $\Psi^0(E(mH_{\SX, x} )|_{mH_{\SX, x}} )$.

From Proposition \ref{prop:V-0-for-higher-ample-twist}--(i), for large enough $m \in \ZZ_{>0}$, 
$E(m H_{\SX, x})  \in V^{\Psi}_{\Coh(Y)}(0)$. From part (i) 
$\Psi^0(E(mH_{\SX, x}))$ is locally free. 
Hence  
$\Psi^0(E)$ is a torsion free sheaf. Similarly, one can show that $\Psi^0(E(m H_{\SX, x})|_{m H_{\SX, x}} )$ is torsion free. 
Therefore, from Lemma  \ref{prop:reflexive-sheaf-results}--(2), $\Psi^0(E)$ is a reflexive sheaf on $Y$.
\end{proof}

\begin{prop}
\label{prop:FMT-0-g-cohomo-vanishing}
Let $E \in \Coh(X)$. Then we have the following:
\begin{enumerate}
\item If $E \in \HN_{\lx, -D_{\SX}}^{\mu}((0, +\infty])$ then $\Psi^g(E) = 0$.
\item If $E \in \HN_{\lx, -D_{\SX}}^{\mu}(0)$ then $\Psi^g(E) = \Coh_0(Y)$.
\item If $E \in \HN_{\lx, -D_{\SX}}^{\mu}((-\infty,0])$ then $\Psi^0(E) = 0$.
\end{enumerate}
\end{prop}
\begin{proof}
(i) \
Let $E \in \HN_{\lx, -D_{\SX}}^{\mu}((0, +\infty])$. Then for any $y \in Y$, we have
\begin{align*}
\Hom_{\SY}(\Psi^g(E), \oO_y) 
 & \cong \Hom_{\SY}(\Psi(E)[g], \oO_y) \\
                  & \cong \Hom_{\SY}(\Psi(E)[g], \Psi(\eE^*_{X \times\{y\}})[g]) \\
                 & \cong \Hom_{\SX}(E, \eE_{X \times\{y\}}^*) = 0,
\end{align*}
as $\eE_{X \times\{y\}}^* \in \HN_{\lx, -D_{\SX}}^{\mu}(0)$. Therefore, $\Psi^g(E) = 0$ as required. \\
 
 \noindent (ii) \ Suppose $E \in \HN_{\lx, -D_{\SX}}^{\mu}(0)$ is slope stable. If $\Psi^g(E) \ne 0$ then there exists $y  \in Y$ such that 
 $\Hom_{\SY}(\Psi^g(E), \oO_{y }) \ne 0$. Hence,    as in part (i) there exists a non-trivial map 
 $E \to \eE_{X \times\{y\}}^*$. Since $E$ is slope stable, this map is an injection with a quotient in $\Coh_{\le (g-2)}(X)$. By applying the  Fourier-Mukai  transform $\Psi$ to this short exact sequence of sheaves on $X$,
 and considering the long exact sequence of $\Coh(Y)$ cohomologies, we obtain that $\Psi^g(E) \cong \oO_{y }$. This completes the proof.
\\

\noindent (iii) \
Let $E \in \HN_{\lx, -D_{\SX}}^{\mu}((-\infty,0])$. We can assume $E$ is slope stable using the Harder-Narasimhan and
Jordan-H\"older filtrations. 

Since 
$$
\calExt^{i}(\Psi^j(E),\oO_{\SY}) \in \Coh_{\le (g-i)}(Y),
$$
for generic $y \in  Y$ we have 
\begin{align*}
& \Hom_{\SY}(\tau_{\ge 1} \Psi(E), \oO_y) =0, \\
& \Hom_{\SY}(\tau_{\ge 1} \Psi(E) [-1], \oO_y).
\end{align*}
Hence, by applying the functor $\Hom_{\SY}(-, \oO_y)$ to the distinguished triangle
$$
 \tau_{\ge 1} \Psi(E)[-1] \to  \Psi^0(E) \to  \Psi(E) \to \tau_{\ge 1} \Psi(E)
$$
for generic $y \in Y$, we have
\begin{align*}
\Hom_{\SY}(\Psi^0(E), \oO_y) 
 & \cong \Hom_{\SY}(\Psi(E), \oO_y) \\
                 & \cong \Hom_{\SY}(\Psi(E), \Psi(\eE_{X \times\{y\}}^* )[g]) \\
                 & \cong \Hom_{\SX}(E,  \eE_{X \times\{y\}}^* [g]) \\
                 & \cong \Hom_{\SX} (\eE_{X \times\{y\}}^* ,  E)^\vee.
\end{align*}

If $\mu_{\lx, -D_{\SX}} (E) <0$  then $\Hom_{\SX} (\eE_{X \times\{y\}}^* ,  E) = 0$.
Otherwise, $\mu_{\lx, -D_{\SX}} (E) =0$ and since $E$ is assumed to be slope stable, any non-trivial map in
$\Hom_{\SX} (\eE_{X \times\{y\}}^* ,  E)$ gives rise to an isomorphism of sheaves; and  in this case we have  $\Psi^0(E)=0$.

Therefore, for generic $y \in Y$, $\Hom_{\SY}(\Psi^0(E), \oO_y) = 0$. By Proposition~\ref{prop:FMT-0-cohomology-reflexive}, 
 $\Psi^0(E)$   is reflexive, and so we have $\Psi^0(E) = 0$. 
\end{proof}
\begin{prop}
\label{prop:sheaf-cohomo-vanishing-FMT-vanishing}
We have the following for $E \in \Coh(X)$:
\begin{enumerate}
\item If $H^g(X, E \otimes \eE_{X \times\{y\}}) =0 $ for any $y \in Y$, then $\Psi^g(E) = 0$.
\item If  $H^0(X, E \otimes \eE_{X \times\{y\}}) =0 $ for any $y \in Y$,  then $\Psi^0(E) = 0$.
\end{enumerate}
\end{prop}
\begin{proof}
(i) \
As similar to the proof Proposition \ref{prop:V-0-for-higher-ample-twist}--(i), for any $y \in Y$  
\begin{align*}
\Hom_{\SY}(\Psi^g(E), \oO_y) 
                 & \cong H^g(X, E \otimes \eE_{X \times\{y\}}) =0.
\end{align*}
 Therefore, $\Psi^g(E) = 0$ as required. \\
 
\noindent (ii) \
Suppose $H^0(X, E \otimes \eE_{X \times\{y\}}) =0 $ for any $y \in Y$.  
By similar arguments in the proof of (iii) of Proposition \ref{prop:FMT-0-g-cohomo-vanishing}, 
for generic $y \in Y$, we have
\begin{align*}
\Hom_{\SY}(\Psi^0(E), \oO_y) 
                 & \cong \Hom_{\SX} (\eE_{X \times\{y\}}^* ,  E)^\vee \\
                 & \cong \Hom_{\SX} (\oO_X ,  E \otimes \eE_{X \times\{y\}})^\vee \\
                 & \cong H^0(X, E \otimes \eE_{X \times\{y\}})^\vee = 0.
\end{align*}
Since  
 $\Psi^0(E)$   is reflexive (Proposition \ref{prop:FMT-0-cohomology-reflexive}--(ii)),  we have $\Psi^0(E) = 0$ as required. 
\end{proof}
\begin{prop}
\label{prop:slope-bound-FMT-0-g}
Let $E \in \Coh(X)$. Then we have the following:
\begin{enumerate}
\item $\Psi^0(E) \in \HN_{\ly, D_{\SY}}^{\mu}((-\infty,0])$.
\item If $E \in \Coh_{\ge 1}(X)$ then $\Psi^0(E) \in \HN_{\ly, D_{\SY}}^{\mu}((-\infty,0))$.
\end{enumerate}
\end{prop}
\begin{proof}
 For generic $x \in X$ and large enough  $m \in \ZZ_{>0} $, apply the  Fourier-Mukai  transform 
$\Psi$ to the $\oO_X(mH_{\SX, x})$ twisted short exact sequence \eqref{eqn:restriction-ses-higher-degree}:
$$
0 \to E \to E(mH_{\SX, x}) \to E(mH_{\SX, x})|_{m H_{\SX, x}} \to 0.
$$
By considering the long exact sequence of $\Coh(Y)$ cohomologies   we get 
 $$
\Psi^0(E) \hookrightarrow \Psi^0(E(mH_{\SX, x})).
 $$
 Therefore, it is enough to 
 show the corresponding claims for $E(mH_{\SX, x})$ with  large enough $m \in \ZZ_{>0}$ and generic $x \in X$.
For such $m$, $E(mH_{\SX, x}) \in V^{\Psi}_{\Coh(Y)}(0)$.  \\

\noindent (i) \  
For any $ T \in \HN_{\ly, D_{\SY}}^{\mu}((0, +\infty])$, we have 
\begin{align*}
\Hom_{\SY}(T, \Psi^0(E(m H_{\SX, x}))) & \cong \Hom_{\SX}(\HPsi(T), \HPsi \Psi^0(E(m H_{\SX, x}))) \\
& \cong \Hom_{\SX}(\HPsi(T),E(m H_{\SX, x})[-g]) =0,
\end{align*}
as from Proposition \ref{prop:FMT-0-g-cohomo-vanishing}--(i), $\HPsi^g(T) =0$.
Hence, $\Psi^0(E(mH_{\SX, x})) \in \HN_{\ly, D_{\SY}}^{\mu}((-\infty,0])$ as required. \\

\noindent (ii) \ Let us assume $E \in \Coh_{\ge 1}(X)$. For any $ T \in \HN_{\ly, D_{\SY}}^{\mu}([0, +\infty])$, we have 
\begin{align*}
\Hom_{\SY}(T, \Psi^0(E(m H_{\SX, x}))) & \cong \Hom_{\SX}(\HPsi(T), \HPsi \Psi^0(E(m H_{\SX, x}))) \\
& \cong \Hom_{\SX}(\HPsi(T),E(m H_{\SX, x})[-g]) \\
& \cong \Hom_{\SX}(\HPsi^g(T),E(m H_{\SX, x}))  =0,
\end{align*}
as from Proposition \ref{prop:FMT-0-g-cohomo-vanishing}--(ii), $\HPsi^g(T) \in \Coh_0(X)$.
Hence, $\Psi^0(E(m H_{\SX, x})) \in \HN_{\ly, D_{\SY}}^{\mu}((-\infty,0))$ as required. \\
\end{proof}

\begin{prop}
\label{prop:slope-bound-torsion-sheaf}
Let $1 \le k \le g$. If  $E \in \Coh_{\le k}(X)$, then 
$\Psi^k(E) \in  \HN_{\ly, D_{\SY}}^{\mu}((0, +\infty])$.
\end{prop}
\begin{proof}
Consider the torsion sequence of $E \in \Coh_{\le k}(X)$; so $E$ fits into the short exact sequence
$$
0 \to E_{\le (k-1)} \to E \to E_k \to 0,
$$
for some $E_{\le (k-1)} \in \Coh_{\le (k-1)}(X)$ and $E_k \in \Coh_k(X)$. By applying the  Fourier-Mukai  transform $\Psi$ and considering the long exact sequence of $\Coh(Y)$ cohomologies, we obtain 
$$
\Psi^k (E) = \Psi^k(E_k).
$$
Hence, we can assume $E \in \Coh_k(X)$. 

 For generic $x \in X$ and large enough $m \in \ZZ_{>0} $, apply the  Fourier-Mukai  transform 
$\Psi$ to the $\oO_X(mH_{\SX, x})$ twisted short exact sequence \eqref{eqn:restriction-ses-higher-degree}:
$$
0 \to E \to E(mH_{\SX, x}) \to E(mH_{\SX, x})|_{m H_{\SX, x}} \to 0.
$$
Here $E(m H_{\SX, x})|_{m H_x} \in \Coh_{(k-1)}(X)$. 
By considering long exact sequence of $\Coh(Y)$ cohomologies, we get 
\begin{equation*}
 \Psi^k(E) \cong  \Psi^{k-1}\left(  E(m H_{\SX, x})|_{m H_{\SX, x}}\right),
\end{equation*}
as for large enough $m \in \ZZ_{>0}$, $E(m H_{\SX, x})  \in V^{\Psi}_{\Coh(Y)}(0)$. 
Therefore, inductively we only need to consider the case $k=1$. 

Suppose $E \in \Coh_1(X)$. For generic $x \in X$ and large enough $m \in \ZZ_{>0}$, apply the  Fourier-Mukai  transform 
$\Psi$ to the short exact sequence \eqref{eqn:restriction-ses-higher-degree}:
$$
0 \to E(-m H_{\SX, x}) \to E \to E|_{m H_{\SX, x}} \to 0,
$$
where $E|_{m H_{\SX, x}} \in \Coh_{0}(X)$. 
By considering long exact sequence of $\Coh(Y)$ cohomologies, we get $E(-m H_{\SX, x}) \in  V^{\Psi}_{\Coh(Y)}(1)$ and also
\begin{equation}
\label{eqn:quotient-torsion-1}
 \Psi^1( E(-m H_{\SX, x})) \twoheadrightarrow \Psi^1( E).
\end{equation}

From Lemma \ref{prop:dual-FMT}, $\left(\Phi^{\SX \to \SY}_{\eE}( E(-m H_{\SX, x}))\right)^\vee \cong \Phi^{\SX \to \SY}_{\eE^{\vee}} \left( \left( E(-m H_{\SX, x})\right)^\vee \right)[g]$,  and from Proposition \ref{prop:slope-bound-FMT-0-g}--(ii), 
$$
\hH^0 (\Phi^{\SX \to \SY}_{\eE^{\vee}}(\calExt^1(E, \oO_X) (m H_{\SX, x})) \in \HN_{\ly, -D_{\SY}}^{\mu}((-\infty ,0)),
$$
so we deduce 
$$
\Psi^1( E(-m H_{\SX, x})) \cong \left( \hH^0 (\Phi^{\SX \to \SY}_{\eE^{\vee}}(\calExt^1(E, \oO_X) (m H_{\SX, x})) \right)^* \in  \HN_{\ly, D_{\SY}}^{\mu}((0, +\infty)).
$$
From \eqref{eqn:quotient-torsion-1} we have $\Psi^1( E)  \in  \HN_{\ly, D_{\SY}}^{\mu}((0, +\infty])$.
This completes the proof.
\end{proof}

\section{Equivalences of Stability Condition Hearts on Abelian Surfaces}
\label{sec:equivalence-stab-hearts-surface}
In this section  we show that the expectation in the end of Section \ref{sec:action-FMT-central-charge}, more precisely Conjecture \ref{prop:conjecture-equivalence-stab-hearts} holds on abelian surfaces. 
This result is already known due to Huybrechts and Yoshioka \cite{HuyK3Equivalence, YoshiokaFMT}. However, as for completeness and as a warm-up to study the abelian threefold case in the next sections we present the complete proof and we closely follow that of Yoshioka. 


Let $X, Y$ be derived equivalent abelian surfaces  and let $\lx, \ly$ be ample classes on them respectively as in Theorem \ref{prop:general-cohomo-FMT}.
Let $\Psi$ be the Fourier-Mukai transform $\Phi_{\eE}^{\SX \to \SY}$ from $X$ to $Y$ with kernel  $\eE$, and let $\HPsi =  \Phi_{\eE^\vee}^{\SY \to \SX}$. We have 
\begin{equation*}
\label{imagesurface}
\Psi(\Coh(X)) \subset \langle \Coh(Y), \Coh(Y)[-1], \Coh(Y)[-2] \rangle,
\end{equation*}
and similar relation for $\HPsi$.
Since 
 $\HPsi \circ \Psi \cong  [-2]$ and $\Psi \circ \HPsi \cong  [-2]$, 
 we have the following convergences of the spectral sequences.

\begin{equation}
\label{eqn:mukai-specseq}
\left.\begin{aligned}
 &  E_2^{p,q} = \HPsi^{p} \Psi^q(E) \Longrightarrow \hH^{p+q-2}(E) \\        
 &  E_2^{p,q} = \Psi^{p}\HPsi^q(E) \Longrightarrow \hH^{p+q-2}(E)
       \end{aligned}
  \ \right\}.
\end{equation}
Here and elsewhere we write $\HPsi^{p}(E) = \hH^p (\HPsi(E))$ and $\Psi^{q}(E) = \hH^q (\Psi(E))$. Immediately from the convergence of this spectral sequence for $E \in \Coh(X)$, we deduce that
\begin{itemize}
\item $ \Psi^0(E) \in V^{\HPsi}_{\Coh(X)}(2) $, and $\Psi^2(E) \in V^{\HPsi }_{ \Coh(X)}(0)$;
\item there is an injection $\HPsi^0\Psi^1(E) \hookrightarrow \HPsi^2\Psi^0(E)$, and
 a surjection $ \HPsi^0\Psi^2(E) \twoheadrightarrow \HPsi^2\Psi^1(E)$.
\end{itemize}

Let us recall the notation in Conjecture \ref{prop:conjecture-equivalence-stab-hearts} for our derived equivalent abelian surfaces. 
Consider  the complexified ample classes $\Omega = -D_{\SX} + i \lambda \lx$, 
$\Omega' = D_{\SY} + i (1 /\lambda) \ly $ on $X$, $Y$ respectively. 
The function defined by $Z^{(1)}_{\Omega} = -i \int_{X} e^{- \Omega } (\ch_0, \ch_1, 0)$ together with the 
standard heart $\Coh(X)$ defines a very weak stability condition $\sigma_1$ on $D^b(X)$. 
Define the subcategories 
$$
\fF^{\SX} = \pP^{\SX}_{\sigma_1}((0,\, 1/2]), \ \ \tT^{\SX}= \pP^{\SX}_{\sigma_1}((1/2 ,\, 1])
$$
of $\Coh(X)$ in terms of the associated slicing $\pP^{\SX}_{\sigma_1}$.
In other words,  
$$
\fF^{\SX} =  \HN_{\lx, -D_{\SX}}^{\mu}([0, +\infty)), \ \
 \tT^{\SX}=  \HN_{\lx, -D_{\SX}}^{\mu}((0, +\infty)).
$$
Then the Bridgeland  stability condition heart in Conjecture \ref{prop:conjecture-equivalence-stab-hearts} is 
$$
\bB^{\SX} = \langle \fF^{\SX}[1] , \tT^{\SX} \rangle = \pP^{\SX}_{\sigma_1}((1/2 ,\, 3/2]).
$$
We consider similar subcategories  associated to $\Omega'$ on $Y$. 


We need the following results about cohomology sheaves of the images under the Fourier-Mukai transforms, and closely follow the arguments in the author's PhD thesis \cite[Section 6]{PiyThesis} and which is also adopted from \cite{YoshiokaFMT}.

\begin{prop}
\label{prop:slope-bound-FMT-surface}
Let $E \in \Coh(X)$. Then we have the following:
\begin{enumerate}[label=(\arabic*)]
\item (i) If $E \in \tT^{\SX}$ then $\Psi^2(E) = 0$, and (ii) if $E \in  \fF^{\SX}$ then $\Psi^0(E) = 0$.
\item (i) $\Psi^2(E) \in \tT^{\SY}$, and (ii) $\Psi^0(E) \in \fF^{\SY}$.
\item (i) if $E \in \tT^{\SX}$ then $\Psi^1(E) \in \tT^{\SY}$, and (ii) if $E \in \fF^{\SX}$ then $\Psi^1(E) \in \fF^{\SY}$.
\end{enumerate}
\end{prop}
\begin{proof}
\noindent (1) and (2) follows from  Propositions \ref{prop:FMT-0-g-cohomo-vanishing}-(i), \ref{prop:FMT-0-g-cohomo-vanishing}-(iii), \ref{prop:slope-bound-FMT-0-g}--(i), and  \ref{prop:slope-bound-torsion-sheaf}.  \\

\noindent Let us prove part (3)--(i). Let $E \in \tT^{\SX}$.
By the Harder-Narasimhan  filtration of $\Psi^1(E)$ there exists $T \in \tT^{\SY}$ and $F \in \fF^{\SY}$ such that
$0 \to T \to \Psi^1(E) \to F \to 0$
is a short exact sequence  in $\Coh(Y)$. Assume $F \neq 0$ for a contradiction.
Now apply the Fourier-Mukai transform $\HPsi$ to this short exact sequence  and then consider the long exact sequence of $\Coh(X)$ cohomologies.
By (1)(i) of this proposition, $\Psi^2(E) = 0$. 
So from the convergence of the  Spectral Sequence \ref{Spec-Seq-Mukai} for $E$, $\HPsi^2 \Psi^1(E) = 0$ and
$\HPsi^1 \Psi^1(E)$ is quotient of $E \in \tT^{\SX}$. Hence, we have $\HPsi^1 \Psi^1(E) \in \tT^{\SX}$.
By (1)(i) of this proposition, $\HPsi^2(T) = 0$ and so there is a surjection $\HPsi^1 \Psi^1(E) \twoheadrightarrow \HPsi^1(F)$.
Therefore, 
$$
\lx \cdot \ch_1^{-D_{\SX}}(\HPsi^1(F))  \ge 0,
$$ 
where the equality holds when $\HPsi^1(F) \in \Coh_0(X)$.
Also $\HPsi (F)  \in \Coh(X)[-1]$, and so by Theorem \ref{prop:antidiagonal-rep-cohom-FMT}, 
$$
\lx \cdot \ch_1^{-D_{\SX}}(\HPsi^1(F))  \le 0.
$$
Therefore, $\lx \cdot \ch_1^{-D_{\SX}}(\HPsi^1(F))  = 0$, and so $\HPsi^1(F) \in \Coh_0(X)$. But this is not possible as $\Psi \HPsi^1(F) \in \Coh(Y)[-1]$.
This is the required contradiction to complete the proof.

Similarly one can prove (3)(ii).
\end{proof}

In other words, the results of the above proposition say
\begin{equation*}
\left.\begin{aligned}
 & \Psi (\tT^{\SX}) \subset \langle \fF^{\SY} , \tT^{\SY}[-1]  \rangle \\        
 &  \Psi (\fF^{\SX}) \subset \langle \fF^{\SY}[-1] , \tT^{\SY}[-2]  \rangle
       \end{aligned}
  \ \right\}.
\end{equation*}
Similar results hold for $\HPsi$. Since $\bB^{\SX} = \langle \fF^{\SX}[1] , \tT^{\SX} \rangle$ and  $\bB^{\SY} = \langle \fF^{\SY}[1] , \tT^{\SY} \rangle$,
we have $\Psi [1] (\bB^{\SX} ) \subset \bB^{\SY}$ and $\HPsi [1] (\bB^{\SY} ) \subset \bB^{\SX}$. Hence, 
\begin{equation}
\label{eqn:equivalence-surfaces-hearts}
\Psi [1] (\bB^{\SX} ) \cong  \bB^{\SY}.
\end{equation}
as expected in Conjecture \ref{prop:conjecture-equivalence-stab-hearts} for $g=2$ case. 

\begin{note}
\label{prop:BG-ineq-surface-FMT}
\rm
We can use the equivalence \eqref{eqn:equivalence-surfaces-hearts} of the tilted hearts  to prove the usual Bogomolov-Gieseker type inequality for slope stable torsion free sheaves on an abelian surface. 

Let $E$ be a slope stable torsion free sheaf on an abelian surface $X$ with respect to an ample class 
$\lx \in \NS_{\QQ}(X)$. 
Then it fits into the short exact sequence $0 \to E \to E^{**} \to T \to 0$ for some torsion sheaf 
$T \in \Coh_0(X)$.
Let 
$$
-D_{\SX} = \frac{c_1(E)}{\rk(E)}.
$$
Then consider  the corresponding Fourier-Mukai transform $\Phi_{\eE}^{\SX \to \SY}: D^b(X) \to D^b(Y)$ as in Section \ref{sec:cohomological-FMT}.
Similar to Lemma \ref{prop:minimal-objects-threefold-hearts}, for surfaces (see \cite[Theorem 0.2]{HuyK3Equivalence}), the object 
$$
E^{**}[1] \in \bB^{\SX} 
$$
is a minimal object. Therefore under equivalence \eqref{eqn:equivalence-surfaces-hearts}, the object 
$$
F := \Phi_{\eE}^{\SX \to \SY}[1](E^{**}[1]) \in \bB^{\SY}
$$
is also a minimal object in $\bB^{\SY}$. 
Since $F$ fits in to the short exact sequence 
$$
0 \to \hH^{-1}(F)[1] \to F \to \hH^0(F) \to 0,
$$
we have either $\hH^{-1}(F) =0$ or $\hH^0(F)=0$.

In the first case one can show that $\hH^0(F) \cong \oO_y$ some $y\in Y$, and so $E^{**} \cong \eE^{*}_{X \times \{y\}}$; which satisfies   
$\ch_2^{-D_{\SX}}(E^{**}) = 0$.

In the remaining case, since 
$$
 -\ch_0(F) = \rk (\hH^{-1}(F) )  >0,
$$
from Theorem \ref{prop:antidiagonal-rep-cohom-FMT} we get 
$\ch_2^{-D_{\SX}}(E^{**}[1]) >0$. 

So we have 
$$
\ch_2^{-D_{\SX}}(E) \le 0
$$
as required in the usual Bogomolov-Gieseker inequality for $E$.
\end{note}
\section{FM Transform of Sheaves on  Abelian Threefolds}
\label{sec:FMT-sheaves-abelian-threefolds}
 In this section we further study the slope stability of sheaves under the  Fourier-Mukai  transforms on abelian threefolds continuing Section \ref{sec:FMT-abelian-varieties}. 

Let $X, Y$ be derived equivalent abelian threefolds  and let $\lx, \ly$ be ample classes on them respectively as in Theorem \ref{prop:general-cohomo-FMT}.
Let $\Psi$ be the Fourier-Mukai transform $\Phi_{\eE}^{\SX \to \SY}$ from $X$ to $Y$ with kernel  $\eE$, and let $\HPsi =  \Phi_{\eE^\vee}^{\SY \to \SX}$. Then 
 $\HPsi \circ \Psi \cong  [-3] $ and $\Psi \circ \HPsi \cong  [-3] $.

\begin{nota}
\rm
As in Section \ref{sec:FMT-abelian-varieties}, we write 
$$
\Psi^p(E) = \hH^p\left( \Psi(E)\right)
$$
 and use similar notation for $\HPsi$. 
\end{nota}
\begin{mukaispecseq}
\label{Spec-Seq-Mukai}
$$
E_2^{p,q} = \HPsi^p  \Psi^q (E) \Longrightarrow \hH^{p+q-3}(E).
$$
\end{mukaispecseq}

We can describe the second page of the Mukai Spectral Sequence for $E \in \Coh(X)$  as in the following diagram:

\begin{center}
\begin{tikzpicture}[scale=1.7]
\draw[gray,very thin] (0,0) grid (4,4);
\draw[->,thick] (3.75,0.25) -- (4.5, 0.25) node[above] {$p$};
\draw[->,thick] (0.25,3.75) -- (0.25,4.5) node[left] {$q$};
\draw (2.5,0.5) node(a) {$\scriptstyle \HPsi^2  \Psi^0 (E) $};
\draw (0.5,1.5) node(b) {$\scriptstyle \HPsi^0  \Psi^1 (E)$};
\draw[->,thick] (b) -- node[below] {$\cong$} (a);
\draw (3.5,2.5) node(c) {$\scriptstyle \HPsi^3  \Psi^2 (E)$};
\draw (1.5,3.5) node(d) {$\scriptstyle \HPsi^1  \Psi^3 (E)$};
\draw[->,thick] (d) -- node[above] {$\cong$} (c);
\draw (3.5,0.5) node(e) {$\scriptstyle \HPsi^3  \Psi^0 (E) $};
\draw (1.5,1.5) node(f) {$\scriptstyle \HPsi^1  \Psi^1 (E)$};
\draw[>->,thick] (f) -- node[above] {$ $} (e);
\draw (2.5,2.5) node(g) {$\scriptstyle \HPsi^2  \Psi^2 (E) $};
\draw (0.5,3.5) node(h) {$\scriptstyle \HPsi^0  \Psi^3 (E) $};
\draw[->>,thick] (h) -- node[above] {$ $} (g);
\draw (2.5,1.5) node(i) {$\scriptstyle \HPsi^2  \Psi^1 (E) $};
\draw (3.5,1.5) node(j) {$\scriptstyle \HPsi^3  \Psi^1 (E) $};
\draw (0.5,2.5) node(k) {$\scriptstyle \HPsi^0  \Psi^2 (E)$};
\draw (1.5,2.5) node(l) {$\scriptstyle \HPsi^1  \Psi^2 (E)$};
\draw[->,thick] (l) -- node[above] {$ $} (j);
\draw[->,thick] (k) -- node[above] {$ $} (i);
\end{tikzpicture}
\end{center}  

We deduce the following immediately from the convergence of the Mukai Spectral Sequence for $E \in \Coh(X)$:
\begin{align*}
& \HPsi^0  \Psi^0 (E) =  \HPsi^1  \Psi^0(E) =  \HPsi^2 \Psi^3 (E) =  \HPsi^3 \Psi^3 (E) = 0, \\
& \HPsi^0  \Psi^1 (E) \cong  \HPsi^2  \Psi^0(E), \\
& \HPsi^1  \Psi^3 (E) \cong  \HPsi^3  \Psi^2(E).
\end{align*}

\begin{prop}
\label{prop:FMT-F-1-reflexivity}
Let $E \in \HN_{\lx, -D_{\SX}}^{\mu}((-\infty,0])$. Then $\Psi^1(E)$ is a reflexive sheaf.
\end{prop}
\begin{proof}
By Proposition \ref{prop:FMT-0-g-cohomo-vanishing}--(iii), $\Psi^0(E) = 0$.
Let $y \in Y$. From the convergence of Mukai Spectral Sequence~\ref{Spec-Seq-Mukai} for $E$ and $0 \le i \le 2$, we have
\begin{align*}
\Ext^i_{\SY}(\oO_y, \Psi^1(E))
                     &  \cong \Hom_{\SY}(\oO_y, \Psi^1(E)[i]) \\
                     & \cong  \Hom_{\SX}(\HPsi(\oO_y), \HPsi(\Psi^1(E))[i]) \\
                     & \cong  \Hom_{\SX}(\eE^*_{X \times\{y\}},  \HPsi^2\Psi^1(E)[i-2])
\end{align*}
as $\Hom_{\SX}(\eE^*_{X \times\{y\}}, \tau_{\ge 3}\HPsi(\Psi^1(E))[i])
\cong\Hom_{\SX}(\eE^*_{X \times\{y\}}, \HPsi^3\Psi^1(E)[i-3])=0$. 
Therefore, $\Hom_{\SY}(\oO_y, \Psi^1(E)) =\Ext_{\SY}^1(\oO_y, \Psi^1(E))=0$, 
and $\Ext_{\SY}^2(\oO_y, \Psi^1(E))  \cong \Hom_{\SX}(\eE^*_{X \times\{y\}} , \HPsi^2\Psi^1(E))$.

From the convergence of  Mukai Spectral Sequence~\ref{Spec-Seq-Mukai} for $E$,
$$
0 \to \HPsi^0\Psi^2(E) \to \HPsi^2\Psi^1(E) \to F \to 0
$$
is a short exact sequence  in $\Coh(X)$. Here $F$ is a subobject of $E$ and so $F \in  \HN_{\lx, -D_{\SX}}^{\mu}((-\infty,0])$.
By
applying the functor $\Hom_{\SX}(\eE^*_{X \times\{y\}}, -)$, we obtain the  exact sequence
$$
0 \to \Hom_{\SX}(\eE^*_{X \times\{y\}}, \HPsi^0\Psi^2(E)) 
\to \Hom_{\SX}(\eE^*_{X \times\{y\}},  \HPsi^2\Psi^1(E)) 
\to \Hom_{\SX}(\eE^*_{X \times\{y\}} , F) \to \cdots.
$$
Here $F\in \HN_{\lx, -D_{\SX}}^{\mu}((-\infty,0])$, and by Proposition~\ref{prop:slope-bound-FMT-0-g}--(i), $\HPsi^0\Psi^2$ is also in $\HN_{\lx, -D_{\SX}}^{\mu}((-\infty,0])$. Therefore,
  $\Hom_{\SX}(\eE^*_{X \times\{y\}} , F) \ne 0$ 
  or $\Hom_{\SX}(\eE^*_{X \times\{y\}}, \HPsi^0\Psi^2(E)) \ne 0$ for at
most a finite number of points $y \in Y$. 
Therefore, from Lemma~\ref{prop:reflexive-sheaf-threefold}, $\Psi^1(E)$ is a reflexive sheaf.
\end{proof}
For any positive
 integer $s$, the semihomogeneous bundle
 \begin{equation*}
 \widehat{\oO_{\HY}(sH_{\SHY })}  = \Phi^{\SHY \to \SY}_{\pP}(\oO_{\HY}(sH_{\SHY}))
\end{equation*}  
  is slope stable on $Y$.
In the rest of this section we abuse notation to write $\widehat{\oO_{\HY}(sH_{\SHY})}$ for the
functor $\widehat{\oO_{\HY}(sH_{\SHY})} \otimes(-)$.
\begin{prop}
\label{prop:slope-limit-check-higehst-HN}
Let $E_n \in  \HN_{\ly, D_{\SY}}^{\mu}([0, +\infty))$, $n \in \ZZ_{>0}$ be a sequence of coherent sheaves on $Y$.
For any $s>0$ there is $N(s) > 0$ such that for any $n > N(s)$ we have $\widehat{\oO_{\HY}(sH_{\SHY})}  E_n \in V_{\Coh(X)}^{\HPsi}(3)$.
Then
$\mu_{\ly, D_{\SY}}^+(E_n) \to 0$ as $n \to +\infty$.
\end{prop}
\begin{proof}
Let $s$ be a positive integer. Let us prove that for $n > N(s)$ we have
$\widehat{\oO_{\HY}(sH_{\SHY})} E_n \in \HN_{\ly, D_{\SY}}^{\mu}((-\infty,0])$. From the Harder-Narasimhan property there exists 
$T \in  \HN_{\ly, D_{\SY}}^{\mu}((0 , +\infty])$ and 
$F \in  \HN_{\ly, D_{\SY}}^{\mu}((-\infty,0])$ such that 
$$
0 \to T \to \widehat{\oO_{\HY}(sH_{\SHY})} E_n \to F \to 0
$$
is a short exact sequence in $\Coh(Y)$. 
By applying the  Fourier-Mukai  transform $\HPsi$ and considering the long exact sequence of $\Coh(X)$-cohomologies, we obtain $T \in V^{\HPsi}_{\Coh(X)}(2)$ and $F  \in V^{\HPsi}_{\Coh(X)}(1,3)$. Moreover, $\HPsi^2(T) \cong 
\HPsi^1(F)$. Hence, from the convergence of Mukai Spectral Sequence \ref{Spec-Seq-Mukai},    $T \cong \Psi^1\HPsi^2(T) \cong \Psi^1\HPsi^1(F) =0$. Therefore $\widehat{\oO_{\HY}(sH_{\SHY})} E_n \cong F  \in \HN_{\ly, D_{\SY}}^{\mu}((-\infty,0])$.

We have $\widehat{\oO_{\HY}(sH_{\SHY})}$ is slope stable with $\mu_{\ly,0} = -k/s$ for some constant $k >0$. Hence, for $n > N(s)$
$$
E_n \in \HN_{\ly, D_{\SY}}^{\mu}([0, k/s]).
$$
Therefore, the claim  follows by considering large enough $s$.
\end{proof}
Let $s$ be a positive integer. Consider the  Fourier-Mukai  functor from $D^b(X)$ to $D^b(X)$ defined by
$$
\Pi = \HPsi  \circ \widehat{\oO_{\HY}(sH_{\SHY})} \circ \Psi [3].
$$
Then  $\Pi^i(\oO_x)= 0$ for $i \ne 0$ and $\Pi^0(\oO_x)$ is a semistable semihomogeneous bundle on $X$.
Define the  Fourier-Mukai  functor
$$
\widehat{\Pi} = \HPsi \circ \widehat{\oO_{\HY}(sH_{\SHY})}^* \circ \Psi.
$$
One can see $\widehat{\Pi}[3]$ is right and left adjoint to $\Pi$ (and vice versa).
We have  $\widehat{\Pi}^i(\oO_x) = 0$ for $i \ne 0$, and
$\widehat{\Pi}^0(\oO_x)$ is a semistable semihomogeneous bundle on $X$.
Therefore, $\Pi$ is a Fourier-Mukai  functor with kernel a locally free sheaf $\mathcal{U}$ on $X \times X$.

We have the spectral sequence
\begin{equation}
\label{Spec-Seq-FM-Functor}
\HPsi^p  \left( \widehat{\oO_{\HY}(sH_{\SHY})} \  \Psi^q(E) \right) \Longrightarrow \Pi^{p+q-3} (E)
\end{equation}
for $E$.
\begin{prop}
\label{prop:slope-bound-FMT-torsion-1}
Let $E \in \Coh_1(X)$. Then $\mu_{\ly, D_{\SY}}^+(\Psi^1(E(-nH_{\SX}))) \to 0$ as $n \to +\infty$.
\end{prop}
\begin{proof}
Since $E \in \Coh_1(X)$,  for sufficiently large $n \in \ZZ_{>0}$, we have $E(-nH_{\SX}) \in V_{\Coh(Y)}^{\Psi}(1)$.
By Proposition~\ref{prop:slope-bound-torsion-sheaf}, $\Psi^1(E(-nH_{\SX})) \in \HN_{\ly, D_{\SY}}^{\mu}((0, +\infty))$.
Let $s$ be a positive integer. Consider the convergence of the Spectral Sequence \eqref{Spec-Seq-FM-Functor} for $E(-nH_{\SX})$.
For large enough $n \in \ZZ_{>0}$, we also have $E(-nH_{\SX}) \in V_{\Coh(X)}^{\Pi}(1)$.
Therefore, $\widehat{\oO_{\HY}(sH_{\SHY})} \Psi^1(E(-nH_{\SX}))  \in V_{\Coh(X)}^{\HPsi}(3)$, and so the claim follows from Proposition~\ref{prop:slope-limit-check-higehst-HN}.
\end{proof}
\begin{prop}
\label{prop:FMT-reflexive-large-negative-twist}
Let $E$ be  a reflexive sheaf. Then for sufficiently large $n \in \ZZ_{>0}$,
\begin{enumerate}[label=(\roman*)]
\item $E(-nH_{\SX}) \in V^{\Psi}_{\Coh(Y)}(2,3)$, and
\item $\Psi^2(E(-nH_{\SX})) \in \Psi^0(T_0)$ for some $T_0 \in \Coh_0(X)$.
\end{enumerate}
\end{prop}
\begin{proof}
(i) \ Consider  a minimal locally free resolution of $E$:
$$
0 \to F_2 \to F_1 \to E \to 0.
$$
By applying the  Fourier-Mukai  transform $\Psi \circ \oO_X(-nH_{\SX}) $ for  sufficiently large  $n \in \ZZ_{> 0}$, we obtain
$E(-nH_{\SX}) \in V^{\Psi}_{\Coh(Y)}(2,3)$ as required. \\

\noindent (ii) \
Since $E$ is a reflexive sheaf, there is a locally free sheaf $P$ and a torsion free sheaf $Q$ such that
$$
0 \to E \to P \to Q \to 0
$$
is a short exact sequence in $\Coh(X)$ (see Lemma~\ref{prop:reflexive-sheaf-results}--(2)).
By applying the  Fourier-Mukai  transform $\Psi \circ \oO_X(-nH_{\SX})$ for sufficiently large $n$ we have $\Psi^2(E(-nH_{\SX}) ) \cong \Psi^1(Q(-nH_{\SX}) )$.

The torsion free sheaf $Q$ fits into the short exact  sequence $0 \to Q \to Q^{**} \to T \to 0$ for
some $T \in  \Coh_{\le 1}(X)$. Apply the  Fourier-Mukai  transform $\Psi \circ \oO_X(-nH_{\SX})$ for sufficiently large $n$ and consider the long exact sequence  of $\Coh(Y)$-cohomologies. Since
$Q^{**}(-nH_{\SX}) \in V_{\Coh(Y)}^{\Psi}(2,3)$, we have
 $\Psi^1(Q(-nH_{\SX}) ) \cong  \Psi^0(T(-nH_{\SX}) )$.
The torsion sheaf $T \in \Coh_{\le 1}(X)$ fits into short exact sequence
$0 \to T_0 \to T \to T_1 \to 0$ in $\Coh(X)$ for $T_i \in \Coh_i(X)$, $i=0,1$.
Therefore,   $ \Psi^0(T(-nH_{\SX}) ) \cong  \Psi^0(T_0 )$, and 
so $\Psi^2(E(-nH_{\SX}) ) \cong \Psi^0(T_0 )$ as required.
\end{proof}
\begin{prop}
\label{prop:supprt-result-FMT}
Let $E \in \Coh_1(X)$ with $E \in V^{\Psi}_{\Coh(Y)}(1)$.
If $0 \ne T \in \HN_{\ly, D_{\SY}}^{\mu}([0, +\infty])$ is a subsheaf of $\Psi^1(E)$ then
$\ly \ch_2^{D_{\SY}}(T) \le 0$.
\end{prop}
\begin{proof}
Recall Note\ref{prop:restriction-setup}; 
choose $x \in X$ such that $\dim (\Supp(E) \cap H_{\SX,x}) \le 0$. Then
 for $n \in \ZZ_{>0}$, we have the short exact sequence
$$
0 \to E (-nH_{\SX, x}) \to E  \to T_0 \to 0
$$
in $\Coh(X)$, where $T_0  =  E|_{ n H_{\SX, x}} \in \Coh_0(X)$.

By applying the  Fourier-Mukai  transform $\Psi$ we get the following commutative diagram for some $A \in \HN_{\ly, D_{\SY}}^{\mu}([0, +\infty])$.
$$
\xymatrixcolsep{3pc}
\xymatrixrowsep{3pc}
\xymatrix{
0 \ar[r]  & \Psi^0(T_0) \ar[r]   & \Psi^1( E (-nH_{\SX, x}))  \ar[r]  &   \Psi^1(E) \ar[r] & 0 \\
0 \ar[r] & \Psi^0(T_0) \ar[r] \ar@{=}[u] & A \ar[r] \ar@{^{(}->}[u] &   T \ar[r] \ar@{^{(}->}[u] & 0
}
$$
For $k=1,2,3$, we have $\ch_k^{D_{\SY}}(\Psi^0(T_0))= 0$; so $\ch_k^{D_{\SY}}(A) = \ch_k^{D_{\SY}}(T)$.

Let $G$ be a slope semistable Harder-Narasimhan factor of $A$. Then, from the usual Bogomolov-Gieseker inequality,
we have
\begin{align*}
2 \ly \ch_2^{D_{\SY}}(G) & \le \frac{\left( \ly^2  \ch_1^{D_{\SY}}(G)\right)^2}{\ly^3 \ch_0^{D_{\SY}}(G)} \\
                & \le \ly^2 \ch_1^{D_{\SY}}(A) \ \mu_{\ly, D_{\SY}}(G) \\
                & \le  \ly^2 \ch_1^{D_{\SY}}(T) \   \mu_{\ly, D_{\SY}}^+ (\Psi^1( E (-nH_{\SX, x}))).
\end{align*}

Let 
$$
c_0 = \min \{2 \ly\ch_2^{D_{\SY}}(F)  >0 : F \in \Coh(Y)\}.
$$

By Proposition~\ref{prop:slope-bound-FMT-torsion-1}, $ \mu_{\ly, D_{\SY}}^+ (\Psi^1( E (-nH_{\SX, x}))) \to 0$ as $n \to +\infty$.
So choose large enough $n \in \ZZ_{>0}$ such that
$$
 \ly^2 \ch_1^{D_{\SY}}(T) \   \mu_{\ly, D_{\SY}}^+ (\Psi^1( E (-nH_{\SX, x}))) < c_0;
$$  
hence, we have $\ly \ch_2^{D_{\SY}}(G)  \le 0$.
Therefore, $\ly \ch_2^{D_{\SY}}(T) = \ly \ch_2^{D_{\SY}}(A) \le 0$.
\end{proof}

\begin{prop}
\label{prop:FMT-divisor-restrict-large-negative-twist}
Let $E$ be a reflexive sheaf on $X$.
Therefore, $\dim \Sing(E) \le 0$, and so for generic $ x \in X$
we have $\Sing(E) \cap  H_{\SX, x} = \emptyset$.

Let $m$ be any positive integer. For large enough $n \in \ZZ_{>0}$,
$$
 E(-n H_{\SX})|_{mH_{\SX, x}} \in V^{\Psi}_{\Coh(Y)}(2).
$$
\end{prop}
\begin{proof}
The dual sheaf $E^*$ is also reflexive (Lemma~\ref{prop:reflexive-sheaf-results}--(4)).
Consider a minimal locally free resolution of $E^*$:
$$
0 \to G \to F \to E^* \to 0.
$$
By applying the dualizing functor $\RRR \calHom(- ,\oO_X)$ to this short exact sequence, we get the following long exact sequence in $\Coh(X)$:
$$
0 \to E \to F^* \to G^* \to \calExt^1(E^*, \oO_X) \to 0.
$$
Let $Q = \coker(E \to F^*)$.
Since $E$ is reflexive,
$$
\Sing(E) = \Sing(E^*) = \Supp(\calExt^1(E^*, \oO_X)).
$$
By choice  $\Sing(E) \cap  H_{\SX, x} = \emptyset$, and so from the short exact sequence
$0 \to Q \to  G^* \to \calExt^1(E^*, \oO_X) \to 0$,
$Q|_{mH_{\SX, x}} \cong G^*|_{mH_{\SX, x}}$.
So we have the short exact sequence
$$
0 \to  E|_{mH_{\SX, x}} \to  F^*|_{mH_{\SX, x}} \to  G^*|_{mH_{\SX, x}} \to 0
$$
in $\Coh(X)$.
Since $F^*$ and $G^*$ are locally free, for large enough $n \in \ZZ_{>0}$  we have
$$
F^*(-nH_{\SX}) |_{mH_{\SX, x}} ,  G^*(-nH_{\SX}) |_{mH_{\SX, x}}  \in V^{\Psi}_{\Coh(Y)}(2),
$$ 
and so
$ E(-n H_{\SX})|_{mH_{\SX, x}} \in V^{\Psi}_{\Coh(Y)}(2)$.
\end{proof}
\begin{prop}
\label{prop:FMT-bridge-result}
We have the following:
\begin{enumerate}[label=(\roman*)]
\item
Let $E \in   \HN_{\lx, -D_{\SX}}^{\mu}((-\infty,0])$  be a reflexive sheaf.
If $T \in   \HN_{\ly, D_{\SY}}^{\mu}([0, +\infty])$ is a non-trivial subsheaf of $\Psi^1(E)$
 then $\ly \ch_2^{D_{\SY}}(T) \le  0$.

\item
Let  $E \in \HN_{\lx, -D_{\SX}}^{\mu}((0, +\infty])$  be a torsion free sheaf.
If  $ F \in  \HN_{\ly, D_{\SY}}^{\mu}((-\infty,0])$ is a non-trivial quotient of 
$\Psi^2(E)$ then $\ly \ch_2^{D_{\SY}}(F) \le  0$.
\end{enumerate}
\end{prop}
\begin{proof}
(i) \
Since $E$ is reflexive, $\dim \Sing(E) \le 0$.
Choose $x, x' \in X$ such that
\begin{itemize}
\item $\dim (H_{\SX, x} \cap H_{\SX, x'}) = 1$,
\item $\Sing(E) \cap  H_{\SX, x} = \emptyset$, and
\item $\Sing(E) \cap  H_{\SX, x'} = \emptyset$.
\end{itemize}

Since $E$ is a reflexive sheaf, Proposition~\ref{prop:FMT-reflexive-large-negative-twist}--(i) implies, for sufficiently large $m \in \ZZ_{>0}$,
 $E(-mH_{\SX, x}) \in V^{\Psi}_{\Coh(Y)}(2,3)$.
By applying the  Fourier-Mukai  transform $\Psi$ to the short exact sequence
$$
0 \to E(-mH_{\SX, x}) \to E \to E|_{mH_{\SX, x}} \to 0
$$
in  $\Coh(X)$ and then considering the long exact sequence of 
$\Coh(Y)$-cohomologies, 
we have $E|_{mH_{\SX, x}} \in V^{\Psi}_{\Coh(Y)}(1,2)$ and
 $\Psi^1(E) \hookrightarrow \Psi^1\left(E|_{mD_x} \right)$.
By Proposition~\ref{prop:FMT-divisor-restrict-large-negative-twist}, for large enough $n \in \ZZ_{>0}$, 
$ E(-n H_{\SX})|_{mH_{\SX, x}} \in
V^{\Psi}_{\Coh(Y)}(2)$.
By applying the  Fourier-Mukai  transform $\Psi$ to the short exact sequence
$$
0 \to E(-n H_{\SX, x'})|_{mD_x} \to E|_{mD_x} \to E|_{mH_{\SX, x} \cap nH_{\SX, x'}} \to 0
$$
in $\Coh(X)$ and then considering the long exact sequence of 
$\Coh(Y)$-cohomologies, we get 
$E|_{mH_{\SX, x} \cap nH_{\SX, x'}} \in V^{\Psi}_{\Coh(Y)}(1)$ 
and $\Psi^1 \left(E|_{mH_{\SX, x}} \right)  \hookrightarrow \Psi^1 \left(E|_{mH_{\SX, x} \cap nH_{\SX, x'}} \right) $.
Therefore, we have
$$
T \hookrightarrow \Psi^1(E)   
\hookrightarrow \Psi^1 \left(E|_{mH_{\SX, x}} \right) 
 \hookrightarrow \Psi^1 \left(E|_{mH_{\SX, x} \cap nH_{\SX, x'}} \right).
$$
The claim follows from Proposition~\ref{prop:supprt-result-FMT}. \\

\noindent (ii) \
Since $F \ne 0$ is a quotient of $\Psi^2(E)$, we have
$F^{*} \hookrightarrow \left(\Psi^2(E)\right)^*$.
Here $F^* \in \HN^{\mu}_{\ly, -D_{\SY}}([0, +\infty))$ fits into short exact sequence
$0 \to T \to F^* \to F_0 \to 0$ in $\Coh(Y)$
for some $T \in \HN^{\mu}_{\ly, -D_{\SY}}((0, +\infty))$ and $F_0 \in \HN^{\mu}_{\ly, -D_{\SY}}(0)$.
By the usual Bogomolov-Gieseker inequality $\ly \ch_2^{-D_{\SY}}(F_0) \le 0$.

Let 
$$
\TPsi := \Phi^{\SX \to \SY}_{\eE^\vee}: D^b(X) \to D^b(Y).
$$
By Proposition~\ref{prop:FMT-0-g-cohomo-vanishing}, $\Psi^3(E) = 0 = \TPsi^0(E^*)$.

Consider the co-convergence of the ``Duality'' Spectral Sequence~\ref{Spec-Seq-Dual}
for $E$ and the following diagram describes its second page.

$$
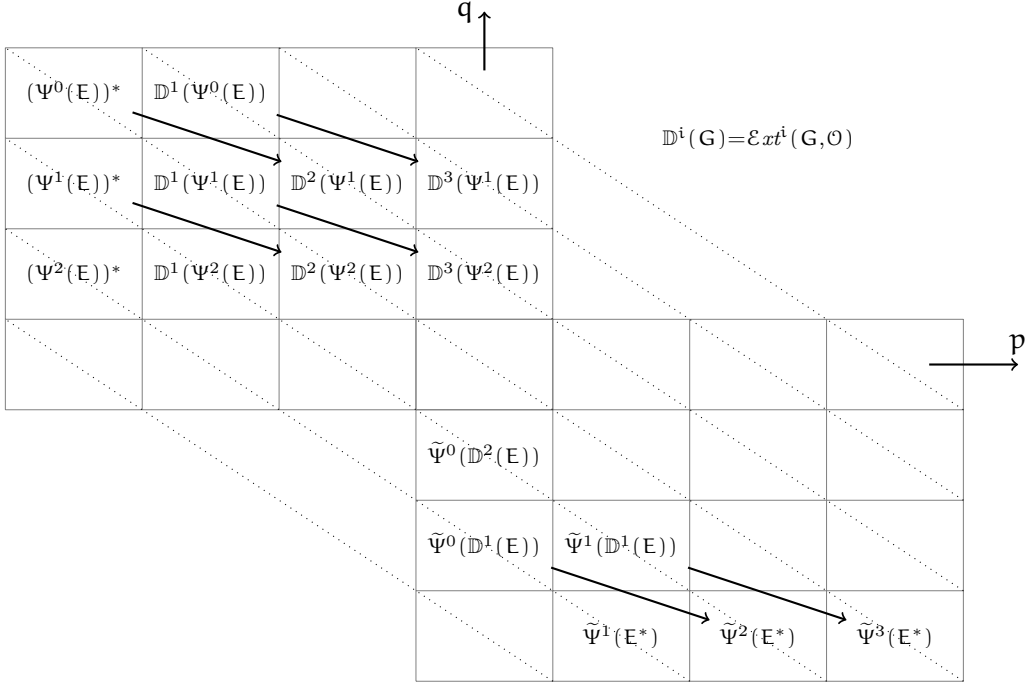
\begin{tikzpicture}[xscale=1.8, yscale=1.2]
\draw[gray,very thin] (-3,0) grid (1,4);
\draw[gray,very thin] (0,-3) grid (4,1);
\draw[->,thick] (3.75,0.5) -- (4.4,0.5) node[above] {$p$};
\draw[->,thick] (0.5,3.75) -- (0.5,4.4) node[left] {$q$};
\draw[style=dotted] (2,-3) -- (-3,2);
\draw[style=dotted] (3,-3) -- (-3,3);
\draw[style=dotted] (4,-3) -- (-3,4);
\draw[style=dotted] (1,-3) -- (-3,1);
\draw[style=dotted] (4,-2) -- (-2,4);
\draw[style=dotted] (4,-1) -- (-1,4);
\draw[style=dotted] (4,0) -- (0,4);
\draw (2.5,3) node {$\scriptstyle\DD^i(G) = \calExt^i(G, \oO)$};
\draw (-2.5,1.5) node {$\scriptstyle(\Psi^2(E))^{*}$};
\draw (-1.5,1.5) node {$\scriptstyle \DD^1(\Psi^2(E))$};
\draw (-0.5,1.5) node (f2) {$\scriptstyle\DD^2(\Psi^2(E))$};
\draw (0.5,1.5) node (f4) {$\scriptstyle\DD^3(\Psi^2(E))$};
\draw (-2.5,2.5) node (f1) {$\scriptstyle(\Psi^1(E))^{*}$};
\draw (-1.5,2.5) node (f3) {$\scriptstyle\DD^1(\Psi^1(E))$};
\draw (-0.5,2.5) node (f6) {$\scriptstyle\DD^2(\Psi^1(E))$};
\draw (0.5,2.5) node (f8) {$\scriptstyle\DD^3(\Psi^1(E))$};
\draw (-2.5,3.5) node (f5) {$\scriptstyle(\Psi^0(E))^{*}$};
\draw (-1.5,3.5) node (f7) {$\scriptstyle\DD^1(\Psi^0(E))$};
\draw[->,thick] (f1) -- node[above] {$ $} (f2);
\draw[->,thick] (f3) -- node[above] {$ $} (f4);
\draw[->,thick] (f5) -- node[above] {$ $} (f6);
\draw[->,thick] (f7) -- node[above] {$ $} (f8);
\draw (1.5,-2.5) node {$\scriptstyle \TPsi^1(E^*)$};
\draw (2.5,-2.5) node (s2) {$\scriptstyle \TPsi^2(E^*)$};
\draw (3.5,-2.5) node (s4) {$\scriptstyle  \TPsi^3(E^*)$};
\draw (0.5,-1.5) node (s1) {$\scriptstyle  \TPsi^0(\DD^1(E))$};
\draw (1.5,-1.5) node (s3) {$\scriptstyle  \TPsi^1(\DD^1(E))$};
\draw (0.5,-0.5) node  {$\scriptstyle  \TPsi^0(\DD^2(E))$};
\draw[->,thick] (s1) -- node[above] {$ $} (s2);
\draw[->,thick] (s3) -- node[above] {$ $} (s4);
\end{tikzpicture}
$$

We have the short exact sequence
$$
0 \to  \TPsi^1 (E^*) \to \left(\Psi^2(E)\right)^*  \to P \to 0
$$
in $\Coh(Y)$, for some subsheaf $P$ of $\TPsi^0(\calExt^1(E, \oO_X))$.
By Proposition~\ref{prop:slope-bound-FMT-0-g}--(i), $\TPsi^0(\calExt^1(E, \oO_X)) \in \HN^{\mu}_{\ly, -D_{\SY}}((-\infty, 0])$ and so $P \in \HN^{\mu}_{\ly, -D_{\SY}}((-\infty, 0])$.
Therefore, $\Hom_{\SY} (T, P) = 0$, and so $T \hookrightarrow   \TPsi^1 (E^*) $.
Here $E^* \in \HN_{\ly, D_{\SX}}^{\mu}((-\infty,0))$ and so by part (i), $\ly \ch_2^{-D_{\SY}}(T) \le 0$.
Therefore,  
$$
\ly \ch_2^{D_{\SY}}(F) \le \ly \ch_2^{D_{\SY}}(F^{**})=  \ly \ch_2^{-D_{\SY}} (F^{*}) = \ly \ch_2^{-D_{\SY}}(F_0) + \ly \ch_2^{-D_{\SY}}(T) \le 0
$$
as required.
\end{proof}
\begin{prop}
\label{prop:slope-bound-FMT-F-1}
For $E \in \Coh(X)$, we have the following:
\begin{enumerate}[label=(\roman*)]
\item If $E \in \HN^{\mu}_{\lx, -D_{\SX}}((-\infty , 0])$ then  $\Psi^1(E) \in  \HN^{\mu}_{\ly, D_{\SY}}((-\infty , 0])$, and
\item If $E \in \HN^{\mu}_{\lx, -D_{\SX}}([0,+\infty))$ with $\Psi^3(E)=0$ then  $\Psi^2(E) \in  \HN^{\mu}_{\ly, D_{\SY}}([0,+\infty])$.
\end{enumerate}
\end{prop}
\begin{proof}
(i) \
Assume the opposite for a contradiction. From the Harder-Narasimhan filtration of $\Psi^1(E)$ there exists 
$T \in \HN^{\mu}_{\ly, D_{\SY}}((0, +\infty])$ and $F \in \HN^{\mu}_{\ly, D_{\SY}}((-\infty , 0])$ such that 
\begin{equation}
\label{eqn:local-ses-FMT-1}
0 \to T \to \Psi^1(E) \to F \to 0
\end{equation}
is a short exact sequence in $\Coh(Y)$.
By Proposition~\ref{prop:FMT-0-cohomology-reflexive}--(ii), $\Psi^1(E)$ is reflexive.
Therefore, 
\begin{equation}
\label{eqn:local-ch1-bound}
\ly \ch_1^{D_{\SY}}(T) > 0.
\end{equation}
By Lemma \ref{prop:reflexive-sheaf-results}--(2) there exists a
locally free sheaf $G_1$ such that $\Psi^1(E)$ is a subsheaf of it with a torsion free  quotient sheaf $G_1/ \Psi^1(E)$.
 Hence,  $T$ is a a subsheaf of $G_1$ with a torsion free quotient sheaf. 
 Therefore, again by Lemma \ref{prop:reflexive-sheaf-results}--(2),  
$T$ is   a  reflexive sheaf.

By applying the functor $\RRR \calHom(-, \oO_Y)$ to the short exact sequence \eqref{eqn:local-ses-FMT-1}, we obtain the long exact sequence:
$$
0 \to F^{*} \to \left( \Psi^1(E)\right)^* \to T^* \to \calExt ^1(F, \oO_Y) \to \calExt ^1(\Psi^1(E), \oO_Y) \to \calExt ^1(T, \oO_Y) \to \calExt ^2(F, \oO_Y) \to 0.
$$
Since $\Psi^1(E)$ and $T$ are reflexive, 
$  \calExt^1(\Psi^1(E), \oO_Y)  , \calExt ^1(T, \oO_Y) \in \Coh_0(Y)$, and so
$$
 \calExt ^1(F, \oO_Y) ,   \calExt ^2(F, \oO_Y)  \in \Coh_0(Y).
$$
Therefore $F$ fits into the short exact sequence 
$$
0 \to F \to F^{**} \to R \to 0
$$
for some $R \in \Coh_0(Y)$. By applying the Fourier-Mukai  transform $\HPsi$, we get 
the short exact sequence 
$$
0 \to \HPsi^0(R) \to \HPsi^1(F) \to \HPsi^1(F^{**}) \to 0.
$$
From the Harder-Narasimhan filtration, let 
$T_1$ be the subsheaf of $\HPsi^1(F^{**})$ in 
$\HN^{\mu}_{\lx, -D_{\SX}}((0, +\infty])$
with the quotient $F_1 \in \HN^{\mu}_{\lx, -D_{\SX}}((-\infty , 0])$.
Then $\HPsi^1(F)$ has a subsheaf $T_2 \in \HN^{\mu}_{\lx, -D_{\SX}}([0, +\infty])$ with 
quotient $F_1$. 
Here $T_2$ fits into the short exact sequence 
$$
0 \to \HPsi^0(R) \to T_2 \to   T_1 \to 0.
$$
From Proposition \ref{prop:FMT-bridge-result}--(i), $\lx \ch_2^{-D_{\SX}}(T_1) \le 0$, and since $\ch_2^{-D_{\SX}}(\HPsi^0(R))=0$,
$$
\lx \ch_2^{-D_{\SX}}(T_2) \le 0.
$$

By applying the  Fourier-Mukai  transform $\HPsi$ to the  short exact sequence \eqref{eqn:local-ses-FMT-1}, we obtain that
$T \in V_{\Coh(X)}^{\HPsi}(2)$ and $F \in V_{\Coh(X)}^{\HPsi}(1,2,3)$.
Moreover, we have the short exact sequence
$$
0 \to \HPsi^1(F) \to \HPsi^2(T) \to  E_1 \to 0
$$
in $\Coh(X)$
for some subsheaf $E_1$ of $\HPsi^2 \Psi^1(E)$.
From the Mukai Spectral Sequence~\ref{Spec-Seq-Mukai}
for $E$, we have the short exact sequence
$$
0 \to \HPsi^0 \Psi^2(E)   \to \HPsi^2 \Psi^1(E)   \to E_2 \to 0
$$
in $\Coh(X)$ for some subsheaf $E_2$ of $E$. Therefore, $E_2 \in  \HN^{\mu}_{\lx, -D_{\SX}}((-\infty , 0])$. 
By Proposition~\ref{prop:slope-bound-FMT-0-g}--(i),
$ \HPsi^0 \Psi^2(E) \in \HN^{\mu}_{\lx, -D_{\SX}}((-\infty , 0])$. 
So we have $ \HPsi^2 \Psi^1(E)  \in \HN^{\mu}_{\lx, -D_{\SX}}((-\infty , 0])$. 
Hence,  $E_1 \in \HN^{\mu}_{\lx, -D_{\SX}}((-\infty , 0])$.

So we have the following commutative diagram for some $F_2 \in \HN^{\mu}_{\lx, -D_{\SX}}((-\infty , 0])$.

$$
\xymatrixcolsep{3pc}
\xymatrixrowsep{2.2pc}
\xymatrix{
 & 0 & 0 &   &  \\
 0 \ar[r] & F_1 \ar[r]\ar[u] & F_2 \ar[r]\ar[u] & E_1 \ar[r] & 0 \\
 0 \ar[r] & \HPsi^1(F) \ar[r] \ar[u] &  \HPsi^2(T)  \ar[r] \ar[u] &  E_1 \ar[r] \ar@{=}[u] & 0 \\
    & T_2 \ar[u] \ar@{=}[r]  & T_2 \ar[u] &   &  \\
     & 0 \ar[u] & 0 \ar[u] &   &   }
$$

By Proposition~\ref{prop:FMT-bridge-result}--(ii),   $\lx \ch_2^{-D_{\SX}}(F_2) \le 0$. Therefore, 
$$
\lx \ch_2^{-D_{\SX}}(\HPsi^2(T)) = \lx \ch_2^{-D_{\SX}}(T_2) + \lx \ch_2^{-D_{\SX}}(F_2) \le 0.
$$
So from Theorem \ref{prop:antidiagonal-rep-cohom-FMT}, $\ly \ch_2^{D_{\SY}}(T) \le 0$; but this is not possible as we have \eqref{eqn:local-ch1-bound}.
This is the required contradiction. \\

\noindent  (ii) \
Let $\TPsi:= \Phi^{\SX \to \SY}_{\eE^\vee}$.
Since $E^* \in  \HN^{\mu}_{\lx, D_{\SX}}((-\infty,0])$, from (i)  $\TPsi^1(E^*) \in \HN^{\mu}_{\ly, -D_{\SY}}((-\infty,0])$.
By the co-convergence of the ``Duality'' Spectral Sequence~\ref{Spec-Seq-Dual}
for $E$, we have $(\Psi^2(E))^* \in  \HN^{\mu}_{\ly, -D_{\SY}}((-\infty,0])$.
So $\Psi^2(E) \in  \HN^{\mu}_{\ly, D_{\SY}}([0,+\infty])$ as required.
\end{proof}

\begin{prop}
\label{prop:slope-bound-FMT-T-2}
Let $E \in \HN^{\mu}_{\lx, -D_{\SX}}((0, + \infty])$. Then $\Psi^2(E) \in \HN^{\mu}_{\ly, D_{\SY}}((0,+\infty])$.
\end{prop}
\begin{proof}
From the Harder-Narasimhan filtration of $\Psi^2(E)$, there exist 
$T \in  \HN^{\mu}_{\ly, D_{\SY}}((0,+\infty])$ and $F \in  \HN^{\mu}_{\ly, D_{\SY}}((-\infty, 0])$ such that 
$0 \to T \to \Psi^2(E) \to F \to 0$ is a short exact sequence  in $\Coh(Y)$.
Now we need to show $F = 0$.
Apply the  Fourier-Mukai  transform  $\HPsi$ and consider the long exact sequence of $\Coh(X)$-cohomologies. 
So we have $F \in V^{\HPsi}_{\Coh(X)}(1)$ and
$$
0 \to \HPsi^1(T)  \to \HPsi^1 \Psi^2(E)  \to \HPsi^1(F)  \to \HPsi^2(T) \to 0
$$
is a long exact sequence in $\Coh(X)$. 
From the convergence of the Mukai Spectral
Sequence~\ref{Spec-Seq-Mukai} for $E$, we have the short exact sequence
$$
0 \to Q \to \HPsi^1 \Psi^2(E) \to \HPsi^3 \Psi^1(E) \to 0
$$
in $\Coh(Y)$, where $Q$ is a quotient of $E$. Then $Q\in \HN^{\mu}_{\lx, -D_{\SX}}((0, + \infty])$ and by
Proposition~\ref{prop:slope-bound-torsion-sheaf}, $\HPsi^3 \Psi^1(E) \in \HN^{\mu}_{\lx, -D_{\SX}}((0, + \infty])$;  
so $\HPsi^1 \Psi^2(E)  \in \HN^{\mu}_{\lx, -D_{\SX}}((0, + \infty])$.
On the other hand, by Proposition~\ref{prop:slope-bound-FMT-F-1}--(i), $\HPsi^1(F) \in  \HN^{\mu}_{\lx, -D_{\SX}}((-\infty,0])$.
So the map $ \HPsi^1 \Psi^2(E)  \to \HPsi^1(F) $ is zero and  $\HPsi^1(F)  \cong \HPsi^2(T)$. Hence,
 $F \cong \Psi^2 \HPsi^1(F)  \cong  \Psi^2 \HPsi^2(T) = 0$  as required.
\end{proof}
\section{Further Properties of Slope Stability under FM Transforms}
\label{sec:further-FMT-sheaves-abelian-threefolds}
\subsection{Some slope bounds of the FM transformed sheaves}
\label{sec:FMT-0-special-bound}
Recall  that $\Psi$ is the  Fourier-Mukai  transform $\Phi^{\SX \to \SY}_{\eE}:D^b(X) \to D^b(Y)$ between the abelian threefolds such that  
$$
\ch(\eE_{\{x\} \times Y}) = r \, e^{D_{\SY}}, \ \text{and} \ \ch(\eE_{X \times \{y\}} ) = r \, e^{D_{\SX}}.
$$
Also $\lx \in \NS_{\QQ}(X), \ly \in \NS_{\QQ}(Y)$ are some ample classes such that
$$
 e^{- D_{\SY}} \, \Phi_{\eE}^{\SH} \, e^{-D_{\SX}} ( e^{\ell_{\SX}}) =  (r \, {\lx^3}/{3!}) \,  e^{-\ell_{\SY}},
$$
with $({\lx^3}/3!)({\ly^3}/3!)=  1/r^2$.
Moreover, Theorem \ref{prop:antidiagonal-rep-cohom-FMT} says, 
if we consider $v^{-D_{\SX} ,\lx}, v^{D_{\SY},\ly}$ as column vectors, then 
$$
v^{D_{\SY},\ly}\left(\Phi_\eE^{\SX  \to \SY}(E) \right) = \frac{3!}{r \, \ell_{\SX}^3} \,  \Adiag\left(1,-1,1,-1\right)  \ v^{-D_{\SX}, \lx}(E).
$$
Here the vector $v^{B,{\lx}}(E)$ is defined by 
\begin{equation*}
v^{B,{\lx}}(E) = \left( v^{B,{\lx}}_0(E) , v^{B,{\lx}}_1(E) , v^{B,{\lx}}_2(E) , v^{B,{\lx}}_3(E) \right).
\end{equation*}
\begin{prop}
\label{prop:special-slope-bound-FMT-0-3}
For $\lambda \in \QQ_{> 0}$,
\begin{enumerate}[label=(\roman*)]
\item if $E \in \HN_{\lx, -D_{\SX}}^{\mu}((0,\lambda])$
 then $\Psi^0(E) \in \HN_{\ly, D_{\SY}}^{\mu}((-\infty, -\frac{1}{\lambda}])$,
\item if $E \in  \HN_{\lx, -D_{\SX}}^{\mu}([-\lambda,0])$ 
 then $\Psi^3(E) \in \HN_{\ly, D_{\SY}}^{\mu}([\frac{1}{ \lambda }, +\infty])$.
\end{enumerate}
\end{prop}
\begin{proof}
(i) \
Let $E \in \HN_{\lx, -D_{\SX}}^{\mu}((0,\lambda])$.

 Let $Z$ be the fine moduli space of simple semihomogeneous bundles $E$ on $X$ with $c_1(E)/\rk(E) =D_{\SX} - \lambda \lx$. Then there is some fixed $r' \in\ZZ_{>0}$ such that
 $$
r' = \rk(E) 
 $$
  for such $E$. 
 Due to Mukai and Orlov,  $Z$ is an abelian threefold. Let $\fF$ be the associated universal bundle on $Z \times X$; so by Lemma \ref{prop:semihomo-numerical}--(1) we have 
$$
\ch(\fF_{ \{z\} \times X   }) = r' \, e^{ D_{\SX} -  \lambda \lx }.
$$
 Let 
 $$
 \Pi: = \Phi_\fF^{\SX \to \SZ} : D^b(X) \to D^b(Z)
 $$
  be the corresponding Fourier-Mukai transform from $D^b(X)$ to $D^b(Z)$ 
  with kernel $\fF$.  
 Then its quasi inverse is given by 
$\Phi_{\fF^\vee}^{\SZ \to \SX}[3]$. Again, by Lemma \ref{prop:semihomo-numerical}--(2)
$$
\ch(\fF_{Z \times \{x\}} ) = r' \, e^{D_{\SZ}}
$$
for some $D_{\SZ} \in \NS_{\QQ}(Z)$.
Similar to   the  Fourier-Mukai  transform $\Psi$ in Section \ref{sec:cohomological-FMT}, there exists an  ample class $\lz \in \NS_{\QQ}(Z)$ such that
$$
 e^{- D_{\SZ}} \, \Pi^{\SH} \, e^{-D_{\SX}+\lambda \lx} ( e^{\ell_{\SX}}) =  (r' \, {\lx^3}/{3!}) \,  e^{-\ell_{\SZ}},
$$
with $({\lx^3}/3!)({\lz^3}/3!)=  1/r'^2$ (Theorem \ref{prop:general-cohomo-FMT}). 
Moreover, Theorem \ref{prop:antidiagonal-rep-cohom-FMT} says,  
\begin{equation*}
v^{D_{\SZ},\lz}\left(\Pi(E) \right) = \frac{3!}{r' \, \ell_{\SX}^3} \,  \Adiag\left(1,-1,1,-1\right)  \ v^{-D_{\SX}+\lambda \lx, \lx}(E).
\end{equation*}

Let $\Xi: D^b(Y) \to D^b(Z)$ be the  Fourier-Mukai  transform defined by
$$
\Xi := \Pi  \circ \HPsi[3].
$$
We have $\HPsi(\oO_y) = \eE^*_{X \times \{y\}}$ is  a stable semihomogeneous bundle in 
$\HN^{\mu}_{\lx, -D_{\SX}}(0) = \HN^{\mu}_{\lx, -D_{\SX}+\lambda \lx}(-\lambda)$. Therefore
 the image 
 $
 \Xi(\oO_y) = \Pi ( \eE^*_{X \times \{y\}}[3])
 $
 of the skyscraper sheaf $\oO_y$ is also a stable semihomogeneous bundle on $Z$.
 Hence, $\Xi$ is a  Fourier-Mukai  transform $\Phi^{\SY \to \SZ}_{\gG}$ with kernel $\gG$ on $Y \times Z$ such that 
  $$
 \gG_{\{y\} \times Z} = \Xi(\oO_y) = \Pi( \eE^*_{X \times \{y\}}[3]).
 $$
From Theorem \ref{prop:general-cohomo-FMT} and Lemma \ref{prop:reflexive-sheaf-results}, there is $r''>0$ such that  
\begin{align*}
& \ch( \gG_{\{y\} \times Z}) = r'' e^{D_{\SZ} + \frac{1}{\lambda}\lz}, \\
&  \ch( \gG^*_{Y \times \{z\}}) = r'' e^{D_{\SY} - \frac{1}{\lambda}\ly}.
\end{align*}

The isomorphism $\Xi \circ \Psi \cong \Pi$  gives us the
convergence of  the spectral sequence:
\begin{equation}
\label{Spec-Seq-Local-FMT-0-bound}
E_2^{p,q} = \Xi^p  \Psi^q (E) \Longrightarrow \Pi^{p+q}(E)
\end{equation}
for $E$. 

Since  $E \in \HN^{\mu}_{\lx, -D_{\SX} + \lambda \lx}((-\lambda,0])$, 
from   Proposition~\ref{prop:FMT-0-g-cohomo-vanishing}--(iii), 
\begin{equation*}
\Pi^0 (E) = 0.
\end{equation*}
Now from the convergence of the above spectral sequence \eqref{Spec-Seq-Local-FMT-0-bound},
$\Xi^0  \Psi^0 (E) = 0$ and
\begin{equation*}
\Xi^1  \Psi^0 (E) \hookrightarrow \Pi ^1(E).
\end{equation*}
By Proposition~\ref{prop:slope-bound-FMT-F-1}--(i), $\Pi^1(E) \in \HN^{\mu}_{\lz, D_{\SZ}}((-\infty, 0])$.
Since we have $ \HN^{\mu}_{\lz, D_{\SZ}}((-\infty, 0])  \subset  \HN^{\mu}_{\lz, D_{\SZ}+\frac{1}{\lambda}\lz}((-\infty, 0])$,
\begin{equation}
\label{firstbound}
\Xi^1  \Psi^0 (E) \in \HN^{\mu}_{\lz, D_{\SZ}+\frac{1}{\lambda}\lz}((-\infty, 0]).
\end{equation}

From the Harder-Narasimhan filtration property, $\Psi^0(E) \in \HN^{\mu}_{\ly, D_{\SY}}((-\infty, 0])$ fits into the short exact sequence
\begin{equation}
\label{mu-ses}
0 \to F \to \Psi^0(E) \to G \to 0
\end{equation}
in $\Coh(Y)$ for some
 $F \in  \HN^{\mu}_{\ly, D_{\SY}}((-\frac{1}{ \lambda },0])$ 
 and $G \in  \HN^{\mu}_{\ly, D_{\SY}}((-\infty , -\frac{1}{ \lambda}])$. 
 Assume $F \ne 0$ for a
    contradiction.
Then we can write $v^{D_{\SY}, \ly}(F) = (a_0, \mu a_0, a_2, a_3)$ with
\begin{equation*}
0 \ge \mu > - \frac{1}{ \lambda}.
\end{equation*} 

By applying the  Fourier-Mukai  transform $\HPsi$ to short exact sequence \eqref{mu-ses} we have the
following exact sequence in $\Coh(X)$:
$$
0 \to \HPsi^1(G) \to \HPsi^2(F) \to \HPsi^2\Psi^0(E) \to \cdots.
$$
By Mukai Spectral Sequence \ref{Spec-Seq-Mukai},
 $\HPsi^2 \Psi^0(E) \cong \HPsi^0 \Psi^1(E)$ and
so by Proposition~\ref{prop:slope-bound-FMT-0-g}--(i), it is in $\HN^{\mu}_{\lx, -D_{\SX}}((-\infty, 0])$.
Also by  Proposition~\ref{prop:slope-bound-FMT-F-1}--(i), 
$\HPsi^1(G)  \in  \HN^{\mu}_{\lx, -D_{\SX}}((-\infty, 0])$. 
Therefore, $\HPsi^2(F) \in \HN^{\mu}_{\lx, -D_{\SX}}((-\infty, 0])$.
By Proposition~\ref{prop:slope-bound-torsion-sheaf}, $\HPsi^3(F) \in \HN^{\mu}_{\lx, -D_{\SX}}((0, +\infty])$.
Therefore, we have $v_1^{-D_{\SX},\lx}(\HPsi(F)) =  \lx^2 \ch_1^{-D_{\SX}}(\HPsi(F)) \le 0$, and so 
from Theorem \ref{prop:antidiagonal-rep-cohom-FMT}
\begin{equation*}
a_2 \ge 0.
\end{equation*}

Since $\Xi^0(F) \hookrightarrow \Xi^0 \Psi^0(E) = 0$,  we have $\Xi^0(F) =0$.
Moreover, since 
$$
F \in  \HN^{\mu}_{\ly, D_{\SY}}((-\frac{1}{ \lambda },0])=
  \HN^{\mu}_{\ly, D_{\SY}-\frac{1}{\lambda}\ly}((0, \frac{1}{\lambda}]),
  $$
from Proposition \ref{prop:FMT-0-g-cohomo-vanishing}--(i)   we have
    \begin{equation*}
    \Xi^3(F) = 0.
    \end{equation*}
Apply the  Fourier-Mukai  transform $\Xi$ to  short exact sequence  \eqref{mu-ses} and consider the long exact sequence of $\Coh(Z)$-cohomologies:
$$
0 \to \Xi^0(G) \to \Xi^1(F) \to \Xi^1 \Psi^0 (E)  \to \cdots.
$$
By \eqref{firstbound}, $\Xi^1 \Psi^0 (E) \in \HN^{\mu}_{\lz, D_{\SZ}+\frac{1}{\lambda}\lz}((-\infty, 0])$, and
by  Proposition~\ref{prop:slope-bound-FMT-0-g}--(i), $\Xi^0(G) \in \HN^{\mu}_{\lz, D_{\SZ}+\frac{1}{\lambda}\lz}((-\infty, 0])$.
Therefore, $\Xi^1(F) \in \HN^{\mu}_{\lz, D_{\SZ}+\frac{1}{\lambda}\lz}((-\infty, 0])$.
By  Proposition~\ref{prop:slope-bound-FMT-T-2},
$\Xi^2(F) \in \HN^{\mu}_{\lz, D_{\SZ}+\frac{1}{\lambda}\lz}((0,+\infty])$.
So 
\begin{equation}
\label{first-v1-bound}
v_1^{D_{\SZ}+\frac{1}{\lambda}\lz, \lz} (\Xi(F)) \ge 0.
\end{equation} 

On the other hand, we have
\begin{align*}
& v^{D_{\SZ}+\frac{1}{\lambda}\lz, \lz} (\Xi(F)) \\
& \quad \quad \quad = \frac{3!}{r''\ly^3} \begin{pmatrix}
 &  &  & 1 \\
& & -1 &  \\
&1 & &   \\
-1 &   &  &  
\end{pmatrix}  \ v^{D_{\SY}-\frac{1}{\lambda}\ly, \ly}(F) \\
&  \quad \quad \quad = 
\frac{3!}{r''\ly^3} \begin{pmatrix}
 &  &  & 1 \\
& & -1 &  \\
&1 & &   \\
-1 &   &  &  
\end{pmatrix} 
 \begin{pmatrix}
1 &  &  &   \\
\frac{1}{\lambda} & 1 &   &  \\
\frac{1}{\lambda^2}&\frac{1}{\lambda} & 1 &   \\
\frac{1}{\lambda^3} & \frac{1}{\lambda^2}  & \frac{1}{\lambda} &  1
\end{pmatrix} \ v^{D_{\SY}, \ly}(F) \\
&  \quad \quad \quad = 
\frac{3!}{r''\ly^3} \begin{pmatrix}
*& * & * & * \\
-1/\lambda^2 & -1/\lambda & -1 &  0\\
*& * & * & * \\
*& * & * & * 
\end{pmatrix}
\begin{pmatrix}
a_0  \\
\mu a_0 \\
a_2\\
a_3
\end{pmatrix}  \\
&  \quad \quad \quad = 
 \frac{3!}{r''\ly^3} \ \left(*, - \frac{a_0}{\lambda}\left( \mu+ \frac{1}{\lambda}\right) - a_2, *, * \right).
\end{align*}
Here $a_0 >0$, $ \left( \mu+ \frac{1}{ \lambda} \right)>0$, $ a_2 \ge 0$
and so $v_1^{D_{\SZ}+\frac{1}{\lambda}\lz, \lz} (\Xi(F))  < 0$.  This contradicts with \eqref{first-v1-bound}. \\

\noindent (ii) \
Let $E \in  \HN_{\lx, -D_{\SX}}^{\mu}([-\lambda,0])$  for some  $\lambda  \in \QQ_{> 0}$.

Let $\TPsi:= \Phi^{\SX \to \SY}_{\eE^\vee}$.
From the co-convergence of the ``Duality'' Spectral Sequence \ref{Spec-Seq-Dual} for $E$ we have
$$
\left( \Psi ^3(E) \right)^* \cong \TPsi^0(E^*).
$$
We have $E^{*} \in \HN_{\lx, D_{\SX}}^{\mu}([0,\lambda])$. 
So by  Proposition~\ref{prop:FMT-0-g-cohomo-vanishing}--(iii) and part (i),
 we have
 $\TPsi^0(E^*)\in \HN_{\ly, -D_{\SY}}^{\mu}((-\infty, -\frac{1}{ \lambda }])$.
Therefore, 
$\Psi^3(E) \in \HN_{\ly, D_{\SY}}^{\mu}([\frac{1}{ \lambda }, +\infty])$
as required.
\end{proof}

%
\subsection{Images of the first tilted hearts under the FM transforms}
Let us recall the first tilting associated to numerical parameters of Theorem \ref{prop:equivalence-hearts-abelian-threefolds} involving the  Fourier-Mukai  transform $\Psi : D^b(X) \to D^b(Y)$.
\begin{nota}
The subcategories 
\begin{align*}
& \fF^{\SX}_1= \HN^{\mu}_{\lx, -D_{\SX} + \frac{\lambda \lx}{2}}((-\infty, 0]) = \HN^{\mu}_{\lx, -D_{\SX} }((-\infty,   \frac{\lambda }{2}]), \\
& \tT^{\SX}_1= \HN^{\mu}_{\lx, -D_{\SX} + \frac{\lambda \lx}{2}}((0, +\infty]) =  \HN^{\mu}_{\lx, -D_{\SX} }(( \frac{\lambda }{2} , \infty])
\end{align*}
of $\Coh(X)$ forms a torsion pair, and the corresponding tilted category is
$$
\bB^{\SX} = \langle \fF^{\SX}_1[1] , \tT^{\SX}_1 \rangle. 
$$
Similarly, the subcategories 
\begin{align*}
& \fF^{\SY}_1= \HN^{\mu}_{\ly, D_{\SY} - \frac{\ly}{2\lambda}}((-\infty, 0]) = \HN^{\mu}_{\ly, D_{\SY}}((-\infty, - \frac{1}{2\lambda}]), \\
& \tT^{\SY}_1= \HN^{\mu}_{\ly, D_{\SY} - \frac{ \ly}{2 \lambda}}((0, +\infty]) = \HN^{\mu}_{\ly, D_{\SY}}((-\frac{1}{2\lambda}, +\infty])
\end{align*}
of $\Coh(Y)$ forms   a   torsion pair, and the corresponding tilted category is
$$
\bB^{\SY} = \langle \fF^{\SY}_1[1] , \tT^{\SY}_1 \rangle. 
$$
\end{nota}
\begin{thm}
\label{prop:image-B-under-FMT}
We have the following:
\begin{itemize}
\item[(i)] $\Psi \left(\bB^{\SX} \right)  \subset \langle \bB^{\SY} ,  \bB^{\SY}[-1],  \bB^{\SY}[-2]
  \rangle$, and
\item[(ii)] $\HPsi[1] \left(\bB^{\SY} \right)  \subset \langle \bB^{\SX} ,  \bB^{\SX}[-1],  \bB^{\SX}[-2]
  \rangle$.
\end{itemize}
\end{thm}
\begin{proof}
(i) \
 We can visualize $\bB^{\SX}$ and $\bB^{\SY}$ as follows:
$$
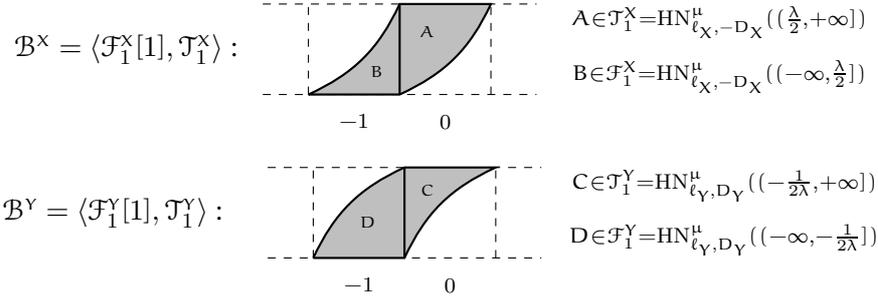
\begin{tikzpicture}[scale=1.2]
\draw[style=dashed] (0.5,0) grid (3.5,1);
\fill[lightgray] (1,0) -- (2,0)  to[out=25,in=-115] (3,1) -- (2,1) to[out=-115,in=25] (1,0);
\draw[style=thick] (1,0) -- (2,0) to[out=25,in=-115] (3,1)-- (2,1) to[out=-115,in=25] (1,0);
\draw[style=thick] (2,0) -- (2,1);
\draw (1.75,0.25) node {$\scriptscriptstyle B$};
\draw (2.3,0.7) node {$\scriptscriptstyle A$};
\draw (1.5,- 0.3) node {$\scriptstyle{-1}$};
\draw (2.5,-0.3) node {$\scriptstyle{0}$};
\draw (-1,0.5) node {$\bB^{\SX} = \langle  \fF^{\SX}_1[1] , \tT^{\SX}_1 \rangle : $};
\draw (5.5,0.8) node {$\scriptstyle A \in  \tT^{\SX}_1 = \HN^{\mu}_{\lx, -D_{\SX} }(( \frac{\lambda }{2} , +\infty])$ };
\draw (5.5,0.2) node {$\scriptstyle B \in \fF^{\SX}_1 = \HN^{\mu}_{\lx, -D_{\SX} }(( -\infty, \frac{\lambda }{2}])$ };
\end{tikzpicture}
$$
$$
\begin{tikzpicture}[scale=1.2]
\draw[style=dashed] (0.5,0) grid (3.5,1);
\fill[lightgray] (1,0) -- (2,0) to[out=65,in=-155] (3,1) -- (2,1) to[out=-155,in=65] (1,0);
\draw[style=thick] (1,0) -- (2,0) to[out=65,in=-155] (3,1)-- (2,1) to[out=-155,in=65] (1,0);
\draw[style=thick] (2,0) -- (2,1);
\draw (1.6,0.4) node {$\scriptscriptstyle D$};
\draw (2.25,0.75) node {$\scriptscriptstyle C$};
\draw (1.5,- 0.3) node {$\scriptstyle{-1}$};
\draw (2.5,-0.3) node {$\scriptstyle{0}$};
\draw (-1.2,0.5) node {$\bB^{\SY} = \langle  \fF^{\SY}_1[1] , \tT^{\SY}_1 \rangle  : $};
\draw (5.5,0.8) node {$\scriptstyle C \in \tT^{\SY}_1=  \HN^{\mu}_{\ly, D_{\SY}}((-\frac{1}{2\lambda}, +\infty])$ };
\draw (5.5,0.2) node {$\scriptstyle D \in \fF^{\SY}_1= \HN^{\mu}_{\ly, D_{\SY}}((-\infty, - \frac{1}{2\lambda}])$ };
\end{tikzpicture}
$$

If $E \in \fF^{\SX}_1 = \HN^{\mu}_{\lx, -D_{\SX} }(( -\infty, \frac{\lambda }{2}])$ then
  by    Propositions \ref{prop:FMT-0-g-cohomo-vanishing}--(iii) and \ref{prop:special-slope-bound-FMT-0-3}--(i),   $\Psi^0(E) \in  \fF^{\SY}_1$.
   Also by  Proposition  \ref{prop:slope-bound-torsion-sheaf},
$\Psi^3(E) \in   \HN^{\mu}_{\ly, D_{\SY}}((0, +\infty]) \subset \tT^{\SY}_1$.
Therefore, $\Psi(E)$ has $\bB^{\SY}$-cohomologies in 1,2,3
  positions. That is
$$
\Psi\left(\fF^{\SX}_1 [1]\right) \subset \langle \bB^{\SY}, \bB^{\SY}[-1], \bB^{\SY}[-2] \rangle.
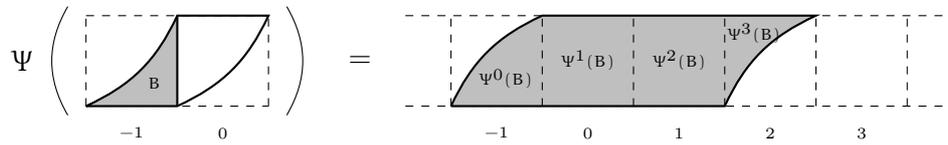
$$
$$
\begin{tikzpicture}[scale=1.2]
\draw[style=dashed] (0.5,-2) grid (6.5,-1);
\fill[lightgray] (1,-2) -- (4,-2) to[out=65,in=-155] (5,-1) -- (2,-1) to[out=-155,in=65] (1,-2);
\draw[style=thick] (1,-2) -- (4,-2) to[out=65,in=-155] (5,-1) -- (2,-1) to[out=-155,in=65] (1,-2);
\draw[style=dashed] (2,-2) -- (2,-1) ;
\draw[style=dashed] (3,-2) -- (3,-1) ;
\draw[style=dashed] (4,-2) -- (4,-1) ;
\draw (1.6,-1.7) node {$\scriptscriptstyle \Psi^0(B)  $};
\draw (2.5,-1.5) node {$\scriptscriptstyle \Psi^1(B) $};
\draw (3.5,-1.5) node {$\scriptscriptstyle \Psi^2(B)$};
\draw (4.32,-1.2) node {$\scriptscriptstyle \Psi^3(B)$};
\draw (1.5,- 2.3) node {$\scriptscriptstyle{-1}$};
\draw (2.5,-2.3) node {$\scriptscriptstyle{0}$};
\draw (3.5,- 2.3) node {$\scriptscriptstyle{1}$};
\draw (4.5,-2.3) node {$\scriptscriptstyle{2}$};
\draw (5.5,-2.3) node {$\scriptscriptstyle{3}$};
\draw[style=dashed] (-3 ,-2) grid (-1,-1);
\fill[lightgray] (-3,-2) -- (-2,-2) -- (-2,-1) to[out=-115,in=25] (-3,-2);
\draw[style=thick] (-3,-2) -- (-2, -2) to[out=25,in=-115] (-1, -1) -- (-2, -1) to[out=-115,in=25] (-3, -2);
\draw[style=thick] (-2,-2) -- (-2, -1);
\draw (-2.25,-1.75) node {$\scriptscriptstyle B$};
\draw (-2.5,- 2.3) node {$\scriptscriptstyle{-1}$};
\draw (-1.5,-2.3) node {$\scriptscriptstyle{0}$};
\draw (-3.2,-2.1) to [out=120,in=240] (-3.2,-0.9);
\draw (-0.8,-2.1) to [out=60,in=300] (-0.8,-0.9);
\draw (-3.7,-1.5) node {$\Psi$};
\draw (0,-1.5) node {$=$};
\end{tikzpicture}
$$
On the other hand, if $\tT^{\SX}_1 = \HN^{\mu}_{\lx, -D_{\SX} }(( \frac{\lambda }{2} , +\infty])$ 
then by  Proposition~\ref{prop:FMT-0-g-cohomo-vanishing}--(i),  $\Psi^3(E) =0$, and by  Proposition \ref{prop:slope-bound-FMT-T-2},
  $\Psi^2 (E) \in  \HN^{\mu}_{\ly, D_{\SY}}((0, +\infty]) \subset \tT^{\SY}_1$.
So $\Psi(E)$ has $\bB^{\SY}$-cohomologies in positions
0,1,2 only. That is
$$
\Psi\left(\tT^{\SX}_1 \right) \subset \langle \bB^{\SY}, \bB^{\SY}[-1], \bB^{\SY}[-2] \rangle.
$$
$$
\begin{tikzpicture}[scale=1.2]
\draw[style=dashed] (0.5,0) grid (6.5,1);
\fill[lightgray] (2,0) -- (4,0)to[out=65,in=-155] (5,1) -- (2,1) --(2,0);
\draw[style=thick] (2,0) -- (4,0)  to[out=65,in=-155]  (5,1) -- (2,1) -- (2,0);
\draw[style=dashed] (3,0) -- (3,1) ;
\draw[style=dashed] (4,0) -- (4,1) ;
\draw (2.5,0.5) node {$\scriptscriptstyle \Psi^0(A)$};
\draw (3.5,0.5) node {$\scriptscriptstyle \Psi^1(A)$};
\draw (4.32,0.8) node {$\scriptscriptstyle \Psi^2(A)$};
\draw (1.5,- 0.3) node {$\scriptscriptstyle{-1}$};
\draw (2.5,-0.3) node {$\scriptscriptstyle{0}$};
\draw (3.5,- 0.3) node {$\scriptscriptstyle{1}$};
\draw (4.5,-0.3) node {$\scriptscriptstyle{2}$};
\draw (5.5,-0.3) node {$\scriptscriptstyle{3}$};
\draw[style=dashed] (-3 ,0) grid (-1,1);
\fill[lightgray] (-2,0) to[out=25,in=-115] (-1,1) -- (-2,1) -- (-2,0);
\draw[style=thick] (-3,0) -- (-2, 0) to[out=25,in=-115]  (-1, 1) -- (-2, 1) to[out=-115,in=25]  (-3, 0);
\draw[style=thick] (-2,0) -- (-2, 1);
\draw (-1.7,0.7) node {$\scriptscriptstyle A$};
\draw (-2.5,- 0.3) node {$\scriptscriptstyle{-1}$};
\draw (-1.5,-0.3) node {$\scriptscriptstyle{0}$};
\draw (-3.2,-0.1) to [out=120,in=240] (-3.2,1.1);
\draw (-0.8,-0.1) to [out=60,in=300] (-0.8,1.1);
\draw (-3.7,0.5) node {$\Psi$};
\draw (0,0.5) node {$=$};
\end{tikzpicture}
$$
Hence, $\Psi\left(\bB^{\SX} \right) \subset \langle \bB^{\SY}, \bB^{\SY}[-1], \bB^{\SY}[-2] \rangle
$
as $\bB^{\SX} = \langle \fF^{\SX}_1 , \tT^{\SX}_1 \rangle$. \\

\noindent (ii) \ We can use  Propositions \ref{prop:FMT-0-g-cohomo-vanishing}--(iii), \ref{prop:slope-bound-FMT-0-g}--(i), \ref{prop:FMT-0-g-cohomo-vanishing}--(i) and \ref{prop:special-slope-bound-FMT-0-3}--(ii) in a
  similar way to the above proof.
\end{proof}

\section{Some Stable Reflexive Sheaves on Abelian Threefolds}
\label{sec:special-stable-reflexive-sheaves-abelain-threefolds}
In this section we shall consider slope semistable sheaves with vanishing first and second parts of
the twisted Chern characters.
Such sheaves arise as the
$\Coh(X)$-cohomology of some of the tilt-stable objects on $X$; see Proposition \ref{prop:trivial-discriminant-tilt-stable-objects}.



\begin{nota}
Let $X$ be an abelian threefold. 
Let $\pP$ be the Poincar\'e bundle on $X \times \HX$. 
We simply write 
\begin{align*}
&\Phi = \Phi^{\SX \to \SHX}_{\pP}: D^b(X) \to D^b(\HX), \\
&  \HPhi =  \Phi^{\SHX \to \SX}_{\pP^\vee}: D^b(\HX) \to D^b(X).
\end{align*}
Let $\lx \in \NS_{\QQ}(X)$ and $ \lhx \in \NS_{\QQ}(\HX)$ be the ample classes as in Lemma \ref{classicalcohomoFMT} (or equivalently  Theorem \ref{prop:general-cohomo-FMT}).
\end{nota}

First we prove the following:
\begin{lem}
\label{prop:trivial-discriminant-reflexive-sheaves}
Let $E$ be a slope semistable sheaf on $X$ with respect to $\lx$ such that $\ch_k(E)=0$ for $k=1,2$. Then
$E^{**}$ is a homogeneous bundle, that is $E^{**}$ is filtered with quotients  from $\Pic^0(X)$. 
\end{lem}
\begin{proof}
Any torsion free sheaf $E$ fits into the
short exact sequence $0 \to E \to E^{**} \to Q \to 0$ in $\Coh(X)$ for some $Q \in \Coh_{\le 1}(X)$.
If  $\ch_k(E)=0$ for $k=1,2$ then $\lx \ch_2(E^{**}) \ge 0$ where the equality holds when $Q \in \Coh_0(X)$.
If $E$ is slope semistable then $E^{**}$ is also slope semistable, and so by the usual Bogomolov-Gieseker inequality
$\lx \ch_2(E^{**}) = 0$. Hence,  $\lx \ch_2(Q) =0$, and so $\ch_2(Q) =0$; that is $\ch_2(E^{**}) =0$.

Assume the opposite for a contradiction.
Then there exists a semistable reflexive sheaf $E$ with $\ch_k(E) = 0$ for $k=1,2$,  and $H^k(X, E \otimes \pP_{X \times \{\widehat{x}\}}) = 0$ for $k=0,3$ and any $\hx \in \HX$.
So we have $\Phi^0(E) = \Phi^3(E) =0$.
By a result of Simpson (see Lemma~\ref{prop:Simpson-result-trivial-disciminant}),  we have $\ch_3(E) = 0$.
Therefore, $\ch(E) = (r, 0 , 0, 0)$ for some positive integer $r$.

By Proposition~\ref{prop:slope-bound-FMT-F-1},  $\Phi^1(E) \in \HN^{\mu}_{\lhx, 0}((-\infty, 0])$, 
and  $\Phi^2(E) \in \HN^{\mu}_{\lhx, 0}([0,+\infty])$.
So we have $\lhx^2 \ch_1(\Phi^1(E)) \le 0$ 
and $ \lhx^2 \ch_1(\Phi^2(E)) \ge 0$. 
Therefore, $\lhx^2 \ch_1(\Phi(E)) \ge 0$. Moreover,   since $\ch_2(E) =0$, from Theorem \ref{prop:antidiagonal-rep-cohom-FMT},
we obtain $ \lhx^2 \ch_1(\Phi(E))=  0$. Hence, $\lhx^2 \ch_1(\Phi^1(E)) =\lhx^2 \ch_1(\Phi^2(E)) =0$.
So we have
\begin{align*}
\ch(\Phi^1(E)) = (a, D , -C, d),  \ \ \ch(\Phi^2(E)) = (a, D, -C, -r+d),
\end{align*}
for some $a>0$, $D \in \NS(\HX)$, $C \in H^4_{\alg}(\HX,\QQ)$ such that $\lhx^2 D=0$ and $\lhx C \ge  0$. 
Moreover, we have $\Phi^1(E) \in \HN^{\mu}_{\lhx, 0}(0)$.

If $\HPhi^3 \Phi^1(E)  \ne 0$ then $\Phi^1(E)$ fits into a
short exact sequence
$0 \to K_1 \to \Phi^1(E) \to \pP_{\{x_1\}\times \HX}  \iI_{C_1} \to 0$
in $\Coh(\HX)$ for some $x_1 \in X$ and $C_1 \in H_2(\HX, \ZZ)$.
Then $K_1  \in \HN^{\mu}_{\lhx, 0}(0)$ and we have the following exact sequence
$$
\cdots \to \HPhi^3(K_1)  \to \HPhi^3 \Phi^1(E)  \to \oO_{x_1} \to 0
$$
in $\Coh(X)$.
If $\HPhi^3(K_1) \ne 0$ then $K_1$  fits into a short exact sequence
$0 \to K_2 \to K_1 \to \pP_{\{x_2\}\times \HX}  \iI_{C_2} \to 0$
 in $\Coh(\HX)$.
Then $K_2  \in \HN^{\mu}_{\lhx, 0}(0)$ and we have the following exact sequence
$$
\cdots \to\HPhi^3(K_2) \to \HPhi^3(K_1) \to \oO_{x_2} \to 0
$$
in $\Coh(X)$.
We can continue this process for only a finite number of steps since $\rk (\Phi^1(E)) < +\infty$, and hence $\HPhi^3 \Phi^1(E) $
 is filtered by skyscraper sheaves.
Moreover, from the convergence of  Mukai Spectral Sequence~\ref{Spec-Seq-Mukai} for $E$, we have the short exact sequence
$$
0 \to \HPhi^0 \Phi^2(E)  \to \HPhi^2 \Phi^1(E)   \to  Q \to 0
$$
in  $\Coh(X)$,
where $Q$ is a subsheaf of $E$ and so $Q  \in \HN^{\mu}_{\lx, 0}((-\infty, 0])$. 
By Proposition~\ref{prop:slope-bound-FMT-0-g}--(i), $\HPhi^0 \Phi^2(E)    \in \HN^{\mu}_{\lx, 0}((-\infty, 0])$. 
This implies $\HPhi^2 \Phi^1(E)   \in \HN^{\mu}_{\lx, 0}((-\infty, 0])$.
Therefore, we have $\lx^2 \ch_1(\HPhi(\Phi^1(E))) \le 0$, and so $\lhx C  \le 0$. 
Hence, $\lhx C  = 0$.
By Proposition~\ref{prop:FMT-F-1-reflexivity}, $\Phi^1(E)$ is a reflexive sheaf and since $\Phi^1(E)  \in \HN^{\mu}_{\lhx, 0}(0)$ it is slope semistable.
So by Lemma~\ref{prop:Simpson-result-trivial-disciminant}, we have $D=0$, $C=0$ and $d= \ch_3(\Phi^1(E)) =0$.
Therefore, $\ch(\HPhi(\Phi^1(E))) = (0,0,0,-a)$. Since $\HPhi^3 \Phi^1(E) \in \Coh_0(X)$, we have $\ch_k(\HPhi^2 \Phi^1(E)) = 0$ for $k=0,1,2$.
So $\HPhi^2 \Phi^1(E) \in \HN^{\mu}_{\lx, 0}((0, +\infty])$. Therefore, $\HPhi^2 \Phi^1(E) = 0$ and we have the
short exact sequence
$$
0 \to  E \to \HPhi^1 \Phi^2 (E) \to \HPhi^3 \Phi^1(E) \to 0
$$
in $\Coh(X)$.
Since $\HPhi^3 \Phi^1(E) \in \Coh_0(X)$ and $E$ is locally free, $\Ext^1_{\SX}(\HPhi^3 \Phi^1(E), E) = 0$.
Therefore, $\HPhi^1 \Phi^2(E) \cong  E \oplus \HPhi^3 \Phi^1(E)$.
Since $\HPhi^1 \Phi^2(E) \in V^{\Phi}_{\Coh(Y)}(2)$, 
we have $\HPhi^3 \Phi^1(E) = 0$ and so $E \in V^{\Phi}_{\Coh(Y)}(2)$.
Therefore, $\ch(\Phi^2(E)) = (0,0,0,-r)$.
But it is not possible to have $-r > 0$ and
this is the required contradiction to complete the proof.
\end{proof}
Then we show the following.
\begin{thm}
\label{prop4.15}
Let $E$ be a
slope stable torsion free sheaf of rank $r$ with $\ch^{B}_k(E) = 0$, 
$k=1,2$ for some $B \in \NS_{\QQ}(X)$. Then $E^{**}$ is a slope
stable semihomogeneous bundle with $\ch(E^{**}) = r e^B$.
\end{thm}
\begin{proof}
The slope stable torsion free sheaf $E$ fits into the short exact sequence
$
0 \to E \to E^{**} \to T \to 0
$
for some $T \in \Coh_{\le 1}(X)$. Now $E^{**}$ is also slope stable
and so by the usual Bogomolov-Gieseker inequality $\ch_k^{B}(E^{**}) = 0$ for
$k=1,2$.
By Lemma \ref{prop:trivial-discriminant-reflexive-sheaves}, $\calEnd(E^{**})$ is a homogeneous
bundle. Therefore, by Lemma \ref{prop:Mukai-semihomognoeus-properties}, $E^{**}$ is a slope stable semihomogeneous bundle, and so from Lemma \ref{prop:semihomo-numerical}--(i) $\ch(E^{**})= r e^B$.
\end{proof}

\section{Equivalences of Stability Condition Hearts on  Abelian Threefolds}
\label{sec:FMT-tilt-stability}
Let us recall the second tilting associated to numerical parameters of Theorem \ref{prop:equivalence-hearts-abelian-threefolds} involving the  Fourier-Mukai  transform $\Psi : D^b(X) \to D^b(Y)$.
\begin{nota}
The subcategories 
\begin{align*}
& \fF^{\SX}_2= \HN^{\nu}_{\lx, -D_{\SX} + \frac{\lambda \lx}{2}, \frac{\lambda }{2}}((-\infty, 0]) , \\
& \tT^{\SX}_2= \HN^{\nu}_{\lx, -D_{\SX} + \frac{\lambda \lx}{2} , \frac{\lambda }{2}}((0, +\infty])
\end{align*}
of $\bB^{\SX}$ forms a torsion pair, and the corresponding tilt is
$$
\aA^{\SX} = \langle \fF^{\SX}_2[1] , \tT^{\SX}_2 \rangle. 
$$
Similarly, 
\begin{align*}
& \fF^{\SY}_2= \HN^{\nu}_{\ly, D_{\SY} - \frac{\ly}{2\lambda}, \frac{1}{2\lambda}}((-\infty, 0]), \\
& \tT^{\SY}_2= \HN^{\nu}_{\ly, D_{\SY} - \frac{ \ly}{2 \lambda}, \frac{1}{2\lambda}}((0, +\infty]) 
\end{align*}
defines a   torsion pair of    $\bB^{\SY}$, and the corresponding tilt  is
$$
\aA^{\SY} = \langle \fF^{\SY}_2[1] , \tT^{\SY}_2 \rangle. 
$$
Let us write the complexified ample classes by 
\begin{align*}
 &\Omega = \left( -D_{\SX} +  \lambda \lx/2 \right)  + i \sqrt{3} \lambda \lx/2, \\        
 &  \Omega' = \left( D_{\SY} - \ly/(2 \lambda) \right) + i  \sqrt{3}  \ly/(2 \lambda).
\end{align*}

 We write the corresponding central charge functions simply by 
\begin{align*}
& Z_{\Omega} = Z_{\lx, -D_{\SX} + \frac{\lambda \lx}{2}, \frac{\lambda }{2}}, \\
& Z_{\Omega'} = Z_{\ly, D_{\SY} - \frac{ \ly}{2 \lambda}, \frac{1}{2\lambda}}.
\end{align*}
\end{nota}
It will be convenient to abbreviate the  Fourier-Mukai  transforms $\Psi$ and $\HPsi[1]$ by
$\Ga$ and $\HGa$ respectively. That is,
\begin{align*}
& \Gamma := \Psi = \Phi^{\SX \to \SY}_{\eE} : D^b(X) \to D^b(Y), \\
& \HGa := \Psi [1]= \Phi^{\SY \to \SX}_{\eE^\vee} [1] : D^b(Y) \to D^b(X).
\end{align*}
Then by Theorem~\ref{prop:image-B-under-FMT}, the images of an object
from $\bB^{\SX}$ (and $\bB^{\SY}$) under $\Ga$ (and $\HGa$) are  complexes whose
cohomologies with respect to $\bB^{\SY}$ (and $\bB^{\SX}$) can only be non-zero in $0,1,2$ positions.
\begin{nota}
In the rest of the paper we write
\begin{align*}
& \Ga^i_{\bB} (-) = H^{i}_{\bB^{\SY}}(\Ga(-)),\\
& \HGa^i_{\bB} (-) = H^{i}_{\bB^{\SX}}(\HGa(-)).
\end{align*}

\end{nota}
We  have $\Ga \circ \HGa \cong [-2]$ and $\HGa \circ \Ga \cong [-2]$.
This gives us the following
convergence of spectral sequences.
\begin{specseq}
\label{Spec-Seq-B}
\begin{enumerate}[label=(\arabic*)]
\item[]
\item $E_2^{p,q} = \HGa^p_{\bB} \Ga^q_{\bB} (E) \Longrightarrow H^{p+q-2}_{\bB^{\SX}} (E)$, and
\item$ E_2^{p,q} = \Ga^p_{\bB} \HGa^q_{\bB} (E) \Longrightarrow H^{p+q-2}_{\bB^{\SY}} (E)$.
\end{enumerate}
\end{specseq}
Such convergence of the spectral sequences for $E \in \bB^{\SX}$ and $E \in \bB^{\SY}$
behave in the same way as  the convergence of the Mukai Spectral Sequence \ref{Spec-Seq-Mukai} for coherent sheaves on an abelian
surface.
The following diagram describes the convergence of Spectral Sequence~\ref{Spec-Seq-B}--(1) for $E \in \bB^{\SX}$.
\begin{center}
\begin{tikzpicture}[scale=1.80]
\draw[gray,very thin] (0,0) grid (3,3);
\draw[->,thick] (2.75,0.25) -- (3.5, 0.25) node[above] {$p$};
\draw[->,thick] (0.25,2.75) -- (0.25,3.5) node[left] {$q$};
\draw (2.5,0.5) node(a) {$ \HGa^2_{\bB} \Ga^0_{\bB}(E)$};
\draw (0.5,1.5) node(b) {$\HGa^0_{\bB} \Ga^1_{\bB}(E)$};
\draw[>->,thick] (b) -- node[above] {$ $} (a);
\draw (2.5,1.5) node(d) {$\HGa^2_{\bB} \Ga^1_{\bB}(E)$};
\draw (0.5,2.5) node(c) {$\HGa^0_{\bB} \Ga^2_{\bB}(E)$};
\draw[->>,thick] (c) -- node[above] {$ $} (d);
\draw (1.5,1.5) node  {$\HGa^1_{\bB} \Ga^1_{\bB}(E)$};
\end{tikzpicture}
\end{center}

\begin{prop}
\label{prop:FMT-B-cat-bounds-Imginary-Z}
We have the following:
\begin{enumerate}[label=(\arabic*)]
\item For $ E \in \tT_{2}^{\SY}$,
(i) $\hH^0(\HGa^2_{\bB}(E)) = 0$, and
(ii)  if $\HGa^2_{\bB}(E) \ne 0$  then $\Imm Z_{\Omega}(\HGa^2_{\bB}(E)) >  0$.

\item For $E \in  \fF_{2}^{\SY}$,
(i) $ \hH^{-1}(\HGa^0_{\bB}(E)) = 0$, and
(ii)  if $\HGa^0_{\bB}(E) \ne 0$  then $\Imm  Z_{\Omega}(\HGa^0_{\bB}(E)) <  0$.

\item For $E \in  \tT_{2}^{\SX}$,
(i) $ \hH^0(\Ga^2_{\bB}(E)) = 0$, and
(ii)  if $ \Ga^2_{\bB}(E) \ne 0$  then $\Imm Z_{\Omega'}(\Ga^2_{\bB}(E)) >  0$.
\item For $E \in  \fF_{2}^{\SX}$,
(i) $\hH^{-1} (\Ga^0_{\bB}(E)) = 0$, and
(ii)  if $\Ga^0_{\bB}(E) \ne 0$  then $\Imm Z_{\Omega'}(\Ga^0_{\bB}(E)) < 0$.
\end{enumerate}
\end{prop}
\begin{proof}
(1) \ Let $E \in \tT_{2}^{\SY}$. \\
\noindent (i) \ For any $x \in X$,
\begin{align*}
\Hom_{\SX} ( \HGa^2_{\bB}(E) , \oO_x  ) & \cong  \Hom_{\SX} ( \HGa^2_{\bB}(E) , \HGa^2_{\bB}(\eE_{\{x\} \times Y} ) ) \\
                               & \cong  \Hom_{\SX} ( \HGa(E) , \HGa  (\eE_{\{x\} \times Y}   ) ) \\
                               & \cong  \Hom_{\SX} ( E, \eE_{\{x\} \times Y}  ) = 0,
\end{align*}
since $E \in \tT^{\SY}$ and $ \eE_{\{x\} \times Y}  \in
\fF^{\SY}$. Therefore, $\hH^0(\HGa^2_{\bB}(E)) = 0$
as required. \\
\noindent (ii) \ From (1)(i), we have $\HGa^2_{\bB}(E) \cong A[1]$ for some
 $0 \neq A \in \HN^{\mu}_{\lx, -D_{\SX}}((-\infty, \frac{\lambda}{2}])$.
Consider the convergence of the spectral sequence:
$$
E^{p,q}_2=\HGa^{p} (\hH^{q} (E)) \Longrightarrow \HGa^{p+q} (E)
$$
for $E$.
By Proposition~\ref{prop:slope-bounds}--(2), we have
$\hH^0(E) \in \HN^{\mu}_{\ly, D_{\SY}}((0, +\infty])$ and so by  Propositions \ref{prop:slope-bound-FMT-T-2} and \ref{prop:slope-bound-torsion-sheaf},
$$
\HPsi^2 (\hH^0(E)), \HPsi^3 (\hH^{-1}(E)) \in \HN^{\mu}_{\lx, -D_{\SX}}((0,+\infty]).
$$
Therefore, from the convergence of the above spectral sequence for $E$, we have
$$
A \in  \HN^{\mu}_{\lx, -D_{\SX}}((-\infty, \frac{\lambda}{2}]) \cap \HN^{\mu}_{\lx, -D_{\SX}}((0,+\infty]) = \HN^{\mu}_{\lx, -D_{\SX}}((0, \frac{\lambda}{2}]).
$$
Let $v^{-D_{\SX}, \lx} (A)= (a_0, a_1, a_2, a_3)$. From the usual Bogomolov-Gieseker inequalities for all the Harder-Narasimhan semistable factors of $A$ we have
$\frac{\lambda}{2}a_1-a_2 \ge 0$ and so by Proposition \ref{prop:imgainary-part-central-charge}--(1),
$$
\Imm Z_{\Omega}(\HGa^2_{\bB}(E)) =  \Imm Z_{\Omega}(A[1])  = \frac{ \sqrt{3} \lambda}{4}(\lambda a_1-a_2) > 0
$$
as required. \\

\noindent (2) \ Let $E \in \fF^{\SY}_2$. \\
\noindent (i) \
For any $x \in X$ we have
\begin{align*}
\Hom_{\SX} ( \HGa^0_{\bB}(E) , \oO_x[1] ) 
                                    & \cong  \Hom_{\SY} (\Ga \HGa^0_{\bB}(E) , \Ga (\oO_x[1]))\\
                                    & \cong  \Hom_{\SY} ( \Ga^2_{\bB} \HGa^0_{\bB}(E)[-2] , \eE_{\{x\} \times Y}[1] ) \\
                                    & \cong  \Hom_{\SY} ( \Ga^2_{\bB} \HGa^0_{\bB}(E) , \eE_{\{x\} \times Y}[3]  ) \\
                                    & \cong  \Hom_{\SY} ( \eE_{\{x\} \times Y} ,  \Ga^2_{\bB} \HGa^0_{\bB}(E) )^\vee .
\end{align*}
From the convergence of the Spectral Sequence~\ref{Spec-Seq-B} for $E$, we have the short exact sequence
$$
0 \to \Ga^0_{\bB} \HGa^1_{\bB}(E) \to \Ga^2_{\bB} \HGa^0_{\bB}(E) \to F \to 0
$$
in $\bB^{\SY}$, where $F$ is a subobject of $E$ and so $F \in \fF^{\SY}$.
Moreover, by the Harder-Narasimhan filtration, $F$ fits into the following short exact sequence  in $\bB^{\SY}$:
$$
0 \to F_0 \to F \to F_1 \to 0,
$$
where $F_0 \in \HN^{\nu}_{\ly, D_{\SY}- \frac{1}{2 \lambda}\ly, \frac{1}{2 \lambda}}(0)$
 and $F_1 \in \HN^{\nu}_{\ly, D_{\SY}- \frac{1}{2 \lambda} \ly, \frac{1}{2 \lambda}}((-\infty, 0))$.
Since  $\eE_{\{x\} \times Y} \in  \HN^{\nu}_{\ly, D_{\SY}- \frac{1}{2 \lambda}, \frac{1}{2 \lambda}}(0)$,
$$
\Hom_{\SY} (\eE_{\{x\} \times Y} , F_1) =0.
$$
Moreover, $F_0$ fits into a filtration with quotients of $\nu_{\ly, D_{\SY}- \frac{1}{2 \lambda}\ly, \frac{1}{2 \lambda}}$-stable objects $F_{0,i}$ with $\nu_{\ly, D_{\SY}- \frac{1}{2 \lambda}\ly, \frac{1}{2 \lambda}}(F_{0,i}) = 0$.
By Proposition~\ref{prop:reduction-BG-ineq-class}, 
each $F_{0,i}$ fits into a non-splitting short exact sequence
$$
0 \to F_{0,i} \to M_i \to T_i \to 0
$$
in $\bB^{\SY}$ for some $T_i \in \Coh_0(Y)$ 
such that $M_i[1] \in \aA^{\SY}$ is a minimal object.
Moreover, $\eE_{\{x\} \times Y}[1] \in \aA^{\SY}$ is a  minimal object.
So finitely many $x \in X$ we can have $\eE_{\{x\} \times Y} \cong M_i$ 
for some $i$. So for generic $x \in X$,
$\Hom_{\SY}(\eE_{\{x\} \times Y}, M_i) = 0$ and 
so $\Hom_{\SY}(\eE_{\{x\} \times Y} , F_{0,i}) = 0$ 
which implies $\Hom_{\SY}(\eE_{\{x\} \times Y} , F_0) = 0$.
Therefore, for generic $x \in X$, $\Hom_{\SY} ( \eE_{\{x\} \times Y} ,  F ) = 0$.

On the other hand,
\begin{align*}
\Hom_{\SY} ( \eE_{\{x\} \times Y} , \Ga^0_{\bB} \HGa^1_{\bB} (E))
                       & \cong  \Hom_{\SY} ( \Ga^0_{\bB} (\oO_x)  , \Ga^0_{\bB} \HGa^1_{\bB} (E) )\\
                                               & \cong  \Hom_{\SY} (  \Ga  (\oO_x)  , \Ga \HGa^1_{\bB} (E) ) \\
                                               & \cong  \Hom_{\SX} (\oO_x , \HGa^1_{\bB} (E) ).
\end{align*}
Here $\HGa^1_{\bB} (E)$ fits into the short exact sequence
$$
0 \to \hH^{-1}(\HGa^1_{\bB}(E))[1] \to \HGa^1_{\bB}(E) \to \hH^0(\HGa^1_{\bB} (E)) \to 0
$$
in $\bB^{\SX}$, where $\hH^{-1}(\HGa^1_{\bB} (E))$ is torsion free and
$\hH^0(\HGa^1_{\bB} (E))$ can have torsion supported
on a 0-subscheme of finite length.
Hence, for generic $x \in X$,  $ \Hom_{\SX} (\oO_x , \HGa^1_{\bB} (E)) = 0$.
Therefore, for generic $x \in X$,  we have
$\Hom_{\SY} (\eE_{\{x\} \times Y} , \Ga^0_{\bB} \HGa^1_{\bB} (E))
=\Hom_{\SY}( \eE_{\{x\} \times Y} ,  F ) =0$.
So $\Hom_{\SY} (\eE_{\{x\} \times Y},  \Ga^2_{\bB} \HGa^0_{\bB}(E)) =0$. 
Hence, for generic $x \in X$
$$
\Hom_{\SX} (\HGa^0_{\bB}(E) , \oO_x[1]) = 0.
$$
But $\hH^{-1}(\HGa^0_{\bB}(E))$ is torsion free and so $\hH^{-1}(\HGa^0_{\bB}(E)) = 0$ as required. \\
\noindent (ii) \ From (2)(i) we have $\HGa^0_{\bB}(E) \cong A$ for some non-trivial coherent sheaf $A \in \HN^{\mu}_{\lx, -D_{\SX}}((\frac{\lambda}{2} , + \infty])$. For any $x \in X$ we have
\begin{align*}
\Ext^1_{\SX}(\oO_x, A) & \cong  \Ext^1_{\SX} (\oO_x, \HGa^0_{\bB}(E)) 
\cong \Hom_{\SY}(\Ga(\oO_x), \Ga \HGa^0_{\bB}(E) [1]) \\
                 & \cong  \Hom_{\SY}(\eE_{\{x\} \times Y},  \Ga^2_{\bB} \HGa^0_{\bB}(E) [-1]) = 0.
\end{align*}
So $A \in \Coh_{\ge 2}(X)$, and if $v^{-D_{\SX}, \lx} (A) = (a_0, a_1, a_2, a_3)$ then we have $a_1 > 0$.

Apply the  Fourier-Mukai  transform $\Ga$ to $\HGa^0_{\bB}(E)$.
 Since $\HGa^0_{\bB}(E) \in V^{\Ga}_{\bB^{\SY}}(2)$,
   $\Ga^2_{\bB} \HGa^0_{\bB}(E) \in \bB^{\SY}$ has $\Coh(Y)$-cohomologies:
\begin{itemize}
\item $\Psi^1 (A)$  in position $-1$, and
\item $\Psi^2 (A)$  in position $0$.
\end{itemize}
So we have $A \in V^{\Psi}_{\Coh(Y)}(1,2)$,
$\Psi^1 (A) \in \HN^{\mu}_{\ly, D_{\SY} }((-\infty, - \frac{1}{2 \lambda}])$, and by  
of Proposition~\ref{prop:slope-bound-FMT-T-2}, $\Psi^2 (A) \in \HN^{\mu}_{\ly, D_{\SY} }((0, +\infty])$.
Therefore,  $v_1^{D_{\SY}, \ly}  (\Psi^1 (A)) \le 0$, 
$v_1^{D_{\SY}, \ly}  (\Psi^2 (A)) \ge 0$,  and so $v_1^{D_{\SY}, \ly}  (\Psi  (A)) \ge 0$.
Hence, by Theorem \ref{prop:antidiagonal-rep-cohom-FMT},
\begin{align*}
 a_2 & = v_2^{-D_{\SX}, \lx} (A)  \le 0.
\end{align*}
So
$$
\Imm Z_{\Omega}(\HGa^0_{\bB}(E)) =  \Imm Z_{\Omega}(A)  = \frac{ \sqrt{3}\lambda}{4}(a_2 - \lambda a_1) < 0
$$
as required. \\

\noindent (3) \  Let $E \in \tT^{\SX}_2$. \\
\noindent (i) \ Similar to the proof of (1)(i). \\
\noindent (ii) \
From (3)(i), we have $\Ga^2_{\bB}(E) \cong A[1]$ for some coherent sheaf
 $0 \ne A \in \HN^{\mu}_{\ly, D_{\SY}}((-\infty, -\frac{1}{2\lambda}])$.
Let $v^{D_{\SY}, \ly}(A) = (a_0, a_1, a_2, a_3)$. So $a_1 < 0$.

Apply the  Fourier-Mukai  transform $\HGa$ to $\Ga^2_{\bB}(E)$. 
Since $\Ga^2_{\bB}(E) \in V^{\HGa}_{\bB^{\SX}}(0)$,  
$\HGa^0_{\bB} \Ga^2_{\bB}(E) \in \bB^{\SX}$ has $\Coh(X)$-cohomologies:
\begin{itemize}
\item $\HPsi^1 (A)$  in position $-1$, and
\item $\HPsi^2(A)$  in position $0$.
\end{itemize}
So we have $A \in V^{\HPsi}_{\Coh(X)}(1,2)$,
$\HPsi^2 (A) \in \HN^{\mu}_{\lx, -D_{\SX}}([\frac{\lambda}{2}, +\infty])$, and
by   Proposition~\ref{prop:slope-bound-FMT-F-1}--(i), $\HPsi^1(A) \in \HN^{\mu}_{\lx, -D_{\SX}}((-\infty, 0])$.
Therefore,  $v_1^{-D_{\SX}, \lx}(\HPsi^1(A))\le 0$
and $v_1^{-D_{\SX}, \lx}(\HPsi^2(A)) \ge 0$.
So $v_1^{-D_{\SX}, \lx}(\HPsi(A)) \ge 0$, and hence, from Theorem \ref{prop:antidiagonal-rep-cohom-FMT}, $a_2 \le 0$. Therefore,
\begin{align*}
\Imm  Z_{\Omega'}(\Ga^2_{\bB}(E)) =  \Imm Z_{\Omega'}(A[1])
= \frac{\sqrt{3}}{4 \lambda}\left(-a_2 - \frac{1}{\lambda}a_1\right)  > 0
\end{align*}
as required. \\

\noindent (4) \  Let $E \in \fF^{\SX}_2 $. \\
\noindent (i) \ Similar to the proof of (2)(i). \\
\noindent (ii) \  From (4)(i) we have $\Ga^0_{\bB}(E) \cong A$ for some non-trivial
coherent sheaf $A \in  \HN^{\mu}_{\ly, D_{\SY}}((-\frac{1}{2 \lambda},+\infty])$.

Consider the convergence of the spectral sequence for $E$:
$$
E^{p,q}_2=\Ga^{p} \hH^{q}(E) \Longrightarrow \Ga^{p+q}(E).
$$

By Proposition~\ref{prop:slope-bounds}--(i), we have
$\hH^{-1}(E) \in \HN^{\mu}_{\lx, -D_{\SX}}((-\infty, 0])$, and so by  
  Propositions \ref{prop:slope-bound-FMT-F-1}-(i) and \ref{prop:slope-bound-FMT-0-g}--(i),
$$
\Psi^1(\hH^{-1}(E)) \in \HN^{\mu}_{\ly, D_{\SY}}((-\infty,0]), \text{ and } \Psi^0 (\hH^{0}(E)) \in \HN^{\mu}_{\ly, D_{\SY}}((-\infty,0]).
$$
Therefore, from the convergence of the above spectral sequence for $E$, we have
$$
A \in  \HN^{\mu}_{\ly, D_{\SY}}((-\frac{1}{2 \lambda },+\infty]) \cap \HN^{\mu}_{\ly, D_{\SY}}((-\infty,0])  
=\HN^{\mu}_{\ly, D_{\SY}}((-\frac{1}{2 \lambda},0]).
$$
Also by  Propositions \ref{prop:FMT-F-1-reflexivity} and \ref{prop:FMT-0-cohomology-reflexive}--(ii),
$\Psi^1 (\hH^{-1}(E))$ and  $\Psi^0 (\hH^{0}(E))$ are reflexive
sheaves, and so $A$ is reflexive.
Let $v^{D_{\SY}, \ly} (A)= (a_0, a_1, a_2, a_3)$. 
By the usual Bogomolov-Gieseker
inequalities for all the Harder-Narasimhan semistable factors of $A$,
we obtain $a_2 + \frac{1}{2\lambda } a_1 \le 0$.
So we have
$$
\Imm Z_{\Omega'}(\Ga^0_{\bB}(E))
 =  \Imm Z_{\Omega'}(\Ga^0_{\bB}(A))  = \frac{ \sqrt{3}}{4 \lambda }\left(a_2
+ \frac{1}{\lambda} a_1 \right) \le 0.
$$
Equality holds when $A \in \HN^{\mu}_{\ly, D_{\SY}}(0)$ 
with $v^{D_{\SY}, \ly}(A) = (a_0, 0, 0, *)$. By considering a Jordan{-}H\"{o}lder filtration
for $A$ together with Theorem~\ref{prop4.15}, $A$ is filtered
with quotients of sheaves $K_i$ each of them fits into the short exact sequence
$$
0 \to K_i \to \eE_{\{x_i\} \times Y} \to \oO_{Z_i} \to 0
$$
in $\Coh(Y)$ for some 0-subschemes $Z_i \subset Y$.
Here $\Ga^0_{\bB}(E) \cong A \in V^{\HGa}_{\bB^{\SX}}(2)$
implies $A \in V^{\Psi}_{\Coh(X)}(2,3)$.  
An easy induction on the number
of $K_i$ in $A$  shows that
$A \in V^{\HPsi}_{\Coh(X)}(1,3)$ and so $A \in
V^{\HPsi}_{\Coh(X)}(3)$. 
Therefore,  $Z_i =\emptyset$ for all $i$ and so
$\HGa^2_{\bB} \Ga^0_{\bB}(E) \in \Coh_0(X)$. 
Now consider the
convergence of the Spectral Sequence~\ref{Spec-Seq-B} for $E$. We have
the short exact sequence
$$
0 \to \HGa^0_{\bB}\Ga^1_{\bB}(E)  \to \HGa^2_{\bB}\Ga^0_{\bB}(E)  \to G \to 0
$$
in $\bB^{\SX}$, where $G$ is a subobject of $E$ and so 
$G \in \fF^{\SX}_2$. Now
$\HGa^2_{\bB} \Ga^0_{\bB}(E) \in \Coh_0(X) \subset \tT^{\SX}_2$ 
implies $G =0$ and so
$\HGa^0_{\bB}\Ga^1_{\bB}(E)  \cong \HGa^2_{\bB}\Ga^0_{\bB}(E)$.
Then we have $\Ga^0_{\bB}(E) \cong \Ga^0_{\bB} \HGa^0_{\bB}\Ga^1_{\bB}(E) =0$.
This is not possible as  $\Ga^0_{\bB}(E) \ne 0$.
Therefore, we have the strict inequality $\Imm Z_{\Omega'}(\Ga^0_{\bB}(E)) <0$ as required.
\end{proof}

\begin{lem}
\label{prop:B-cohomo-vanishing-FMT-0-2}
We have the following:
\begin{enumerate}[label=(\arabic*)]
\item if $E \in \tT^{\SY}_2$ then $\HGa^2_{\bB}(E) =0$,
\item if $E \in \fF^{\SY}_2 $ then $\HGa^0_{\bB}(E) =0$,
\item if $E \in \tT^{\SX}_2$ then $\Ga^2_{\bB}(E) =0$, and
\item if $E \in \fF^{\SX}_2$ then $\Ga^0_{\bB}(E) =0$.
\end{enumerate}
\end{lem}
\begin{proof}
First let us prove (1). Let $E \in \tT^{\SY}_2$.
From the convergence of the Spectral Sequence~\ref{Spec-Seq-B} for $E$, we have the short exact sequence
$$
0 \to Q \to \Ga^0_{\bB} \HGa^2_{\bB}(E) \to \Ga^2_{\bB} \HGa^1_{\bB}(E) \to 0
$$
in $\bB^{\SY}$.
Here  $Q$ is a quotient of $E$ and so $Q \in \tT^{\SY}_2$.
From the Harder-Narasimhan filtration property
$\Ga^0_{\bB} \HGa^2_{\bB}(E)$ fits into the short exact sequence
$$
0 \to T \to \Ga^0_{\bB} \HGa^2_{\bB}(E) \to F \to 0
$$
in $\bB^{\SY}$ for some $T \in \tT^{\SY}_2$ and $F \in \fF^{\SY}_2$.
Now apply the  Fourier-Mukai  transform $\HGa$ and consider the long exact sequence of $\bB^{\SX}$-cohomologies.
Then we have $\HGa^0_{\bB}(T) =0$, $\HGa^1_{\bB}(T) \cong \HGa^0_{\bB}(F)$.
By  Proposition \ref{prop:FMT-B-cat-bounds-Imginary-Z}--(2)(ii), $\Imm Z_{\Omega}(\HGa^0_{\bB}(F)) \le 0$ 
and by  Proposition \ref{prop:FMT-B-cat-bounds-Imginary-Z}--(1)(ii),  $\Imm Z_{\Omega}(\HGa^2_{\bB}(T)) \ge 0$.
So $\Imm Z_{\Omega}(\HGa(T)) \ge 0$, and by  Proposition~\ref{prop:imgainary-part-central-charge}--(2), $\Imm Z_{\Omega'}(T) \le 0$.
Since $T \in \tT^{\SY}_2$, we have $\Imm Z_{\Omega'}(T) = 0$
 and $v_1^{D_{\SY} - \frac{1}{2 \lambda}\ly, \ly}(T)=0$.
From Lemma~\ref{prop:first-tilt-behaves-like-sheaves-surfaces}, $T \cong T_0$ for some $T_0 \in \Coh_0(Y)$.
But $\Coh_0(Y) \subset V^{\HGa}_{\bB^{\SX}}(0)$.
Hence, $T =0$ and so $Q =0$.
Then $\Ga^0_{\bB} \HGa^2_{\bB}(E) \cong \Ga^2_{\bB} \HGa^1_{\bB}(E) $ 
and so we have $\HGa^2_{\bB}(E) \cong \HGa^2_{\bB} \Ga^2_{\bB} \HGa^1_{\bB}(E) = 0$ as required.
\medskip

\noindent Proofs of (2),(3) and (4) are similar to that of (1).
\end{proof}
\begin{prop}
\label{prop:tilt-bounds-FMT-B-cat-0-2}
We have the following:
\begin{enumerate}[label=(\arabic*)]
\item if $E \in \bB^{\SY}$ then (i) $\HGa^2_{\bB}(E) \in \tT^{\SX}_2 $, and (ii) $\HGa^0_{\bB}(E) \in \fF^{\SX}_2$;
\item if $E \in \bB^{\SX}$ then (i) $\Ga^2_{\bB}(E) \in \tT^{\SY}_2$, and (ii) $\Ga^0_{\bB}(E) \in  \fF^{\SY}_2$.
\end{enumerate}
\end{prop}

\begin{proof}
(1) \
Let  $E \in \bB^{\SY}$.
By the definition of torsion theory  $\HGa^2_{\bB}(E)$ fits into the  short exact sequence
$$
0 \to T \to \HGa^2_{\bB}(E)  \to F \to 0
$$
in $\bB^{\SX}$ for some $T \in \tT^{\SX}_2$ and $F \in \fF^{\SX}_2$.
Now apply the  Fourier-Mukai  transform $\Ga$ and consider the long exact sequence of $\bB^{\SY}$-cohomologies.
By Lemma~\ref{prop:B-cohomo-vanishing-FMT-0-2}, $\Ga^{i}_{\bB}(F)=0$ for all $i$, and so $F=0$ as required. 

Similarly one can prove $\HGa^0_{\bB}(E) \in \fF^{\SX}_2$. \\

\noindent (2) \ Similar to the proofs in (1).
\end{proof}
\begin{prop}
\label{prop:tilt-bounds-FMT-B-1}
We have the following:
\begin{enumerate}[label=(\arabic*)]
\item if $E \in \fF^{\SY}_2 $ then $\HGa^1_{\bB}(E) \in \fF^{\SX}_2$,
\item if $E \in \tT^{\SY}_2$ then $\HGa^1_{\bB}(E) \in \tT^{\SX}_2$,
\item if $E \in \fF^{\SX}_2$ then $\Ga^1_{\bB}(E) \in \fF^{\SY}_2$, and
\item if $E \in \tT^{\SX}_2$ then $\Ga^1_{\bB}(E) \in  \tT^{\SY}_2$.
\end{enumerate}
\end{prop}
\begin{proof}
Let us prove (1). Let $E \in \fF^{\SY}_2$.
By the definition of torsion theory $\HGa^1_{\bB}(E)$ fits into  the short exact sequence
$$
0 \to T \to \HGa^1_{\bB}(E) \to F \to 0
$$
in  $\bB^{\SX}$ for some $T \in \tT^{\SX}_2$ and $F \in \fF^{\SX}_2$.
Now we need to show $T=0$.
Apply the  Fourier-Mukai  transform $\Ga$ and consider the long exact sequence of $\bB^{\SY}$-cohomologies.
We get $\Ga^1_{\bB}(T) \hookrightarrow \Ga^1_{\bB} \HGa^1_{\bB}(E)$ and $T \in V^{\Ga}_{\bB^{\SY}}(1)$.
Also by the convergence of the Spectral Sequence~\ref{Spec-Seq-B} for $E \in \fF^{\SY}_2$,
$\Ga^1_{\bB} \HGa^1_{\bB}(E)$ is a subobject of $E$.
Hence, $\Ga^1_{\bB}(T)  \in \fF^{\SY}_2$ implies $\Imm Z_{\Omega'}(\Ga^1_{\bB}(T)) \le 0$.
On the other hand, by Proposition \ref{prop:imgainary-part-central-charge},
$ \Imm Z_{\Omega'}(\Ga^1_{\bB}(T)) =  \frac{3!}{r \lambda^3 \lx^3}\Imm Z_{\Omega}(T) \ge 0$ as $T \in \tT^{\SX}_2$.
Hence, $\Imm Z_{\Omega}(T) = 0$ and  $T \in \tT^{\SX}_2$ implies $v_1^{-D_{\SX}+ \frac{\lambda}{2}\lx, \lx} (T) =0$.
So by Proposition \ref{prop:first-tilt-behaves-like-sheaves-surfaces}, $T \cong T_0$ for some $T_0 \in \Coh_0(X)$. Since any object from $\Coh_0(X)$ 
belongs to $V^{\Ga}_{\bB^{\SY}}(0)$, $\Ga^1_{\bB}(T)= 0$.
So $T=0$ as required.
\medskip

\noindent Proofs of (2), (3) and (4) are similar to that of (1).

\end{proof}

%

From  Lemma \ref{prop:B-cohomo-vanishing-FMT-0-2},   Propositions \ref{prop:tilt-bounds-FMT-B-cat-0-2} and \ref{prop:tilt-bounds-FMT-B-1} we have 
\begin{equation*}
\left.\begin{aligned}
 & \Gamma (\tT^{\SX}_2) \subset \langle \fF^{\SY}_2 , \tT^{\SY}_2[-1]  \rangle = \aA^{\SY}[-1]\\        
 &  \Gamma(\fF^{\SX}_2[1]) \subset \langle \fF^{\SY}_2 , \tT^{\SY}_2[-1]  \rangle = \aA^{\SY}[-1]
       \end{aligned}
  \ \right\},
\end{equation*}
and 
\begin{equation*}
\left.\begin{aligned}
 & \HGa (\tT^{\SY}_2) \subset \langle \fF^{\SX}_2 , \tT^{\SX}_2[-1]  \rangle = \aA^{\SX}[-1]\\        
 &  \HGa(\fF^{\SY}_2[1]) \subset \langle \fF^{\SX}_2 , \tT^{\SX}_2[-1]  \rangle = \aA^{\SX}[-1]
       \end{aligned}
  \ \right\}.
\end{equation*}
Since $\aA^{\SX} = \langle \fF^{\SX}[1] , \tT^{\SX} \rangle$ and  $\aA^{\SY} = \langle \fF^{\SY}[1] , \tT^{\SY} \rangle$,
we have $\Ga [1] (\aA^{\SX} ) \subset \aA^{\SY}$ and $\HGa [1] (\aA^{\SY} ) \subset \aA^{\SX}$. Since we have the
isomorphisms $\HGa[1] \circ \Ga[1] \cong  \id_{D^b(X)}$ and
$\Ga[1] \circ \HGa[1] \cong \id_{D^b(Y)}$, we deduce the following.

\begin{thm}
\label{prop:equivalence-stab-hearts-threefolds}
The Fourier-Mukai transforms $\Ga, \HGa$ give the   equivalences of the double tilted hearts:
$$
\Ga[1]\left(\aA^{\SX} \right) \cong \aA^{\SY}, \ 
\text{ and } \ \HGa[1]\left(\aA^{\SY}\right) \cong \aA^{\SX}.
$$
\end{thm}

\providecommand{\bysame}{\leavevmode\hbox to3em{\hrulefill}\thinspace}
\providecommand{\MR}{\relax\ifhmode\unskip\space\fi MR }
\providecommand{\MRhref}[2]{%
  \href{http://www.ams.org/mathscinet-getitem?mr=#1}{#2}
}
\providecommand{\href}[2]{#2}

\end{document}